\documentclass[english,ngerman,paper=a5,headsepline=true,9pt,DIV=12,BCOR=0.7cm]{scrbook}
\usepackage{typearea}           
\usepackage[utf8]{inputenc}     

\usepackage{lmodern}            

\usepackage[sc,osf]{mathpazo}   

\usepackage[T1]{fontenc}        

\usepackage{babel}              

\usepackage{amsmath,amssymb}    
\usepackage{color}
\usepackage{tikz}               
\usepackage{microtype} 

\usepackage{booktabs} 
\usepackage{varioref} 
\usepackage{icomma}   

\usepackage{marginnote}         
\usepackage{cite}               
\usepackage{listings}           
\usepackage{caption}            

\colorlet{lcolor}{blue!40!black}
\colorlet{ucolor}{magenta!40!black}
\colorlet{ccolor}{green!40!black}

\usepackage[colorlinks=true,%
            linkcolor=lcolor,%
            urlcolor=ucolor,%
            citecolor=ccolor]{hyperref} 
            
\usepackage{graphicx}
\usepackage{subfig}
\usepackage{epstopdf}

\usepackage{fancyhdr}
\usepackage{mathtools}
\usepackage{amsmath,amsthm,amssymb,mathrsfs,dsfont}
\usepackage{siunitx}
\usepackage[europeanresistors,americaninductors]{circuitikz}

\newtheorem{theorem}{Theorem}[section]
\newtheorem{lemma}[theorem]{Lemma}
\newtheorem{definition}[theorem]{Definition}
\newtheorem{corollary}[theorem]{Corollary}
\newtheorem{proposition}[theorem]{Proposition}

\newtheorem{remark}[]{Remark}
\newtheorem{assumption}[theorem]{Assumption}

\newcommand*\mbb[0]{\mathbb}                                                  		
\newcommand{\subscr}[2]{{#1}_{\textup{#2}}}

\newcommand\oprocendsymbol{\hbox{$\square$}}
\newcommand\oprocend{\relax\ifmmode\else\unskip\hfill\fi\oprocendsymbol}

\newcommand{\real}[0]{\mathbb R}
\newcommand*\mycircle[0]{\mathbb S}
\newcommand{\complex}[0]{\mathbb C}
\newcommand{\norm}[1]{\left\lVert#1\right\rVert}

\newcommand{\inv}{^{\raisebox{.2ex}{$\scriptscriptstyle-1$}}}

\usepackage{tikz}
\usetikzlibrary{calc}
\usetikzlibrary{shapes,arrows}
\ctikzset{bipoles/thickness=1}
\ctikzset{bipoles/length=0.8cm}
\ctikzset{bipoles/diode/height=.375}
\ctikzset{bipoles/diode/width=.3}
\ctikzset{tripoles/thyristor/height=.8}
\ctikzset{tripoles/thyristor/width=1}
\ctikzset{bipoles/vsourceam/height/.initial=.7}
\ctikzset{bipoles/vsourceam/width/.initial=.7}
\tikzstyle{every node}=[font=\small]
\tikzstyle{every path}=[line width=0.8pt,line cap=round,line join=round]

\tikzstyle{block} = [draw, fill=blue!10, rectangle, minimum height=5em, minimum width=10em]
\tikzstyle{smallblock} = [draw, fill=blue!10, rectangle, minimum height=2em, minimum width=2em]
\tikzstyle{sum} = [draw, fill=blue!10, circle, node distance=1cm]
\tikzstyle{input} = [coordinate]
\tikzstyle{output} = [coordinate]
\tikzstyle{pinstyle} = [pin edge={to-,thin,black}]

\setcapindent{0em}
\addtokomafont{captionlabel}{\bfseries}
\titlehead{Master Thesis 35 \hfill \raisebox{-4mm}{\includegraphics[height=12mm]{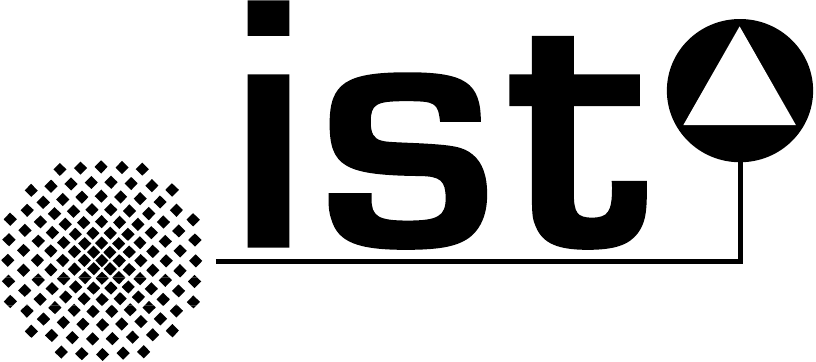}}}

\title{Grid-friendly Matching Control of Synchronous Machines by DC/AC Converters in Bulk Power Networks} 

\newcommand{\authorstring}{Taouba Jouini} 
\author{\authorstring}

\date{June 7, 2016} 
\publishers{
  \begin{tabular}{rl}
    Prüfer: & \emph{Prof.\,Dr. Frank Allgöwer}\\
    Betreuer: & \emph{Msc. Catalin Arghir}\\
    &\emph{Prof.\,Dr. Florian Dörfler}\\
  \end{tabular}\\[1cm]
  Institut für Systemtheorie und Regelungstechnik\\
  Universität Stuttgart\\Prof. Dr.-Ing. Frank Allg\"ower} 

\begin{document}
\selectlanguage{english}        
\KOMAoptions{twoside=false} 
\recalctypearea
\maketitle
\KOMAoptions{twoside=true} 
\recalctypearea

{\em
	An islanded inverter-based microgrid is a collection of heterogeneous DC energy resources, e.g., photovoltaic arrays, fuel cells, and energy-storage devices, interfaced to an AC distribution network and operated independently from the bulk power system. Energy conversion is typically managed by power-electronics in voltage source inverters.
	Drawing from the control of synchronous machines in bulk power systems, different control schemes have been recently adopted in order to achieve a stable network operation.
	The vast majority of academic and industrial efforts opt for these strategies during real-time operation.
	
	Starting with a dynamical averaged DC/AC converter model, we review different controllers by presenting its main scope analytically and through simulations. Next, we explore a new alternative of controlling DC/AC converters in bulk power systems by matching traditional synchronous machines with emphasis on the role that DC-circuit can play in control architecture, usually neglected in conventional strategies. Compared to standard emulation methods, our controller relies solely on readily available DC-side measurements and takes into account the natural DC and AC storage elements. As a result, our  controller is generally faster and less vulnerable to delays and measurement inaccuracies. We additionally provide insightful interpretations of the suggested control, various plug-and-play properties of the closed loop, such as steady-state power flow analysis, passivity with respect to the DC and AC ports, stability proof as well as high-level control architectures contributing to enhancing the controller performance and attaining further control goals, which we illustrate in both analysis and simulation.}

\newpage 
{\em Ein Inselnetz fähiges Microgrid ist eine Sammlung von heterogenen DC- Energiequellen, beispielsweise Photovoltaik Anlagen, Brennstoffzellen und Energiespeicher, die mit einem AC Verteilungsnetz angeschlossen und unabhängig vom Hauptstromsystem betrieben sind. Energieumwandlung wird in der Regel von Leistungselektronik in Spannungszwischenkreisumrichter verwaltet. Aufbauend auf die Konzepte der Regelung von Synchronmaschinen in Großstromanlagen sind letztlich unterschiedliche Schemata angenommen um einen stabilen Netzbetrieb zu erreichen. Die überwiegende Mehrheit der akademischen und industriellen Anstrengungen entscheiden sich für diese Strategien während der Echtzeit-Regelung von DC/AC Wandlern.

Beginnend mit einem dynamischen gemittelten DC/AC- Wandler Modell, führen wir die Grundgedanken verschiedener Regelungsstrategien ein und prüfen die Umsetzbarkeit ihrer Regelungsvorschriften analytisch und durch Simulationen. Als Nächstes erkunden wir eine neue Alternative zur Regelung von DC/AC- Wandler in Stromanlagen durch Anpassung an einem entsprechenden hochdimensionalen Synchronmaschinen Modell mit dem Schwerpunkt aufgesetzt auf die Rolle, die die Gleichstromschaltung in Regelungsarchitektur spielen kann, was in der Regel in konventionellen Strategien vernachlässigt wird. Im Vergleich zu den standardisierten Verfahren stützt sich unser Regler ausschließlich auf leicht verfügbaren Messungen von der DC-Seite und berücksichtigt dabei die natürliche Gleich- und Wechselspeicherelemente vorhanden in einem DC/AC Wandler. Als Ergebnis ist unser Regler im allgemeinen schneller und weniger anfällig zu Verzögerungen und Messungenauigkeiten. Wir bieten zusätzlich interessante Interpretationen der vorgeschlagenen Regelung, verschiedene Plug-und-Play Eigenschaften im geschlossenen Kreis durch stationäre Leistungsflussanalyse, Passivität gegenüber DC- und AC inputs, Stabilitätsbeweis sowie höhere Stufen für eine erweiterte Regelung, um die Leistung des DC/AC Wandlers zu verbessern und andere Ziele zu erreichen, die wir sowohl in Analyse als auch in Simulation darstellen.}

\tableofcontents

\chapter{Introduction}
\label{sec: intro}
\begin{figure}[h!]
	\centering
	\includegraphics[scale=0.2]{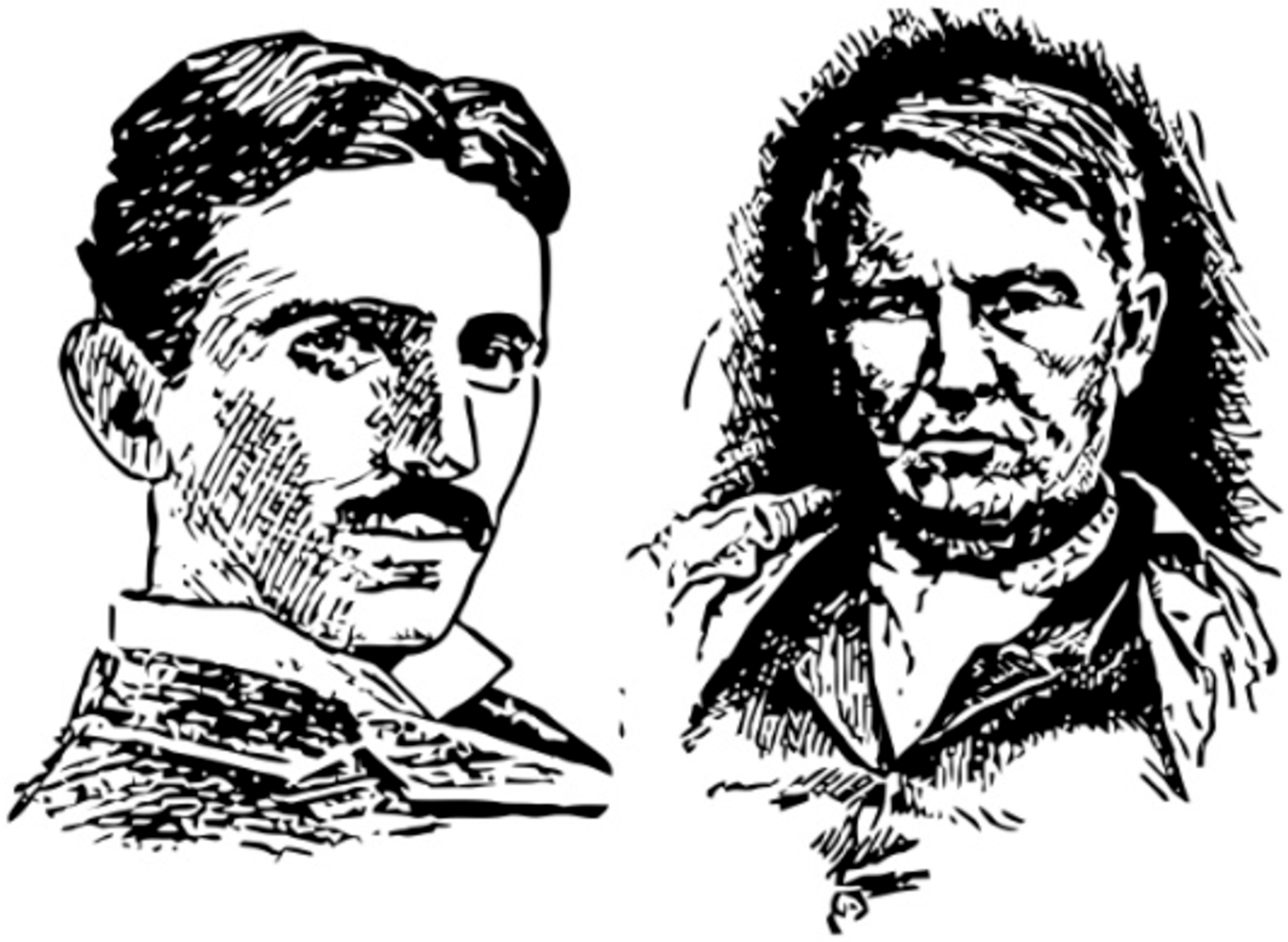}
\end{figure}
The history of electric power systems began with the inventors Nicola Tesla $(1856-1943)$ and Thomas Edison $(1847-1931)$ who contributed to the development of electric power systems that are widely used today. Tesla and Edison set the basis for most of AC and DC machinery. 

In conventional bulk power plants, synchronous machines (SM) dominate. In fact, rotating generators are electro-mechanical converters. The interaction between the mechanical interface and the electrical port is described by the following equation

\begin{equation*}
{M\, \frac{d}{dt}\,  \omega(t) \;=\; \subscr{P}{generation}(t) - \subscr{P}{demand}(t)}
\,,
\end{equation*}
where M=$\omega J$, $\omega$ is AC frequency and $J$ is the moment of inertia.
This implies that a change in the kinetic energy is the instantaneous power balance between the generation and the demand.

\begin{figure}[h!]
	\centering
	\includegraphics[width=0.15\textwidth]{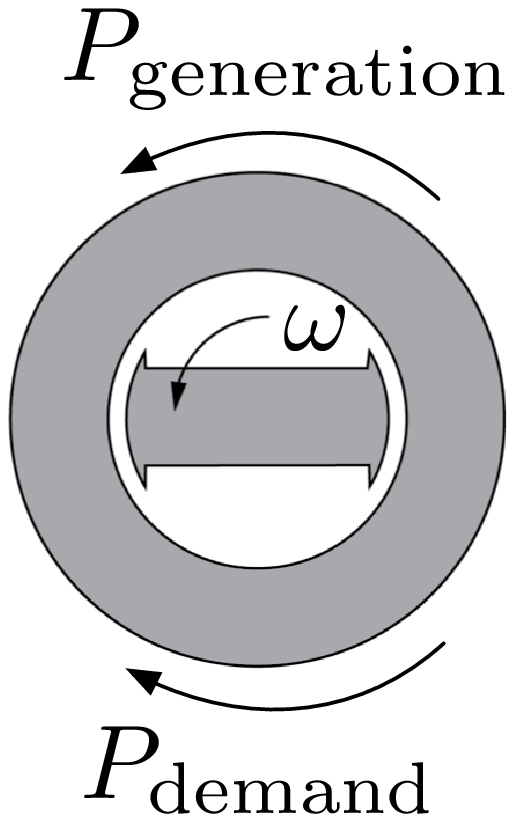}
\end{figure}

In classical power system, SMs offer favorable features to support system operation. Indeed SMs contribute to the system damping through their inertia and participate in the primary frequency regulation as the frequency response, depicted in Figure \ref{fig: freq-resp} found in most of power systems literature. The rotor mass $M$ stabilizes the system by providing damping contributing to the enhancement of system performance. As a matter of fact, synchronous machine dynamics describe a natural differential controller for the frequency $\omega$ with the D-gain $M$. 

\begin{figure}
\centering
\includegraphics[scale = 1]{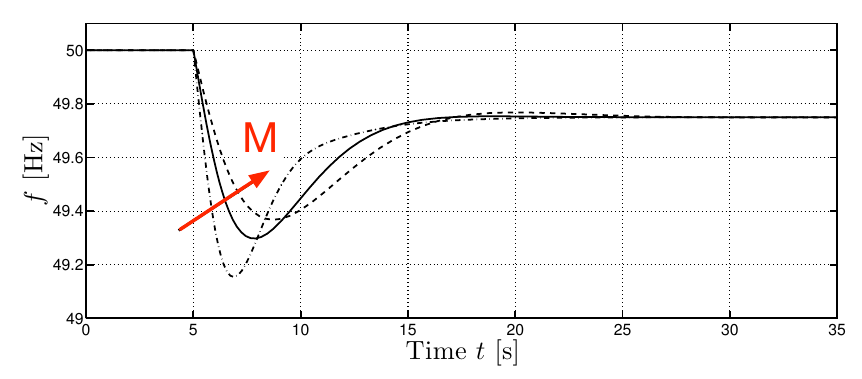}	
\caption{Frequency response of a synchronous machine after a disturbance at $t=5$s}
\label{fig: freq-resp}
\end{figure}

The advantageous capabilities are not inherently offered by power electronics interfaces commonly adopted for the integration of renewable energy. In other words, the increase of small-scale dispersed power generation is likely to impact the structure and operation of power systems. As a matter of fact, power systems are increasingly based on non-rotational generation with power electronic converters interfacing renewable energy storage and batteries with the AC power grid. A major challenge in such low-inertia networks is the replacement of stabilizing rotational inertia of SMs and their ancillary services through control of converters \cite{PT-DVH:16}. The so-called {\em grid-forming} converter control strategies range from droop methods, emulation of synchronous machines to more general limit-cycle oscillator based strategies. This name is attributed to the power electronic converters since they have the capability to connect to a DC and AC grid sides.
These control schemes are often brand-marked as {\em grid-friendly} since they are based on fully decentralized control, naturally backward compatible with SM, and ultimately increase system inertia.

\vspace{1em}
For its correct operation, electric power systems must satisfy a large set of different regulation objectives, that are typically associated to multiple time-scale behaviors of the system. A hierarchal architecture usually serves as nested control loop in order to operate at different time scales as shown in Figure \ref{fig: Swiss-osc}, where the curves describe an overlap of the frequency behavior in Mettlen, Switzerland and that of Athens according to control intervention after disturbances. 
{\em Primary control} is of fundamental importance, when perturbation occurs. Its main objective consists in adjusting promptly the reference to be provided to the inner-loop, consisting in the innermost cascaded control interfaced with the DC/AC converter model and aiming to track a given reference by the primary control. Primary control maintains stability according to a pre-specified power distribution called also {\em power sharing}. It operates in the time scale of $ms$. {\em Secondary control} calculates mainly current and voltage references or recalculates power references derived from higher-level control called {\em tertiary control}. The latter is the outermost level of control. It is usually based on optimization algorithms of power flow which generate nominal or perturbed reference associated with the appropriate operating conditions. The time required for intervention of tertiary control ranges from $20$ min to $1$ h. 

\begin{figure}[h!]
	\centering
	\includegraphics[width=0.8\textwidth]{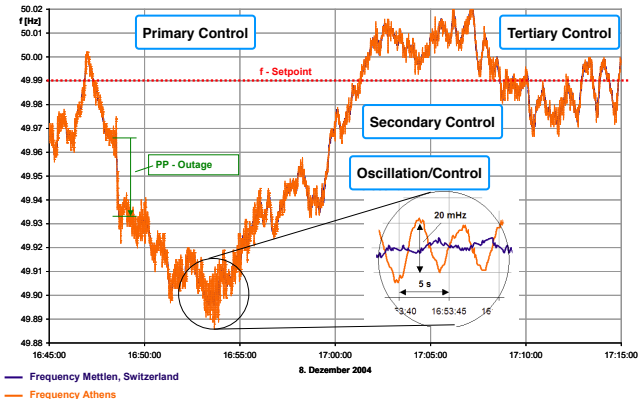}
	\caption{Frequency response of synchronous machine following up a system disturbance according to a predefined hierarchical control architecture.}
	\label{fig: Swiss-osc}
\end{figure}

Bearing in mind all the difficulties imposed by conventional control architecture and the challenges arising after retirement of synchronous machines, we propose {\em in this thesis} a novel strategy of controlling converters compensating for the absence of SMs and assuring a stable network operation. Starting from a model showing attention to the DC-side available in a converter and a high-order model of SM taking into account both rotor and stator dynamics, we find a controller that matches the two models by comparing rotor dynamics to that of a DC circuit. Compared to virtual synchronous machine and emulation algorithms, we use matching rather than virtual emulation due to the involvement of the physical storage rather than the virtual storage element commonly used to replace the mechanical inertia of a SM and rooted in a structural similarity, i.e an equivalence between the two models up to a defined order (in per unit). Towards a generalization for a network consisting of multiple identical converters connected to a grid, our control scheme fulfills certain primary control goals. These promising properties range from inducing droop behavior serving as a typical control requirement and preserving passivity in closed-loop to assuring stability in closed loop fashion. High-order loops can be constructed and added to the matching control by taking into account different objectives in order to keep currents and voltages within limits and track a desired frequency. Our predictions for the network case are validated through simulations which establishes the basis towards assuring power sharing between multiple converters.

The remainder of this work is organized as follows,
Section \ref{sec: modeling of inverter} presents an averaged three-phase DC/AC converter model. Based on it, we review  in Section \ref{sec: converter control schemes} conventional control schemes, described by the innermost control loops called {\em inner-loop} control as well as {\em outer-loops} ranging from droop, oscillation-based to virtual synchronous machine methods. Section \ref{sec: matching control of SM} presents the main analogies between SM and DC/AC converter model and proposes an intuitive control strategy to match them based on a structural equivalence between the two models. Section \ref{sec: properties} deals with the properties of the closed-loop DC/AC converter and highlights its plug-and-play characteristics at steady state including power flow analysis, synchronization for AC signals and passivity with respect to the DC and AC grid inputs.
In Section \ref{sec: stability anal}, a proof of the global asymptotic stability of the desired equilibrium set is derived using Lyapunov theory. We provide further extensions of this novel control strategy in Section \ref{sec: high-level-ctrl}, where one may notice the flexibility of our control design, since it can be extended through its gains by outer loops measuring DC and AC sides and achieving further control goals. Finally, we provide simulations of the multiple DC/AC converters case once connected to the grid in Section \ref{sec: network-case}. Our observations open up horizons of further investigating the proposed approach as a decentralized control in a grid network by enhancing its performance in terms of frequency regulation and power sharing from a networked viewpoint.

\chapter{Modeling of a three-phase inverter}
\label{sec: modeling of inverter}

{\em In this chapter}, we introduce the usual working frame of the three-phase alternating current (AC) signals in electric power as well as the background necessary to assimilate the transformations performed from one domain to the other and required for different analysis purposes in the remainder of this work. Next, we introduce the averaged model of a DC/AC converter serving as basis for our future investigations.

\section{Background }
\subsection{Three-phase signals}
A $3$-phase AC signal represented in (abc) frame is a vector in $z_{abc}\in\real^3$ defined by its amplitude $\hat z$ and angles at each phase $(a)$, $(b)$ or $(c)$ with the phase angle difference of ${2\pi}/{3}$ defined as follows
\begin{equation}
z_{abc}=\hat z \begin{bmatrix}
\sin(\theta) \\ \sin(\theta-\frac{2\pi}{3}) \\ \sin(\theta+\frac{2\pi}{3})
\end{bmatrix},\, \hat z>0
\label{eq: z_abc}
\end{equation}

When the angle difference is exactly of $2\pi/3$, the signal is considered to be {\em balanced} such that it holds
\begin{equation}
z_a+z_b+z_c=1^{\top} z_{abc}= 0
\label{eq: bal-sig}
\end{equation}
as shown in the Figure \ref{fig: three-phase}.

\begin{figure}[h!]
	\centering
	\includegraphics[scale=0.5]{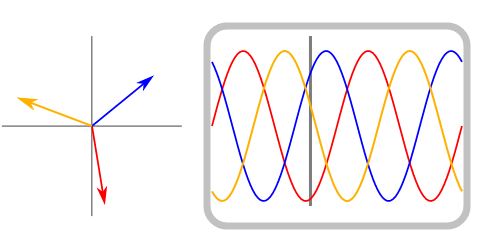}
	\caption{Representation of balanced three phase-signal in the phasor plane and in time-domain}
	\label{fig: three-phase}
\end{figure}
If the condition \eqref{eq: bal-sig} is violated, the three-phase AC signal is considered to be {\em unbalanced}

\subsection{Clarke transformation to $(\alpha\beta)$ frame}
Assuming a balanced three-phase signal $z_{abc}\in\real^3$, i.e $1^{\top}z_{abc}=0$, we can represent AC signals in $(\alpha\beta)$- frame, where we can omit one component and take $z_{\alpha\beta}\in\real^2$. 
We can perform a change of coordinates from $(abc)$ domain into $(\alpha\beta)$ domain using $T_{\alpha\beta}: \mbb R^3 \to \mbb R^3$ defined as follows: 
\begin{equation}
T_{\alpha\beta}=\sqrt{\frac{2}{3}}\begin{bmatrix}
1 & \frac{-1}{2} & \frac{-1}{2}\\
0 & \frac{\sqrt{3}}{2} & -\frac{\sqrt{3}}{2}
\\
\frac{1}{\sqrt{2}} & \frac{1}{\sqrt{2}}    & \frac{1}{\sqrt{2}}
\end{bmatrix}
\label{eq: abc to alphabeta}
\end{equation}
with the inverse $T_{\alpha\beta}^{-1}= T_{\alpha\beta}^{\top}$.
An advantage of the choice of the matrix $T_{\alpha\beta}$ is the easy inverse and invariance power calculation, where the power in $(\alpha\beta)$- frame is the same as in $(abc)$- frame, i.e power calculation does not require a scaling factor.
After transformation into $(\alpha\beta)$- frame 
\begin{subequations}
	\begin{align}
	z_{\alpha\beta} &= T_{\alpha\beta}z_{abc}=\hat z\sqrt{\frac{3}{2}}\begin{bmatrix}
	-\sin(\theta) \\ \cos((\theta) \\0
	\end{bmatrix}
	\\
	\,
	\frac{dz_{\alpha\beta}}{d\theta} &=T_{\alpha\beta}\frac{dz_{abc}}{d\theta}=\hat z\sqrt{\frac{3}{2}}
	\begin{bmatrix}
	\cos(\theta) \\ \sin(\theta) \\ 0
	\end{bmatrix}=J_3 z_{\alpha\beta}
	\,,
	\end{align}
\end{subequations}
with $J_3=\begin{bmatrix}
0 & -1 & 0 \\ 1 & 0 & 0 \\ 0 & 0 & 0 
\end{bmatrix}$.

\begin{figure}
\centering
\includegraphics[scale=0.5]{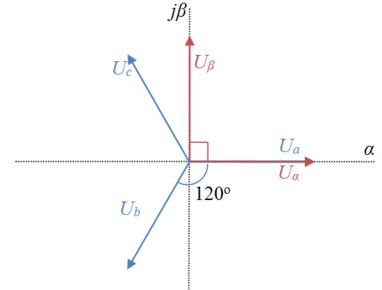}
\caption{Three-phase and two-phase stationary reference frames\cite{Clarke-trafo}}
\label{eq:clarke}
\end{figure}

This coordinate change can be interpreted as a projection from a sphere in $\mbb S^3$ of radius $\hat z>0$ into a circle of radius $\hat z \sqrt{{3}/{2}}>0$ in $\real^2$, where the derivative of a vector with respect to the angle $\theta$, is the same as a rotation with angle ${\pi}/{2}$ of the same vector in $(\alpha \beta)$-frame.

In fact, $(\alpha\beta)$- transformation can be understood as that of a natural polar coordinate representation of the three-phase AC signals $z_{abc}\in\real^3$, since it allows to represent a vector by its amplitude $\hat z\sqrt{{3}/{2}}>0$ and its angle $\theta\in\mycircle$ as $z_{\alpha\beta}\in\real^2$ such that 

\begin{equation*}
z_{\alpha\beta}= \hat z\sqrt{\frac{3}{2}}\begin{bmatrix}
-\sin(\theta)\\ \cos(\theta)
\end{bmatrix}=T_{\alpha\beta} z_{abc} 
\,.
\end{equation*}

\begin{remark}[Complex form]
	Any vector in $(\alpha\beta)$ coordinates can be represented by a complex number $\vec{z}_{\alpha\beta}\in\complex$
	\begin{equation}
	\vec{z}_{\alpha\beta}=z_\alpha+j z_\beta
	\,,
	\end{equation}
	where $j=\sqrt{-1}$.
\end{remark}

\subsection{Park transformation to rotating $dq0$ domain}
Steady state currents and voltages for the $(abc)$ phases of a single generator are sinusoidal waveforms.
There exists a transformation of the $(abc)$ phases to the $(dq0)$ frame using $T_{dq0}: \mbb R^3 \to \mbb R^3$ defined as follows: 
\begin{equation}
	T_{dq0}(\gamma) =\sqrt{\frac{2}{3}}\begin{bmatrix}
	\cos(\gamma) & \cos(\gamma-\frac{2\pi}{3}) & \cos(\gamma+\frac{2\pi}{3})\\
	\sin(\gamma) & \sin(\gamma-\frac{2\pi}{3}) & \sin(\gamma+\frac{2\pi}{3})\\
	\frac{1}{\sqrt{2}} & \frac{1}{\sqrt{2}}    & \frac{1}{\sqrt{2}}
	\end{bmatrix}= R_{-\gamma}\,T_{\alpha\beta} 
	\,,
	\label{eq: trafo to dq}
\end{equation}
with the inverse $T_{dq0}^{-1}= T_{dq0}^{\top}$ and $T_{\alpha\beta}$ corresponding to the transformation matrix to the {\em power invariant formulation} from $(abc)$ to $(\alpha\beta)$- frame defined in \eqref{eq: abc to alphabeta} and $R_{\gamma}$ is the rotation matrix defined as follows
	
\begin{equation*}
	R_{-\gamma}=\begin{bmatrix}
	\cos(\gamma) &\sin(\gamma)& 0\\ -\sin(\gamma) & \cos(\gamma) & 0 \\ 0 & 0 & 1
	\end{bmatrix}
	\,.
\end{equation*}
	
We can therefore map $(abc)$- currents and voltages to $(dq0)$ domain using \eqref{eq: trafo to dq}
\begin{equation}
	z_{dq0}=(z_d, z_q, z_0)= T_{dq0}\,(z_a, z_b, z_c)
	\,.
\end{equation}

It is noteworthy that a transformation from $(dq0)$ into $(\alpha\beta)$- frame, is a rotation of angle $\gamma$, defined as follows 
\begin{equation*}
	z_{\alpha\beta}=R_{\gamma} z_{dq0}
	\,,
\end{equation*}
with $R_{\gamma}=R_{-\gamma}^{\top}$.
\\
We now consider the vector $z_{abc}$ as defined previously in \eqref{eq: z_abc}
\begin{equation*}
	z_{dq0}=T_{dq0}(\gamma)z_{abc}(\theta)=\hat z\sqrt{\frac{3}{2}}\begin{bmatrix}
	\sin(\theta-\gamma) \\ \cos(\theta-\gamma) \\ 0
	\end{bmatrix}
	\,.
\end{equation*}

A transformation of a balanced three-phase vector in $(abc)$ into $(dq0)$ frame defines a projection to a vector rotating with angle $\gamma$ living on a circle with radius $\hat z\sqrt{{3}/{2}}$.
If we additionally choose the angle of the transformation matrix $T_{dq0}$ to be the angle of the balanced three-phase vector $\gamma=\theta$, it yields 
\begin{equation*}
	z_{dq0}=T_{dq0}(\theta)z_{abc}(\theta)=\hat z\sqrt{\frac{3}{2}}
	\begin{bmatrix}
	0 \\ 1 \\ 0
	\end{bmatrix}
	\,,
\end{equation*}
	
which defines a projection from a point on a sphere in $\mbb R^3$ into a point in $\mbb R$ as shown in Figure \ref{fig:clarke}.

\begin{figure}
	\centering
	\includegraphics[scale=0.5]{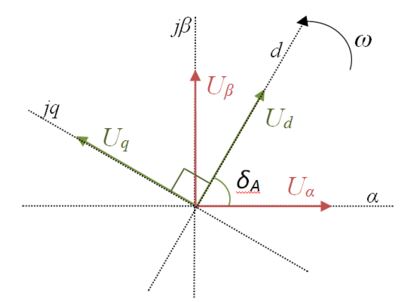}
	\caption{Three-phase and two-phase rotating reference frame $dq0$ with frequency $\omega$\cite{Clarke-trafo}}
	\label{fig:clarke}
\end{figure}

Assuming a balanced system, we ignore the zero component and take $z_{dq}\in\real^2$.

\begin{definition}[Equilibria in the $(dq0)$- frame]
We assume that at steady state AC signals exhibit harmonics synchronous at $\omega_s$ in the $(dq0)$- frame with $\dot \theta=\dot\gamma=\omega_s$ and we consider its dynamics in $(dq0)$- frame, where we use the fact that $J_2 T_{dq}(\gamma)=T_{dq}(\gamma)J_2$ to get
\begin{subequations}
	\begin{align*}
		\dot z_{dq}&=\dot T_{dq}(\gamma) z_{\alpha\beta}+T_{dq}(\gamma) \dot z_{\alpha\beta}
		\\
		&= -J_2 \omega_s T_{dq}(\gamma) z_{\alpha\beta}+ J_2 \omega_s T_{dq}(\gamma) z_{\alpha\beta}
		\\
		&=-\omega_s J_2 z_{dq}+ \omega_s J_2 z_{dq}
		\\
		&=0
		\,,
	\end{align*}
\end{subequations}
	
Therefore, an equilibrium in $(dq)$- frame describes a point in $\real^2$.
\label{def:eq-dq}
\end{definition}

\subsection {Power calculation}

\begin{definition}[Instantaneous AC Power]
	We define the active and reactive power flowing out of an AC voltage node $u_{\alpha\beta}$ on an edge defined by an AC current $y_{\alpha\beta}$ as 
	\begin{subequations}
		\begin{align*}
		P_u &= u_{\alpha\beta}^{\top}\,y_{\alpha\beta}
		\\
		Q_u &=u_{\alpha\beta}^{\top} \begin{bmatrix}0 & -1 \\ 1 & 0\end{bmatrix} y_{\alpha\beta}
		\,.
		\end{align*}
	\end{subequations}
\end{definition}

This definition is in accordance with Akagi's instantaneous power theory \cite{HA-YK-KF-AN:83}.

Instantaneous active and reactive power can be rewritten $(\alpha \beta)$ frame as follows:
\begin{equation}
\begin{bmatrix}
P\\ Q
\end{bmatrix}=
\begin{bmatrix}
v_\alpha & v_\beta \\ v_\beta & -v_\alpha
\end{bmatrix}
\begin{bmatrix}
i_\alpha\\ i_\beta
\end{bmatrix}
\,,
\end{equation} 
where $v$ and $i$ are the voltage and current measured at the same node.

\subsubsection{Complex form}
Instantaneous power expressions can be derived from the complex power $\vec{S}$ defined as follows: 
	\[
	\vec{S}={\vec{v}_{\alpha\beta}}\vec{i}_{\alpha\beta}^*=\underbrace{(v_\alpha i_\alpha+v_\beta i_\beta)}_{P}+j\underbrace{(-v_\alpha i_\beta+v_\beta i_\alpha)}_{Q}
	\,,
	\]
where $j=\sqrt{-1}$ and * is the complex conjugate of a complex number.

\subsection{Modeling principles in power systems}
Any wire can be modeled as inductance. If it is lossy, a resistor is put in series connection to the inductance. Therefore, we model a wire as an RL circuit.
The earth acts as a ground and justifies the usage of the capacitors. Since the transmission lines are usually very long, there exists a capacitance with respect to the earth. They are modeled as shunt capacitors.
A shunt is in general an endpoint connected to an element which is extended in this context to the ground.

We design by {\em terminals}, the endpoint of an electrical component. In the case of a DC/AC converter model, its terminals correspond additionally to a connection point to the grid.

\subsubsection{Averaging of DC/AC converter}
An {\em averaged} value of an AC quantity $x(t)$ over a period $T_s$ is defined by
\begin{equation}
\bar x=\frac{1}{T_s}\int_{t-T_s}^{t} x(\tau) d\tau
\end{equation} 

We extensively use, in the derivation of the dynamic controlled models of the several converters, the fundamental Kirchoff’s current and Kirchoff’s voltage laws. The methodology for the derivation of the models is therefore, quite straightforward. We fix the position of the switch, or switches, and derive
the differential equations of the circuit model. We then combine the derived models into a single one parameterized by the switch position function whose value must coincide, for each possible case, with the numerical values of either “zero” or “one”. In other words, the numerical values ascribed to the switch
position function is the binary set $\{0, 1\}$. The obtained switched model is then interpreted as an {\em average} model by letting the switch position function take values on the closed interval of the real line $[0,1]$. This state averaging procedure has been extensively justified in the literature since the early days of power electronics and, therefore, we do not dwell into the theoretical justifications of such averaging procedure \cite{SRH-SOR:06}.

The request for a certain load voltage is translated into a corresponding requirement for the converter duty cycle. The duty cycle modulation is typically several orders of magnitude slower than the switching frequency. The net effect is attainment of an average voltage with relatively small ripples. See Figure \ref{fig:average} for a zoomed-in view of this dynamics.

\begin{figure}
	\centering
	\includegraphics[scale=0.4]{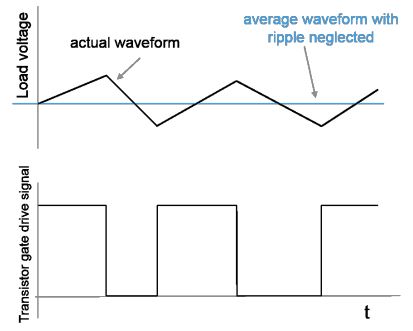}
	\caption{converter output voltage generation (source: internet)}
	\label{fig:average}
\end{figure}

\section{The averaged converter model}


\begin{assumption}[AC Signals]
All $3$-phase AC quantities, i.e. $z_{abc}\in\real^3$, are considered to be balanced in $(abc)$- frame. We denote by $z_{\alpha\beta} \in \real^2$ the representation in $(\alpha\beta)$- domain and omit the third component, denoted by $z_\gamma \in \real$.
\end{assumption}
In this section, we consider an averaged model of a $3$-phase converter composed of a linear DC circuit, a linear AC circuit and a nonlinear modulation block. The diagram in Figure \ref{fig: circuit diagram} depicts the model of a three-phase inverter described in $(abc)$- frame as follows \cite{AT-RI:09}  

\begin{figure}[h!]
	\centering
	\begin{center}
				\begin{circuitikz}[american voltages]
					\draw
					(0,0) to [short, *-] (3,0)
					(5,3) to [american current source, i>=$i_{load}$] (5,0) 
					
					(4,3) to [short, *-] (5,3)
					(4,0) to [short, *-] (5,0)
					(3,0) to (4,0)
					(3,3) to (4,3)
					(0.5,0) to [open, v^>=${v}_x$] (0.5,3) 
					(0,3) 
					to [short,*-, i=$i_{abc}$] (1,3) 
					to [R, l=$R$] (2,3) 
					to [L, l=$L$] (3.3,3) 
					to [short,*-, i_=$ $] (3.3,2) 
					to [C, l=$C$] (3.3,1) 
					to [short] (3.3,0)
					
					(2.2,3) to [open, v^<=$ v_{abc}$] (2.2,0); 
					
					\draw
					(-2,-0.5) to (-2,3.5)
					(-0.5,1) to [Tpnp,n=pnp] (-0.5,2)
					(-2,3.5) to (0,3.5)
					(0,3.5) to (0,-0.5)
					(-2,-0.5) to (0,-0.5);
					
					\draw
					(-2,0) to [short, *-] (-4,0)
					(-4.5,0) to [short, *-] (-4,0)
					(-5.5,0) to (-4.5,0)
					(-5.5,3) to (-4.5,3)
					(-5.5,0) to [american current source, i>=$i_{dc}$] (-5.5,3) 
					(-4.2,3) to [R,l=$G_{dc}$] (-4.2,0) 
					(-3,3) to [C, l=$C_{dc}$] (-3,0) 
					(-2,3) to [short,*-, i<=$i_{x}$] (-3,3)
					(-4.5,3) to [short, *-] (-3,3)
					(-5.3,3) to [open, v^<=${v}_{dc}$] (-5.3,0);
			\end{circuitikz}	
	\end{center}
	\caption{Circuit diagram of a three-phase DC/AC converter} 
	\label{fig: circuit diagram}
\end{figure}
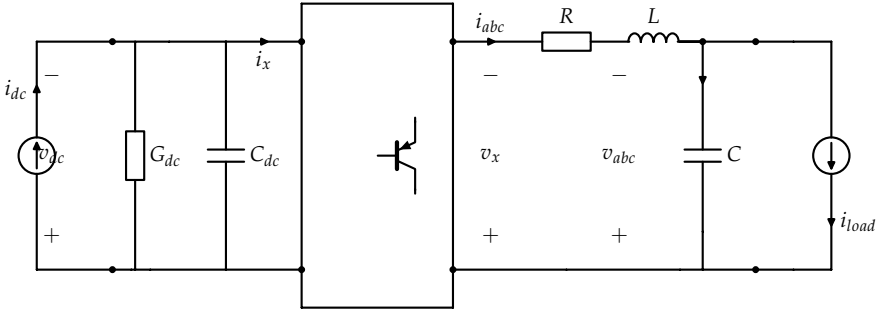

\begin{itemize}
	\item The DC circuit is represented by a constant current source $i_{dc}>0$, in parallel to a capacitance $C_{dc}>0$ and a conductance $G_{dc}>0$. Furthermore $i_{x}$ is the current at the output of the DC circuit and $v_{dc}$, the voltage across the DC capacitance.
	 We can write the DC model equation as
	\begin{equation}
	\label{eq: DC voltage}
	C_{dc} \dot v_{dc} = -G_{dc} v_{dc}+i_{dc}-i_{x}
	\,.
	\end{equation}

	\item The AC circuit contains at each phase an inductance $L>0$ in series with a resistance $R>0$ connected to a shunt capacitance $C>0$ to ground. Here $v_{abc}$ is the AC voltage across the capacitor, $i_{load}$ is the AC current drawn by the external load, $i_{abc}$ is the AC current in the inductance and $v_{x}$ the average AC voltage at the output of the modulation block. The dynamic of the linear AC circuit can be captured in $(abc)$- frame by the following equations: 
	\begin{subequations}
		\label{eq: AC dynamics}
		\begin{align}
		C\, \dot v_{abc} &= -i_{load}+i_{abc}
		\\
		\,
		L\dot {(i_{abc})} &= -R i_{abc} + v_{x} - v_{abc}
		\,,
		\label{eq: modulation equation}
		\end{align}
	\end{subequations}
	where $v_{abc}$ is the capacitor terminal voltage, $i_{load}$ is the current drawn by the grid load, $i_{abc}$ is the inductance current and $v_{x}$ the terminal voltage at the modulation output block. 
	
	\item The switching block represents an averaged model of a $6$-switch, $2$-level inverter which modulates DC voltage into AC voltage according to a complementary switching pattern and a modulation signal $m_{abc}$. For the time scale of interest, we assume a high enough switching frequency which allows us to discard all switching harmonics (no ripples). This block represents the main nonlinearity in our system and is defined using the following identities, as in \cite{AT-RI:09}:
	\[
	i_{x} = \frac{1}{2}m_{abc}^{\top}i_{abc}\,,	
	v_{x} =\frac{1}{2}m_{abc} v_{dc}
	\,,
	\]
	with $m_{abc}\in\real^{3}$ as the modulation signal in $(abc)$-frame, corresponding to the average of the converter duty cycle over one switching period and therefore its components take values in $[-1,1]$. The modulation signal will serve as our main control input later on.
\end{itemize}

By combining \eqref{eq: AC dynamics} and \eqref{eq: modulation equation}, we can rewrite the three-phase inverter model equations in $(abc)$- frame as: 

\begin{subequations}
	\begin{align*}
	C_{dc}\dot v_{dc} &= -G_{dc} v_{dc}+i_{dc}-\frac{1}{2}m_{abc}^{\top}i_{abc}
	\\
	\,
	C \dot v_{abc} &= -i_{load}+i_{abc}
	\\
	\,
	L \dot {(i_{abc})} &= -R i_{abc} + \frac{1}{2}m_{abc} v_{dc} - v_{abc}
	\,,
	\end{align*}
	\label{eq: inverter dynamics abc}
\end{subequations}
where all quantities are averaged over one switching cycle.
\\
We refer to all quantities in the AC circuit as vectors in $\mbb R^3$ with three elements, each describing a phase in the $(abc)$- frame, whereas DC quantities are considered to be real.

\subsubsection{Representation of the DC/AC converter dynamics in $\alpha\beta$- frame}

We perform a transformation from $(abc)$- into $(\alpha\beta)$- frame using the transformation matrix $T_{\alpha\beta}\in\real^{3\times 3}$ for a balanced three-phase system as described in \eqref{eq: abc to alphabeta}. We assume that the dynamics of the third component is asymptotically stable and decoupled from the dynamics of the DC/AC converter and do not consider it further.
AC signals represented in $(\alpha\beta)$- frame are denoted with the index $(\alpha\beta)$.
	
	\begin{subequations}
		\begin{align}
		C_{dc}\dot v_{dc} &= -G_{dc} v_{dc}+i_{dc}-\frac{1}{2}m_{\alpha\beta}^{\top}i_{\alpha\beta}
		\\
		\,
		C \dot v_{\alpha\beta} &= -i_{load}+i_{\alpha\beta}
		\\
		\,
		L \dot {(i_{\alpha\beta})} &= -R i_{\alpha\beta} + \frac{1}{2}m_{\alpha\beta} v_{dc} - v_{\alpha\beta}
		\,.
		\end{align}
		\label{eq: inverter dynamics}
	\end{subequations}
	
We represent AC signals in $(\alpha\beta)$- frame in the remainder of this work.

\chapter{Review of the DC/AC converter control schemes}
\label{sec: converter control schemes}

We dedicate this section to review the main control strategies, adopted by conventional power system community in order to control the three- phase DC/AC converter using as main input the modulation signal $m_{\alpha\beta}$ for tracking a given sinusoidal reference signal across the capacitor voltage $v_{\alpha\beta}$ corresponding to the DC/AC converter terminal voltage. 

\section{Inner-loop control}
{\em Inner-loop} accounts for the cascaded voltage and current control for a given reference for the voltage across the capacitor at the terminals of the DC/AC converter. It is the {\em innermost} level of control in traditional control architectures.

We present in this section a commonly-used control configuration,  similar to that of \cite{SDA-SJA:13}.
Let $v_{ref}$ be a given reference for the voltage across the capacitor at the output of the DC/AC converter as depicted in Figure \ref{fig: circuit diagram}. The DC/AC converter can be controlled via the modulation signal $m_{\alpha\beta}$ in order to achieve fast and exact tracking of this voltage reference. 


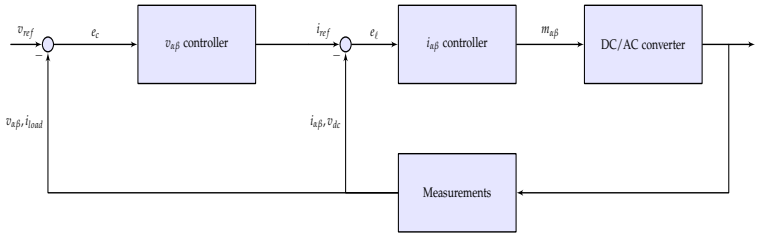
\begin{figure}[h!]
	\begin{center}
		\resizebox{10cm}{3cm}{
		\begin{tikzpicture}[auto, node distance=3cm,>=latex']
		\node [input, name=input] {};
		\node [sum, right of=input] (sum) {};
		\node [block, right of=sum] (controller) {$i_{\alpha\beta}$ controller};
		\node [block, right of=controller, node distance=5cm] (system) {DC/AC converter};
		\node [block, left of=sum, node distance=4cm] (controller2) {$v_{\alpha\beta}$ controller};
		\node [sum, left of=controller2, node distance=4cm] (sum2) {};
		\node [input, name=input2 , left of=sum2, node distance=1cm] {};
		
		\draw [->] (controller) -- node[name=u] {$m_{\alpha\beta}$} (system);
		\node [output, right of=system] (output) {};
		\node [block, below of=controller, node distance=3cm] (measurements) {
			Measurements
		};
		
		\draw [draw,->] (input) -- node {$i_{ref}$} (sum);
		\draw [draw,->] (input2) -- node {$v_{ref}$} (sum2);
		\draw (controller2) -- (input);
		\draw [->] (sum) -- node {$e_\ell$} (controller);
		\draw [->] (sum2) -- node {$e_c$} (controller2);
		\draw [->] (system) -- node [name=y] {$ $}(output);
		\draw [->] (y) |- (measurements);
		\draw [->] (measurements) -| node[pos=0.99] {$-$} 
		node [near end] {$i_{\alpha\beta}, v_{dc}$} (sum);
		\draw [->] (measurements) -| node[pos=0.99] {$-$} 
		node [near end] {$v_{\alpha\beta}, i_{load}$} (sum2);
		\end{tikzpicture}}
		\caption{Control architecture using inner-loop control to track a given reference $v_{ref}$}
		\label{fig: inner-loop-ctrl}
	\end{center}
\end{figure}

Consider Figure \ref{fig: inner-loop-ctrl} and suppose the given reference $v_{ref}$ is defined by its amplitude $\hat {v}$, its angle $\theta$ and its frequency $\dot\theta=\omega_{ref}$ in $(\alpha\beta)$- frame as follows
\begin{subequations}
\begin{align*}
v_{ref}&=\hat {v}\sqrt{\frac{3}{2}} \begin{bmatrix} -\sin(\theta(t)) \\
\cos(\theta(t)) \\ 0 \end{bmatrix} = T_{\alpha\beta}\left(\hat {v}\begin{bmatrix}
-\sin(\theta(t)) \\ \sin(\theta(t)-\frac{2\pi}{3})\\
\sin(\theta(t)+\frac{2\pi}{3})
\end{bmatrix}\right)
\\
\dot{\theta}(t)&=\omega_{ref}
\,, t>0\,,
\end{align*}
\end{subequations}

where $T_{\alpha\beta}$ as defined in \eqref{eq: abc to alphabeta}.

We omit the third component for a balanced three-phase signal and take $v_{ref}\in\real^2$ as
\begin{equation*}
v_{ref}=\hat {v}\sqrt{\frac{3}{2}} \begin{bmatrix} -\sin(\theta(t)) \\
\cos(\theta(t))
\end{bmatrix}
\,.
\end{equation*}

We regard the classical, cascaded loop approach as a feedback-linearization control design.
Using the capacitor equation as defined in \eqref{eq: inverter dynamics}, we can write the voltage tracking problem as follows by expressing it in error coordinates with $e_c=v_{ref}-v_{\alpha\beta}$.

\begin{subequations}
	\begin{align*}
	C \dot v_{\alpha\beta} &= K_{pc} (v_{ref}-v_{\alpha\beta})+K_{ic} \int(v_{ref}-v_{\alpha\beta})
	\\
	\dot e_c &= \frac{-K_{pc}}{C} e_c +\frac{-K_{ic}}{C} \int e_c=-\lambda_p e_c-\lambda_i \int e_c
	\,,
	\end{align*}
\end{subequations}
with $K_{pc},K_{ic}>0$ and $\lambda_p,\lambda_i>0$.

Now, that the inductance current reference can be deduced from the feedback linearization, we have

\begin{equation*}
i_{ref}=C \dot v_{\alpha\beta}+i_{load}=-\lambda_p e_c-\lambda_i \int e_c+i_{load}
\,.
\label{eq: ref-curr}
\end{equation*}

In order to track the given current reference $i_{ref}$, we define the following error dynamics of the current $e_\ell=i_{ref}-i_{\alpha\beta}$, by

\begin{subequations}
	\begin{align*}
	L(\dot i_{\alpha\beta})&=K_{pl} (i_{ref}-i_{\alpha\beta})
	\\
	\dot e_{\ell}&=-\frac{K_{pl}}{L} e_\ell=\lambda_l e_\ell
	\\
	&=-R i_{\alpha\beta} + \frac{1}{2}m_{\alpha\beta}v_{dc} - v_{\alpha\beta}
	\,.
	\end{align*}
\end{subequations}

By applying \eqref{eq: ref-curr} in the inductance equation as defined in \eqref{eq: inverter dynamics}, we have the following modulation signal as input to the DC/AC converter

\begin{equation*}
m_{\alpha\beta}=\frac{2}{v_{dc}}(v_{\alpha\beta} -\frac{K_{pl}}{L} e_\ell + R i_{\alpha\beta})
\,.
\end{equation*}

We proceed as follows to select the parameters $K_{pc},\, K_{ic},\,K_{pl}>0$ 
\begin{equation*}
\ddot e_c = \frac{-K_{pc}}{C}\dot e_c +\frac{-K_{ic}}{C} e_c=-\lambda_p \dot {e}_c-\lambda_i  e_c
\,.
\end{equation*}
with $\lambda_p,\lambda_i>0$.
We calculate the eigenvalues of the closed-loop system including the capacitor voltage controller with DC/AC converter.
We choose $\lambda_p,\,\lambda_i>0$ to satisfy critical damping such that the closed loop system has a double eigenvalue at $\lambda_0<0$ with 
\begin{equation*}
\lambda_0=\frac{-\lambda_p\pm\sqrt{\lambda_p^2 - 4\lambda_i}}{2}<0
\,.
\end{equation*}	
Using $\lambda_p^2=4\lambda_i$, we have
\begin{subequations}
	\begin{align*}
	K_{pc}=C\lambda_p &=-2C\lambda_0
	\\
	K_{ic}=C\lambda_i &= C\lambda_0^2
	\,.
	\end{align*}
\end{subequations}
In order to place the poles of the closed-loop system based on a time-scale separation between the inductance current controller and the voltage controller, we choose the inductance current controller to be at least $10$-times faster than $v_{\alpha\beta}$ controller as follows

\begin{equation*}
|\lambda_\ell|>10|\lambda_0|
,\,\,
K_{pl}=-L\lambda_\ell 
\,.
\end{equation*}

\subsubsection{Simulation results}
	For the following simulation case study, we consider a converter rated in the $10$ KW range, with the choice of parameters as: $i_{dc}=100A, G_{dc}=0.1\Omega^{-1}, C_{dc}=0.001F,\, R=0.1 \Omega,\, L=5.10^{-4}H,\, C=10^{-5}F$ yielding the reference signal defined by its amplitude $\hat {v}= 165V$ and frequency $\omega_{ref}=2 \pi 50\, rad/s$ and nominal DC voltage of $v_{dc, ref} = 1000V$. 
	This example of the reference $v_{ref}$ could be interpreted as a three-phase Sinwave generator with an amplitude $\hat {v}$ and $\omega_{ref}$ after transformation into $(abc)$- frame.

	In order to track the reference voltage ${v}_{ref}$, we choose the PI controller gains as 
	\begin{equation*}
	K_{pc} = 1\,, 
	K_{ic} = 25000\,,
	K_{pl} = 250\,,
	\lambda_0=-5 \cdot 10^4 s^{-1}
	\,,
	\lambda_\ell=-5 \cdot 10^5 s^{-1}
	\,.
	\end{equation*}
	
	Simulation results are depicted in Figure \ref{fig: inner-loop}.
	
	\begin{figure}[h!]
		\centering
		\includegraphics[scale=0.3]{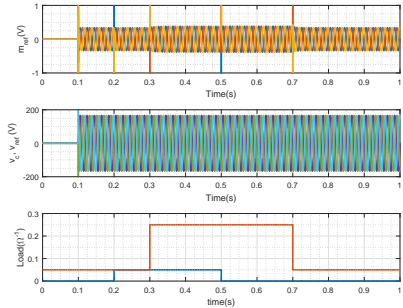}
		\caption{Time-domain simulations of the modulation signal as control input of the DC/AC converter resulting from tracking a given reference signal across the capacitor $v_{ref}$ while interfacing  resistive and reactive load undergoing a step change.}
		\label{fig: inner-loop}
	\end{figure}

The ongoing evolution of control strategies of DC/AC converters suggests that the reference signal $v_{ref}$ can be itself generated from an upper controller called in the following as {\em outer-loop} control which takes into account other control purposes, related for instance to interfacing various inverters with a grid network. 
Next, we propose different forms of outer-loop controllers in order to generate a reference signal $v_{ref}$ for the voltage across the capacitor, considered to be the converter terminal voltage, as shown in Figure \ref{fig: upper-controller}, which can be tracked using the inner-loop controllers introduced above. 

\begin{figure}[h!]
	\centering
	\resizebox{12cm}{3cm}{
		\begin{tikzpicture}[auto, node distance=3cm,>=latex']
		\node [input, name=input] {};
		\node [sum, right of=input] (sum) {};
		\node [block, right of=sum] (controller) {$i_{\alpha\beta}$- controller};
		\node [block, right of=controller, node distance=5cm] (system) {DC/AC converter};
		\node [block, left of=sum] (controller2) {$v_{\alpha\beta}$-controller};
		\node [block, left of=sum2] (controller3) {outer-loop controller};
		\node [sum, left of=controller2, node distance=3cm] (sum2) {};
		\node [input, name=input2 , left of=sum2, node distance=1cm] {};
		
		\draw [->] (controller) -- node[name=u] {$m_{\alpha\beta}$} (system);
		\node [output, right of=system] (output) {};
		\node [block, below of=u, node distance=3cm] (measurements) { Measurements};
		
		\draw [draw,->] (input) -- node {$i_{ref}$} (sum);
		\draw (controller3) -- (input2);
		\draw [draw,->] (input2) -- node {$v_{ref}$} (sum2);
		\draw [->] (sum) -- node {$e_x$} (controller);
		\draw [->] (sum2) -- node {$e_\ell$} (controller2);
		\draw [->] (system) -- node [name=y] {$ $}(output);
		\draw [->] (y) |- (measurements);
		\draw [->] (measurements) -| node[pos=0.99] {$-$} 
		node [near end] {$i_{\alpha\beta}, v_{dc} $} (sum);
		\draw [->] (measurements) -| node[pos=0.99] {$-$} 
		node [near end] {$ v_{\alpha\beta}, i_{load}$} (sum2);
		\end{tikzpicture}}
	\caption{Control architecture for tracking a generated reference $v_{ref}$ from an outer-loop controller forwarded to inner-loop control}
	\label{fig: upper-controller}
\end{figure}
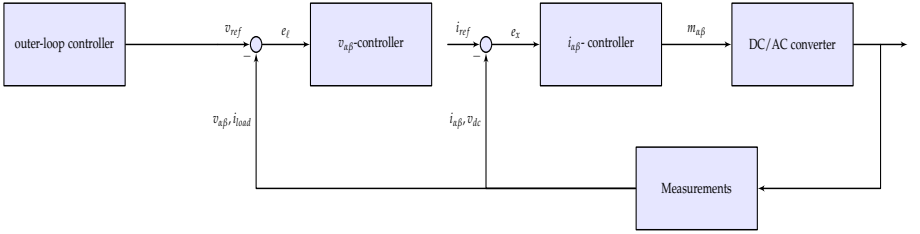

\section{Outer-loop: droop control}
The vast majority of academic and industrial efforts approach  real-time control challenges by means of {\em droop control}. Drawing from the control of synchronous generators in bulk power systems, droop control is a control approach which linearly trades off the active and reactive power injection with the inverter's terminal-voltage amplitude and frequency \cite{MS-FD-BBJ-SVD:14}.
This implies that there exists a relationship between the calculated power at the converter output and that of the voltage at its terminal.

We briefly review {\em frequency droop} control in resistive AC microgrids as described in \cite{JZ-FD:15}. The active power injection $P$ at each source is controlled to be proportional to its frequency deviation $\dot \theta$ (from a nominal frequency $\dot \theta_0$) as

\begin{equation}
{\dot \theta=\dot \theta_0 + n (P_0 - P)}
\,,
\label{eq: droop-freq}
\end{equation}
where $\dot \theta_0=\omega_{ref}$ is the nominal frequency, $n>0$ refers to the frequency droop coefficient and $P_0\in[0, P_{max}]$ is a nominal injection setpoint. $P_{max}$ is the maximal active power injection capacity of a source.

A key feature of AC frequency droop control in a network is that it synthesizes the synchronous frequency as a global variable indicating the load/generation imbalance in the microgrid \cite{FD-SP-FB:14}.

We use the following {\em amplitude droop} controller, inducing a trade-off between the amplitude $\hat v>0$ to the reactive power at the terminals of the DC/AC converter denoted by $Q$ as follows
\begin{equation}
{\hat v= \hat {v}_0 + d\, (Q_0-Q)}
\,,
\label{eq: droop-amp}
\end{equation}
where $d>0$ represents the amplitude droop control coefficient, $Q_0 \in[-Q_{max}, Q_{max}]$ and $\hat{v}_0>0$ refers to respectively the nominal setpoint of the reactive power and the amplitude at the converter output terminal voltage, corresponding to the voltage of the AC capacitor.

The reference voltage resulting from the implementation of the resistive droop control in $(\alpha\beta)$- frame is given by
\begin{equation}
v_{ref, droop}= \hat v\begin{bmatrix}
-\sin(\theta) \\ \cos(\theta)
\end{bmatrix}
\,.
\end{equation}

The curves in Figure \ref{fig: droop-curves} depict droop controller in amplitude as described in \eqref{eq: droop-amp} and frequency as in \eqref{eq: droop-freq}.

\begin{figure}[h!]
	\centering	
	\subfloat[Droop control in frequency for nominal setpoint ($P_0$, $f_0$) with $f_0=\omega_ref/2\pi$]{\includegraphics[scale=0.3]{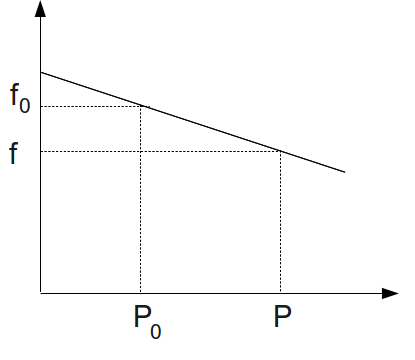}}
	\,
	\subfloat[Droop control in amplitude for nominal setpoint ($Q_0$, $v_0$)]{\includegraphics[scale=0.3]{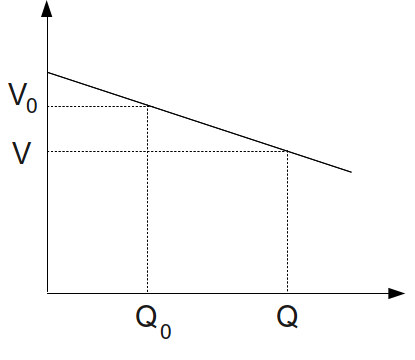}}
	\caption{Curves of droop control method in amplitude $\hat v>0$ and frequency $f=\dot \phi/2\pi$}
	\label{fig: droop-curves}
\end{figure}

Summing up, the control scheme of this outer-loop controller can be described in $(\alpha\beta)$- frame by the  diagram \ref{fig: droop-controller}.

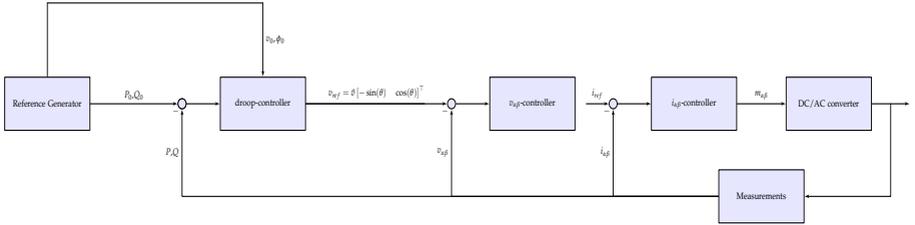
\begin{figure}[h!]
	\centering
	\resizebox{12cm}{3cm}{
		\begin{tikzpicture}[auto, node distance=3cm,>=latex']
		\node [input, name=input] {};
		\node [sum, right of=input] (sum) {};
		\node [block, right of=sum] (controller) {$i_{\alpha\beta}$-controller};
		\node [block, right of=controller, node distance=5cm] (system) {DC/AC converter};
		\node [block, left of=sum] (controller2) {$v_{\alpha\beta}$-controller};
		\node [block, left of=sum2,,node distance=7cm] (controller3) {droop-controller};
		\node [sum, left of=controller2, node distance=3cm] (sum2) {};
		\node [sum, left of=controller3, node distance=3cm] (sum3) {};
		\node [block, left of=sum3,node distance=5cm] (controller4) {Reference Generator};
		\node [input, name=input2 , left of=sum2, node distance=1cm] {};
		\node [input, name=input3 , left of=sum3, node distance=1cm] {};
		\node [input, name=input4 , above of=controller3, node distance=3cm] {};
		\node [input, name=input5 , above of=controller4, node distance=3cm] {};
		\node [sum, left of=controller3, node distance=3cm] (sum3) {};
		\node [input, name=input3 , left of=sum3, node distance=1cm] {};
		\node [input, name=input6 , above of=sum2, node distance=1cm] {};
		\draw [->] (controller) -- node[name=u] {$m_{\alpha\beta}$} (system);
		\node [output, right of=system] (output) {};
		\node [block, below of=u, node distance=3cm] (measurements) {Measurements};
		
		\draw [draw,->] (input) -- node {$i_{ref}$} (sum);
		\draw (controller3) -- (input2);
		\draw [->] (sum) -- node {$ $} (controller);
		\draw [->] (sum2) -- node {$ $} (controller2);
		\draw [->] (sum3) -- node {$ $} (controller3);
		\draw [->] (controller4) -- node {$P_0,Q_0$} (sum3);
		\draw [->] (controller3) -- node {$v_{ref}=\hat v\begin{bmatrix}-\sin(\theta) & \cos(\theta)\end{bmatrix}^{\top}$} (sum2);
		\draw [->] (system) -- node [name=y] {$ $}(output);
		\draw [->] (y) |- (measurements);
		\draw [->] (measurements) -| node[pos=0.99] {$-$} 
		node [near end] {$i_{\alpha\beta}$} (sum);
		
		\draw [->] (input4) -| node [near end]{$v_0, \phi_0$} (controller3);
		\draw [->] (measurements) -| node[pos=0.99] {$-$} 
		node [near end] {$P,Q$} (sum3);
		\draw [->] (measurements) -| node[pos=0.99] {$-$} 
		node [near end] {$v_{\alpha\beta}$} (sum2);
		
		\draw [-] (controller4.north) --  node {$ $} (input5); 
		\draw [-] (input5) --  node {$ $} (input4); 	
		\end{tikzpicture}}
	\caption{Control architecture for tracking a generated reference $v_{ref}$ from an upper controller represented by droop control referenced by a generator setting nominal operating conditions.}
	\label{fig: droop-controller}
\end{figure}

\subsubsection{Simulation results}
	We consider again the single inverter case as introduced previously. In order to design droop control in amplitude and frequency, we define the following nominal active and reactive power as well as the control gains
	\begin{equation*}
	P_0=10^4W,\, Q_0= 2000VAR,\,\, d=2\cdot10^{-3}V/VAR,\,\, n=2\cdot10^{-3}V/W
	\,.
	\end{equation*}  	
	
	Simulation results in Figure \ref{fig: sim-droop-ctrl} reflect droop control laws introduced in \eqref{eq: droop-amp} and  \eqref{eq: droop-freq}.
	
	\begin{figure}[h!]
		\centering
		\includegraphics[scale=0.3]{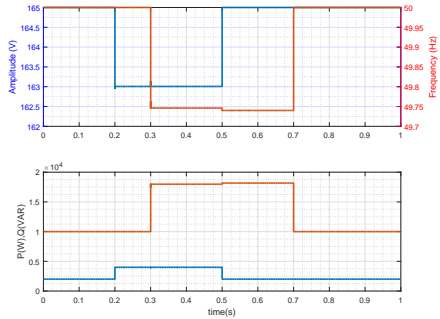}
		\caption{Time-domain simulations of active and reactive power and its effect on the amplitude and frequency at the converter terminal voltage}
		\label{fig: sim-droop-ctrl}	
	\end{figure}

Once the DC/AC converter operating in nominal conditions, is interfaced with active and reactive load, the reference in amplitude and frequency of the capacitor at the output of the DC/AC converter is reacting on an increase/decrease of the load by a decrease/increase in amplitude and frequency. This implies that the frequency and amplitude at the AC capacitor is adjusted to the load generation imbalance of single inverter case according to the resistive droop control law.  

\begin{remark}
	Depending on the choice of the droop coefficients $d, n>0$ the droop in amplitude and frequency, one can define the trade-off to active and reactive power demand of the load. These gain coefficients have positive values in the droop control literature as in \cite{MS-FD-BBJ-SVD:14,JZ-FD:15}.  
\end{remark}

Primary droop control achieves stable proportional load sharing in a fully decentralized way while respecting actuation constraints \cite{JZ-FD:15}.
This strategy can be further extended via secondary control objectives comprising an optimal economic dispatch which defines a general optimization problem to be solved. 
This can be even more enhanced by means of tertiary control objectives aiming to correct the steady state and recover optimality via distributed averaging based control strategy \cite{CZ-EM-FD:15}, which requires the establishment of a communication, i.e local frequency sensing and neighborhood communication between the different decentralized controllers via modern techniques.
Therefore, droop control lends itself useful to ensure stability of power networks and achieve economic dispatch between generators and controllable loads.

Despite its effectiveness in network regulation, droop control methods still presume the existence of a quasi-stationary sinusoidal steady state and operate mainly on phasor quantities \cite{MS-FD-BBJ-SVD:14}. In addition, it assumes the knowledge of given nominal active and reactive power references in advance, which are not always easy to obtain. 
Moreover, to the author's knowledge inner control loops are always considered as "blackbox" and its dynamics are not explicitly included in the analysis of multiple inverters connected to a grid network, which makes relevant signals almost non-tractable.

\section{Outer-loop: virtual oscillator controller}

We study a recently approached control scheme that ensures network stability of power electronic inverters by emulating the dynamics of the Van der Pol oscillator. This compelling time-domain alternative presents a novel control strategy to generate nonlinear oscillations, known for their robustness and structural stability as well as the existence of an only one and unique isolated periodic orbit in comparison to a continuum of closed orbits for harmonic oscillators \cite{HK}.

Virtual oscillator control (VOC) refers to a digital control strategy emulating the behavior of oscillators like that of Van der Pol, programmed on real micro-controllers and applied to H-bridge inverters.

\subsection{Oscillator-based reference generation}

\subsubsection{Van der Pol oscillator in closed-loop fashion}
The Van der Pol oscillator can be described by the following dynamics \cite{HK}
\begin{subequations}
	\begin{align}
	\dot x_1 &=x_2
	\\
	\dot x_2 &=-x_1+\mu(1-x_1^2)x_2
	\,,
	\end{align}
	\label{eq: van-der-pol}
\end{subequations}
with $\mu>0$ is a parameter to be specified.
\\
This ordinary differential equation, which was used by Van der Pol to study oscillations in vacuum tube circuits, is a fundamental example in nonlinear oscillations theory. It is an equation describing self-sustaining oscillations, in which the net exchange of energy over one cycle is zero.
It possesses a periodic solution that attracts every other solution except the zero \cite{HK}. Depending on the values of $\mu$, a small $(0.2)$, medium $(1.0)$ or large value $(5.0)$, we can get different forms of phase portrait of Van der Pol as depicted in Figure \ref{fig: van-der-pol}. If $\mu=0$, the equation reduces to that of a simple harmonic motion
\begin{subequations}
	\begin{align*}
	\dot x_1 &=-x_2
	\\
	\dot x_2 &=x_1
	\,.
	\end{align*}
\end{subequations}

The parameter $\mu>0$ determines as well how fast/slow the dynamics of Van der Pol oscillator are \cite{HK}.

\begin{figure}[h!]
	\centering
	\includegraphics[scale=0.3]{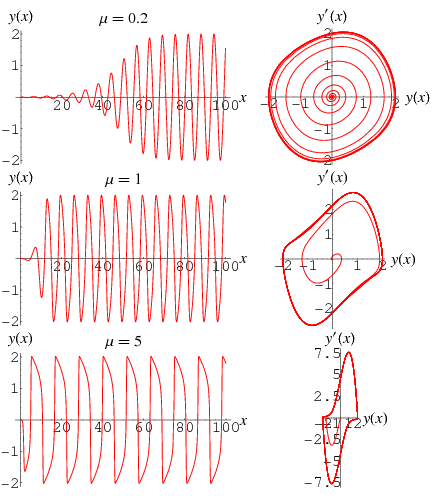}
	\caption{Limit cycle oscillations for different values of the parameter $\mu>0$ with $y=x_1$ and $y'=x_2$ \cite{wolfram}}
	\label{fig: van-der-pol}
\end{figure}

The emulation of Van der Pol goes usually by implementing an oscillator model as in \cite{BBJ-SVD-AOH-PTK:14a} composed of a linear subsystem consisting of an RLC circuit, set in parallel with a nonlinear voltage-dependent current source. The special choice of this current source is based on emulating the nonlinearity of Van der Pol for a given $\mu>0$.

In comparison to the commonly used single-phase model, we design a nonlinear virtual oscillator for the three-phase DC/AC converter by implementing a virtual oscillator emulating Van der Pol for each phase in $(\alpha\beta)$- frame. We design a feedback oscillator controller for the $\alpha$- component as described in \cite{MS-FD-BBJ-SVD:14}.

\begin{subequations}
	\begin{align*}
	\dot x_{1,\alpha} &=\omega_0 x_{2,\alpha}
	\\
	\dot x_{2,\alpha} &=-\omega_0 x_{1,\alpha}+\mu(1-x_{1,\alpha}^2)x_{2,\alpha}+\kappa i_{load, \alpha}(t)
	\,,
	\end{align*}
\end{subequations} 

and a second oscillator for the $\beta$- component defined by

\begin{subequations}
	\begin{align*}
	\dot x_{1,\beta} &=\omega_0 x_{2,\beta}
	\\
	\dot x_{2,\beta} &=-\omega_0 x_{1,\beta}+\mu(1-x_{1,\beta}^2)x_{2,\beta}+\kappa i_{load,\beta}(t)
	\,,
	\end{align*}
\end{subequations} 

with $\kappa>0$ and the load current defined as 

\begin{equation*}
i_{load}=\begin{bmatrix}
i_{\alpha,load}\\
i_{\beta,load}
\end{bmatrix}
\,.
\end{equation*}

\begin{remark}[VOC and droop control \cite{MS-FD-BJ-SD:14b}]
	The virtual oscillator control stabilizes arbitrary initial conditions to a sinusoidal steady state, while droop control acts on the phasor quantities and is only well-defined in the sinusoidal steady state as depicted in the Figure \ref{fig: droop-vs-voc}, where  the voltage $v$ is accounting for the $\alpha$- component and the current $i$ for the $\beta$- component for a circuit realization of Van der Pol oscillator. Hence, droop control and VOC can be implemented together to stabilize AC signals to waveforms with a predefined time-scale separation between the two controllers \cite{MS-FD-BJ-SD:14b}.	
	
	\begin{figure}[h!]
		\centering
		\includegraphics[scale=0.3]{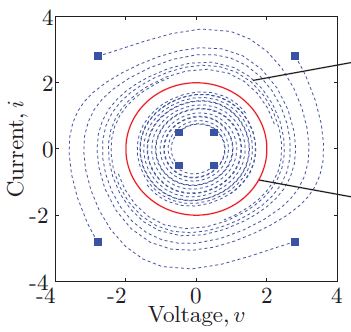}
		\caption{Representation of Droop control laws as embedded within the nonlinear dynamics of Van der Pol Oscillators.}
		\label{fig: droop-vs-voc}
	\end{figure}
\end{remark}

\subsubsection{Choice of initial conditions}
The Van der Pol oscillator is defined in \eqref{eq: van-der-pol} in Cartesian coordinates. For each phase, we redefine the oscillator dynamics in polar coordinates as follows for the $\alpha$- component
	\begin{subequations}
		\begin{align*}
		\hat {v}_\alpha&=\sqrt{x_{1,\alpha}^2+x_{2,\alpha}^2}\\
		\phi_\alpha&=\arctan\left(\frac{x_{1,\alpha}}{x_{2,\alpha}}\right)
		\,,
		\end{align*}
	\end{subequations}
	and the $\beta$- component as follows
	
	\begin{subequations}
		\begin{align*}
		\hat {v}_\beta&=\sqrt{x_{1,\beta}^2+x_{2,\beta}^2}\\
		\phi_\beta&=\arctan\left(\frac{x_{1,\beta}}{x_{2,\beta}}\right)
		\,.
		\end{align*}
	\end{subequations}
	
	Leveraging the fact that the angles of the virtual inductor current and capacitor voltage are orthogonal as introduced in \cite{BBJ-SVD-AOH-PTK:14a} (corresponding here to angles of $x_{1,\alpha}$ and $x_{1,\beta}$ signals), we initialize the $\alpha$- and $\beta$- oscillators orthogonally such that the following holds
	\begin{subequations}
		\begin{align*}
		\phi_\alpha(0)-\phi_\beta(0) &=\frac{\pi}{2}+k\pi,\, k\in Z
		\\
		\hat {v}_\alpha(0)&= \hat {v}_\beta(0)
		\,.
		\end{align*}
	\end{subequations}
	That is both oscillators in $\alpha$- and $\beta$- components are initialized with the same initial amplitude. 
	We then apply a transformation from $(\alpha\beta)$ to $(abc)$- frame to yield three-phase signals.
	
	We assume that, once we start with orthogonal $\alpha$ and $\beta$ components, this condition is not violated for all times $t>0$, which turns out to be well-justified by our simulations.

In summary, we define the reference given by the outer-loop virtual oscillator controller for the AC capacitor voltage as

\begin{equation*}
v_{ref,voc}= \hat{v}_{voc} \begin{bmatrix}
x_{1,\alpha} \\ x_{1,\beta}
\end{bmatrix}=\hat {v} \sqrt{\frac{3}{2}}\begin{bmatrix}
\sin(\theta) \\ \cos(\theta)
\end{bmatrix},\,\dot{\theta}(t)=\omega_{ref},\, t>0
\,.
\end{equation*}

We design the amplitude $\hat {v}_{voc}$ such that 

\begin{equation*}
\hat {v}_{voc}=\frac{\hat  {v}_{ref}\sqrt{\frac{3}{2}}}{\sqrt{x_{1,\alpha}^2+x_{1,\beta}^2}}
\,.
\end{equation*}

\subsubsection{Simulation results}
	We simulate the DC/AC converter using the following VOC  parameters 
	
	\begin{equation*}
	\mu=0.2 s^{-1},\,\kappa=0.8 V/sA
	\,,
	\end{equation*}
	in order to obtain the reference signal for the voltage across the AC capacitor
	\begin{equation*}
	\omega_{ref}=2\pi50\, rad/s,\, \hat {v}=165V
	\,,
	\end{equation*}
	with the initial conditions of the states $x_{1},\, x_{2}$
	\begin{equation*}
	x_{1,\alpha}(0)=-x_{2,\alpha}(0)=x_{1,\beta}(0)=x_{2,\beta}(0)=1V
	\,,
	\end{equation*}
	such that the following condition is satisfied 
	
	\begin{equation*}
	\phi_\alpha(0)-\phi_\beta(0) =\frac{\pi}{2}
	\,,
	\hat{v}_\alpha(0) = \hat{v}_\beta(0)= 1.4142V
	\,,
	\end{equation*}
	
	and initialized with a fully charged DC capacitor units $v_{dc}(0)=1000V$.
	
	Simulation results are shown in Figure \ref{fig: van-der-pol-ctrl}, when load/ no load is acting on the converter
	
	\begin{figure}[h!]
		\subfloat[A zoomed version around $t=0s$ of time domain simulations of the virtual oscillator control using Van der Pol oscillator]{\includegraphics[scale=0.3]{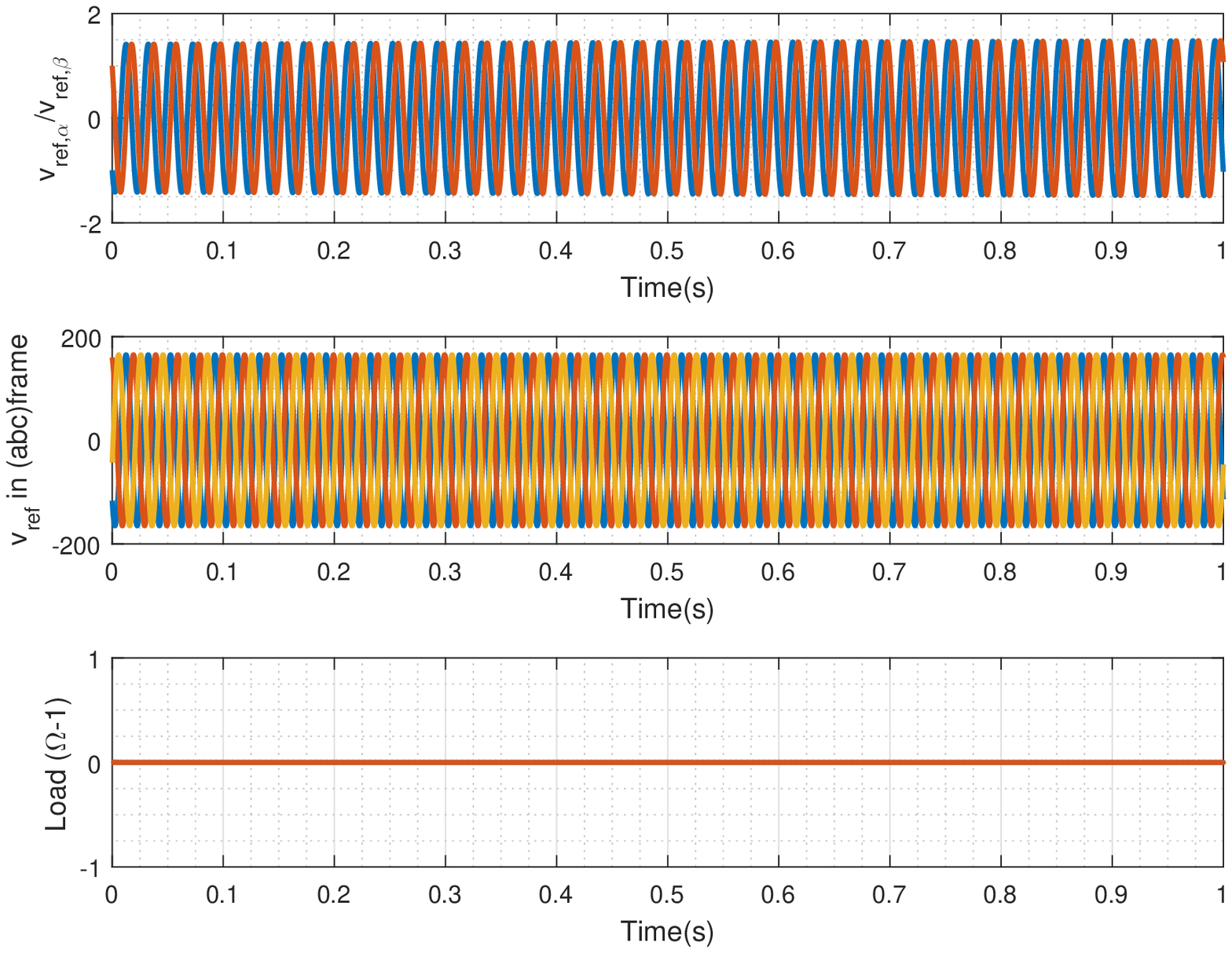}}
		\,
		\subfloat[Time domain simulations of the virtual oscillator control using van der Pol oscillator]{\includegraphics[scale=0.3]{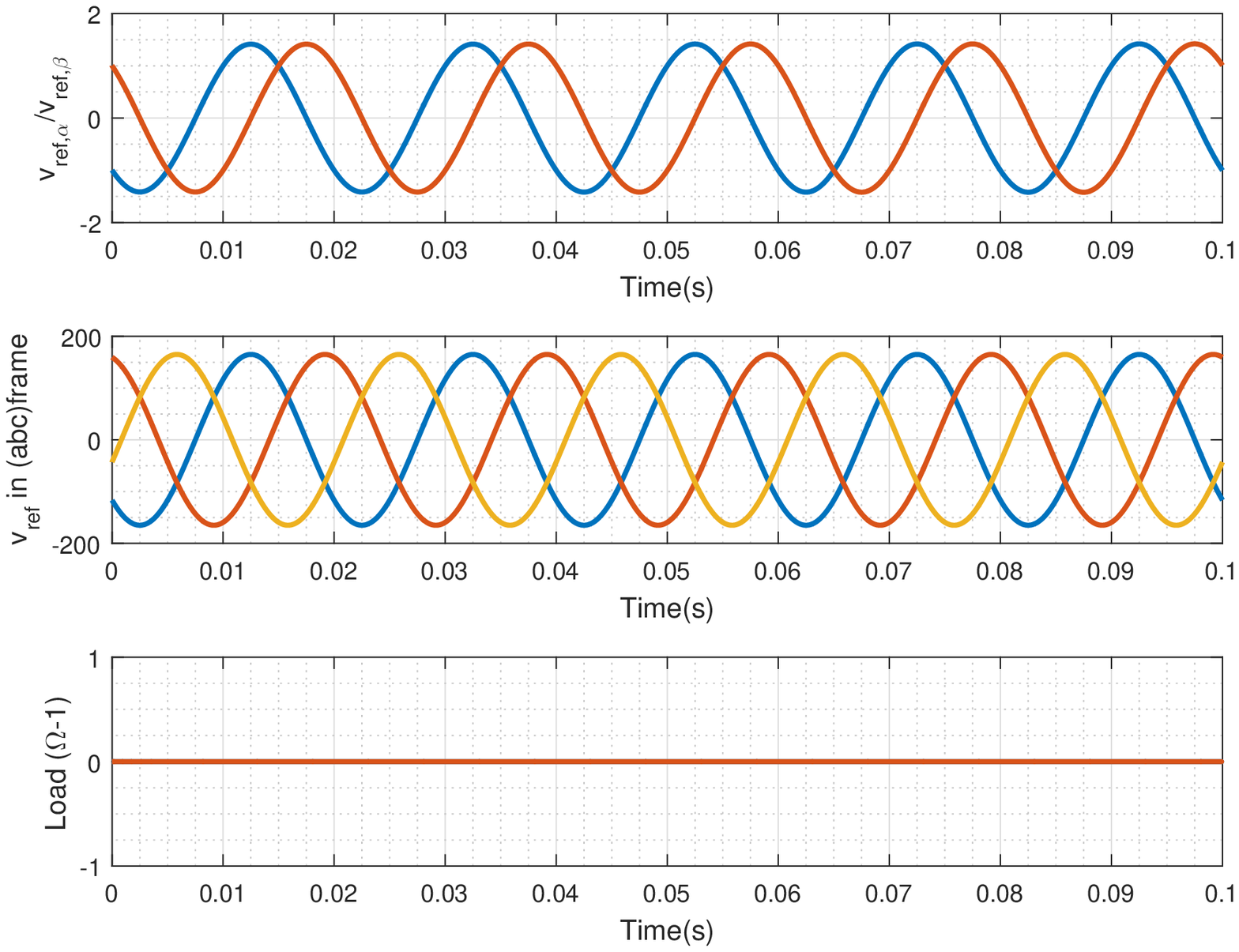}}
		\\
		\subfloat[Time-domain simulations with resistve and reactive load]{\includegraphics[scale=0.3]{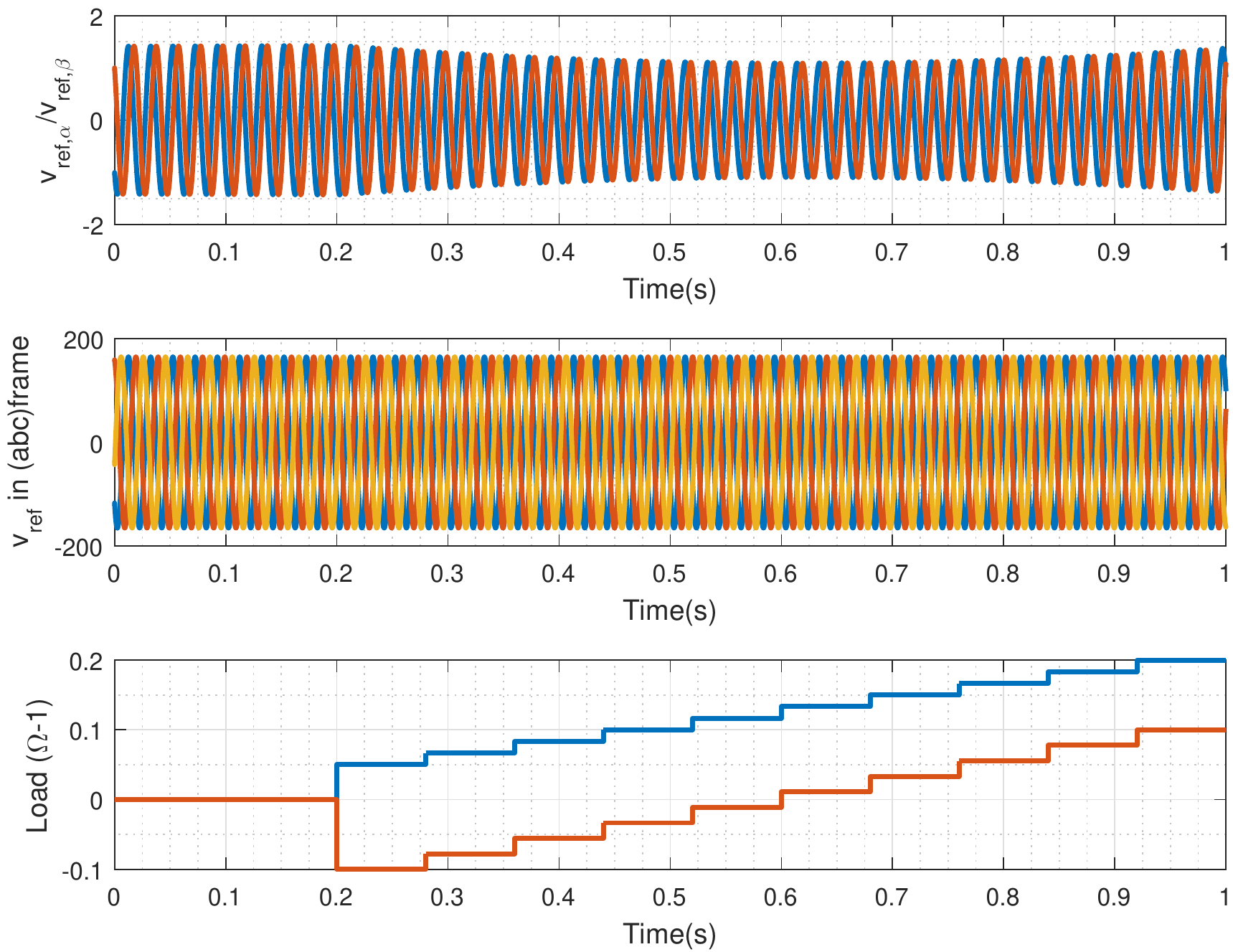}}
		\,
		\subfloat[Time-domain simulations with resistve and reactive load]{\includegraphics[scale=0.3]{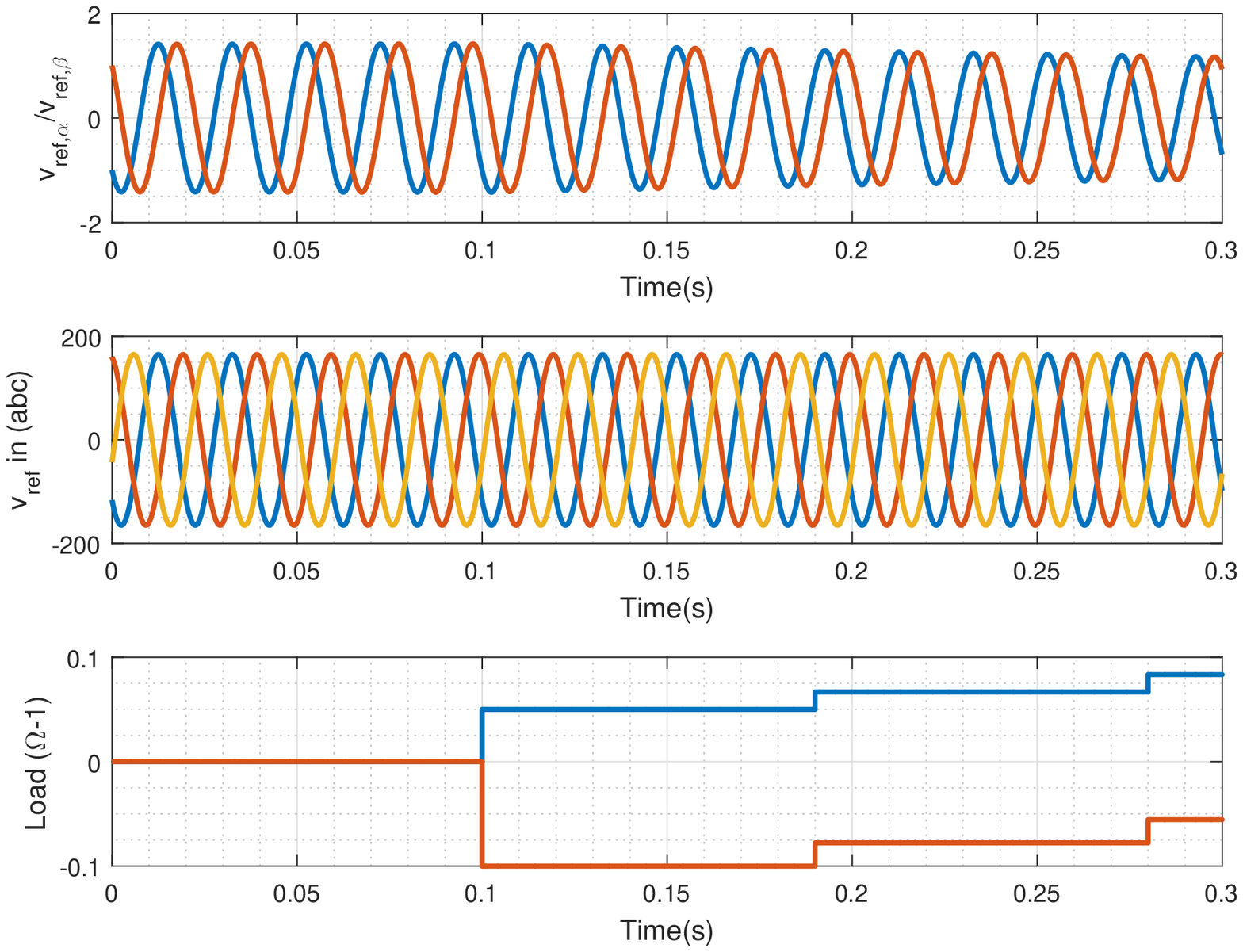}}
		
		\caption{Time-domain simulations of upper controller represented by Van der Pol oscillations for reference generation}
		\label{fig: van-der-pol-ctrl}
	\end{figure}

\subsubsection{Outer loop: Polar VOC in closed-loop fashion}

We consider an alternative of implementing virtual oscillator controllers (VOC), which is the polar virtual oscillator considered to be an outer loop control generating a reference $v_{ref, polar}$ that is handed over to the inner-loop control as depicted in Figure \ref{fig: upper-controller}. In other words, the voltage across the capacitor is referenced by a polar virtual oscillator controller. In fact, we use the features of droop methods mainly the trade-off between active and reactive power to frequency and amplitude of $v_{ref,polar}$, for the implementation of the polar VOC dynamics in order to assure network regulation among many other properties inherited from the droop control method. We can write the control law as follows

\begin{subequations}
	\begin{align*}
	v_{ref,polar}&= \hat v \,\sqrt{\frac{3}{2}} \begin{bmatrix}
	-\sin(\theta) \\ \cos(\theta)
	\end{bmatrix}
	\\
	\dot{\theta} &=\dot\theta_0+d\,(P_0-P)
	\\
	\dot{\hat v} &=\lambda_{osc} (\hat {v}_0+n\,(Q_0-Q)-\hat v),\, \lambda_{osc}>0
	\,.
	\end{align*}
\end{subequations}

\subsubsection{Simulation results}
We simulate the DC/AC converter using the following parameters
	\begin{subequations}
	\begin{align*}
	Q_0&=2000VAR,\,
	P_0=10000W,\,
	d=0.002V/W,\,
	n=0.001V/VAR,
	\\
	\dot\phi_0&=2\pi 50\, rad,\,
	\hat {v}_0= 165V,\,
	\lambda_{osc}=100 s^{-1}\,,
	\end{align*}
	\end{subequations}
	
	and the initial condition for the DC voltage $v_{dc}(0)=0$. Simulation results are depicted in Figure \ref{fig: voc-closed-loop}. 
	\begin{figure}[h!]
		\centering
		\includegraphics[scale=0.3]{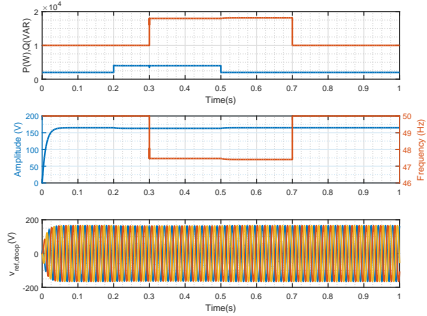}
		\caption{Time-domain simulations of polar VOC in closed-loop fashion, implementing resistive droop control as reference in amplitude and angle after a step change in the load conductance and suseptance}
		\label{fig: voc-closed-loop}
	\end{figure}

\subsection{Oscillator-based modulation assignment}
We aim to directly assign a sinusoid to the converter input, namely the modulation signal $m_{\alpha\beta}$ without using inner-loop control architecture, thus an {\em open-loop} control. 
We introduce in the following the open-loop controller based on virtual oscillations in polar coordinates and defined in $(\alpha\beta)$- frame. Indeed, we define the polar oscillator for the modulation signal by its amplitude and angle defined by
\begin{subequations}
	\begin{align*}
	m_{open-loop} &=\hat {v}_m \sqrt{\frac{3}{2}}\begin{bmatrix}
	-\sin(\theta_m) \\ \cos(\theta_m)
	\end{bmatrix}
	\\
	\dot {\hat{v}}_m &=\lambda_m\,(\hat {v}_{m,ref}-\hat v)
	\\
	\dot\theta_m &=\omega_{ref}
	\,,
	\end{align*}
\end{subequations}   
with $\lambda_m>0$.

\subsubsection{Simulation results}
	We now simulate the open-loop polar oscillator controller using the following parameters. 
	
	\begin{subequations}
	\begin{align*}
	\lambda_m &= 100 s^{-1},\, \hat {v}=165V,\,\omega_{ ref}=2\pi50\, rad/s,\, \hat {v}_{m, ref}=\frac{2\hat {v}}{v_{dc,ref}}=0.33V,\\
	 v_{dc,ref}&=\frac{i_{dc}}{G_{dc}},\,
		\,,
	\end{align*}
	\end{subequations}
	 where $\hat {v}_{m, ref}$ is chosen accordingly in order to get a capacitor voltage of amplitude $\hat {v}=165V$ and frequency $\omega_{ref}$.
	Simulation results are depicted in Figure \ref{fig: polar-voc}.
	
	\begin{figure}[h!]
		\subfloat[Time domain simulation of modulation signal and capacitor voltage of polar VOC]{\includegraphics[scale=0.3]{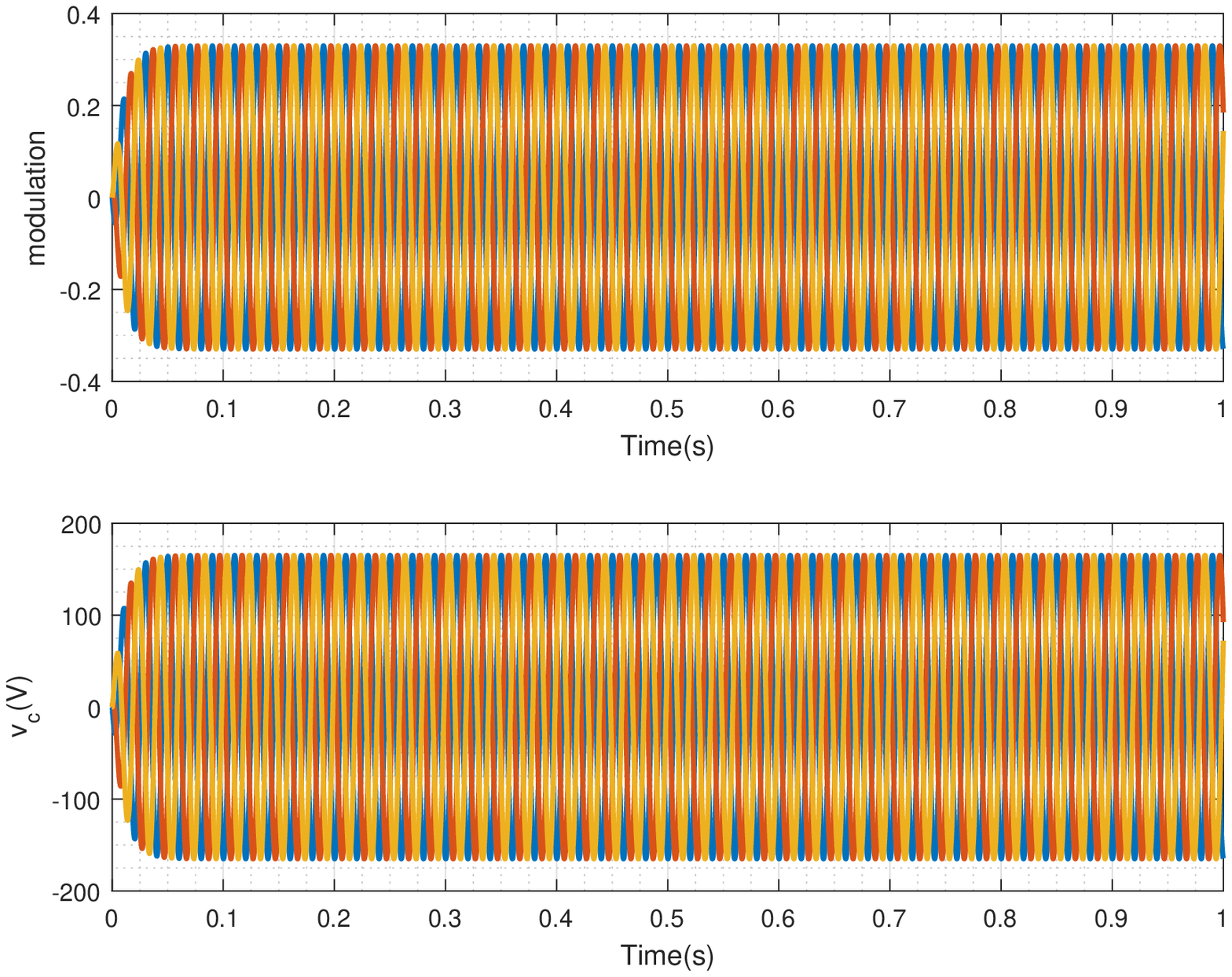}}
		\,
		\subfloat[Evolution of amplitude and frequency of the modulation signal  ]{\includegraphics[scale=0.3]{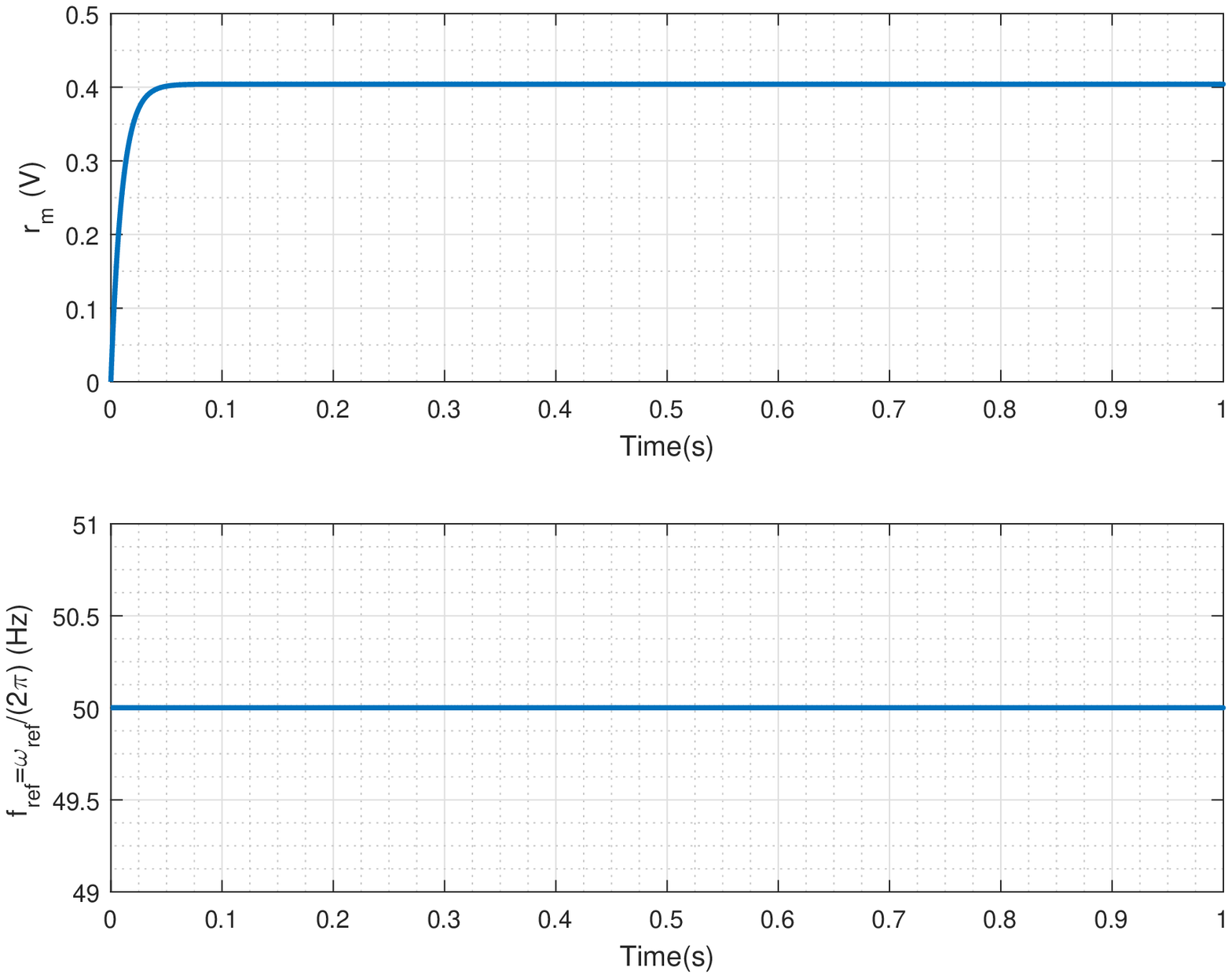}}
		\caption{Simulation results of the virtual oscillator in polar coordinates}
		\label{fig: polar-voc}
	\end{figure}
	
The polar virtual oscillator in open-loop is comparable to a sinwave generator with amplitude $\hat {v}_{m}$ and a frequency $\omega_{ref}$. It has been extended to closed-loop fashion for oscillator-based reference generation by selecting a relevant feedback state corresponsing to the amplitude and frequency of the AC capacitor to accommodate our control objectives.
Next, we review a further outer-loop control design using virtual synchronous machine (VSM) algorithms.

\section{Outer-loop: virtual synchronous machines}

Recent research centralizes in favor of new aggregation and control techniques, where synchronous machines (SM) are gradually replaced by power-electronics based devices, capable of emulating the rotational inertia of SM. These intelligent devices named in the following as virtual synchronous machines (VSM) promise an autonomous operation aiming to ultimately increase the inertia constant for a given power system \cite{VK-SDH-HKZ:11}.

If the goal of the VSM is to emulate the inertia and damping properties of the SM, then these two main aspects can be readily captured by the {\em swing equation} known as
\begin{equation*}
J\dot\omega=T_m-T_e-D\omega
\,,
\end{equation*}
where $J>0$ is the rotor inertia, $\omega\in\real$ the rotating speed of the machine relative to an absolute frequency $\omega_{ref}$, $T_m\in\real$ the mechanical, whereas $T_e$ is the electromagnetic torque, $D>0$ is a damping coefficient accounting for the damping torque associated with the damping windings during transient conditions. 
It can be expressed in terms of power instead by multiplying all terms by the relative frequency $\omega$. For small oscillations around the synchronous conditions, the power balance can be expressed by the following 

\begin{equation}
M\dot\omega=P_m-P_e-D'\omega
\,,
\label{eq: swing-eq}
\end{equation}
with $M=\omega J,\, D'=\omega D$.

The widely-used approach to implement VSM is by providing a reference frequency $\omega$ to inner-loop control. VSM is proved to be equivalent to conventional droop-based methods for standalone and micro-grid operation of converters according to the following scheme in Figure \ref{fig: vsm}, where the block {\em Virtual Inertia and Power Control} implements the swing equation described previously in \eqref{eq: swing-eq}.

\begin{figure}[h!]
	\centering
	\includegraphics[scale=0.3]{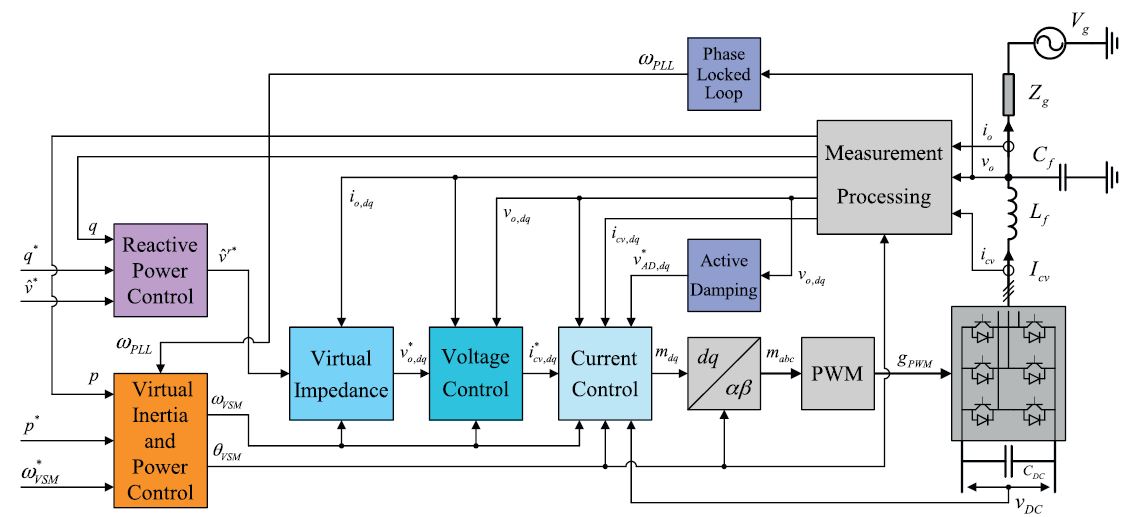}
	\caption{VSM with cascaded voltage and current loops \cite{SDA-SJA:13}}
	\label{fig: vsm}
\end{figure}

Most important topologies rely on interfacing mathematical SM models with power electronic converters arranged with a cascade of controllers which receive reference signals from the VSM and translate them into gate signals for the inverter \cite{SDA-SJA:13}, designed by inner-loop control.
The implementation of the VSM implies an embedded controller computer calculating the references \cite{YC-RH-DT-HPB:11} or a digital signal processor (DSP) associated to a circuit running under special program to control the switches as in the so-called {\em Synchroverters} \cite{QCZ-GW:11}.
An abundant usage of Phase Locked Loops (PLL) ranges from locking the electric power system waveform \cite{VK-SDH-HKZ:11} to generating a reference for the VSM \cite{MW-SH-PV-KV:09} reveals a high dependency on real-time frequency measurements as depicted in an example of realization in Figure \ref{fig: vsm}.

Indeed, frequency estimation and accurate tracking is critical for VSM algorithm and affects its performance \cite{SDA-SJA:13}. Moreover, filter delays and measurement channels often degrade the simulation results \cite{SDA-SJA:13}.

On the other hand, a proper dynamic modeling is a key issue to tackle the inaccuracy in power system simulation by including a full description of its dynamics in transients and at steady state for different operating points. For instance, the swing equation corresponding to the reference model and serving as basis for VSM algorithms, gives only good approximation of frequency transient during the first couple of seconds directly after the power imbalance. Moreover, inner PI loops are non-tractable "blackbox", due to the cascaded control architecture which is difficult to analyze and to keep track of the relevant control signals.
On longer terms, other control actions are needed besides the inertia to determine the frequency response \cite{PT-DVH:16}. New perspective of implementing virtual inertia has been  proposed in \cite{MFMA-EFE:13}, where a super-capacitor connected to the DC-link, responsible for mimicking the mass behavior, outperforms the virtual rotating mass by improving the system stability independently of the disturbance moment \cite{MFMA-EFE:13}. Nevertheless, it remains a conservative way of using the voltage at DC-link.

{\em As a summary}, by reviewing the traditional approach of controlling DC/AC converter, we presented the inner-loop as innermost control hierarchy responsible for tracking a given reference generated from higher-level control loop. Different studies investigate a proper choice of these outer-loops ranging from droop to oscillator-based and virtual synchronous machine methods.
Nonetheless different challenges arise, while implementing these controllers due to merely hard-to-justify assumptions, for instance a quasi-stationary steady state (not valid e.g in the case of a blackout), operation on phasor or large-time delay which may deteriorate system performance. In these settings, inner loops (PI) are  mostly non-tractable and considered as "blackbox" due to the complexity of the analysis of such hierarchical dynamics.

\chapter{Matching control of the synchronous machines}
\label{sec: matching control of SM}
Taking into account all the difficulties imposed by conventional DC/AC control studies targeting to compensate for the retirement of SMs and their ancillary services, we aim to design a controller able to emulate a SM model by making use of the physical storage available in a DC/AC converter.

We introduce an innovative approach that makes use of the natural storage element in the DC circuit of power electronics converter and structurally equivalences a DC/AC converter with a SM by matching the two models. A major difference to the VSM, is that no external referencing is needed at the converter terminal voltage. Moreover, our control strategy does not require additional inner control loops for tracking signals. Instead, it widens the scope of the utility of DC circuit with its natural storage in the regulation mechanism of power systems by including it explicitly and in a more effective way. A proper system modeling involving DC and AC side dynamics is presented for this purpose.
The structural equivalence, due to different physical sizes between a converter and SM, does not influence the performance of the proposed controller and offers more flexibility in tuning it. Since DC measurements are easier to obtain and do not impose additional time delays in comparison to the usual AC measurements required by VSM algorithms, the inverter in closed-loop fashion is advantageous and even more reliable. Another appealing property of the approached controller is that it is well-defined under all operating condition, since it interfaces with recent research field, that of virtual oscillators control (VOC) as in \cite{BBJ-SVD-AOH-PTK:14a,MS-FD-BJ-SD:14b} and this by means of the controller dynamics which encode the inverter terminal dynamics, i.e the dynamics at the output voltage of the converter as a nonlinear limit cycle oscillator adapting to the grid state.

\section{The synchronous machine model}

\label{subsec: emulation control}

The aim of this section is to highlight a particular structure of the SM model which lends itself useful in designing a matching feedback controller for the $3$-phase DC/AC converter. 
We consider a single-pole-pair, non-salient rotor, externally excited SM in $(\alpha\beta)$-frame as in \cite{YS-PT:14}, together with an output AC capacitor at the terminals of the converter, described by the following state space model:
\begin{subequations}
	\begin{align}
	\dot \theta &= \omega
	\\
	M \dot\omega &=-D \omega+\tau_m-\tau_e
	\\
	C \,\dot v_{\alpha\beta}&= -i_{load}+i_{\alpha\beta}
	\\
	\dot\lambda_{\alpha\beta} &=-R i_{\alpha\beta}-v_{\alpha\beta}
	\,.
	\label{eq: fluxes}
	\end{align}%
	\label{eq: SM equations}%
\end{subequations}%
Here $M>0$ and $D>0$ are the rotor inertia and damping, $\tau_m$ is the driving mechanical torque, and $\tau_e$ is the electrical torque. We denote the rotor angle by $\theta\in\mycircle^1$, its angular velocity by $ \omega$, the magnetic flux in the stator winding by $\lambda_{\alpha\beta}$, and the stator resistance by $R>0$. At its terminals the machine is interfaced to the grid through a shunt capacitor with capacitance $C>0$ and capacitor voltage $v_{\alpha\beta}$, and the terminal load current (exciting the machine) is denoted by $i_{load}$. 

\begin{assumption}[Regulated rotor field current]
	The rotor winding described by the magnetic flux $\lambda_f$, is omitted in our equations in \eqref{eq: SM equations} but it is considered with the assumption that the rotor current $i_f$, defined in the following, is externally regulated to a constant value \cite{YS-PT:14}.
\end{assumption}
We define the electromagnetic energy in the machine $W_e$\,as
\[
\label{eq: SM current}
W_e = \begin{bmatrix}\lambda_{\alpha\beta}^\top  & \lambda_f\end{bmatrix}L_{\theta}^{-1} \begin{bmatrix}\lambda_{\alpha\beta} \\ \lambda_f\end{bmatrix}
\,,
\]
where we made use of the inductance matrix $L_{\theta}$
\[
L_{\theta}=\begin{bmatrix}
L_s & 0 & L_m \cos(\theta) \\
0   & L_s & L_m \sin(\theta) \\
L_m \cos(\theta) & L_m \sin(\theta) & L_f
\end{bmatrix}
\label{eq: inductance matrix}
\,,
\]
where $L_m>0$ is the stator-to-rotor mutual inductance, $L_s>0$ the stator inductance and $L_f>0$ the winding field inductance. We obtain the following expressions the inductance current $i_{\alpha\beta}\in\real^{2}$
\begin{equation}
\begin{bmatrix} i_{\alpha\beta} \\ i_f \end{bmatrix} = \begin{bmatrix} \frac{\partial W_e}{\partial \lambda_{\alpha\beta}} \\  \frac{\partial W_e}{\partial\lambda_f} \end{bmatrix} = L_{\theta}^{-1} \begin{bmatrix} \lambda_{\alpha\beta} \\ \lambda_f \end{bmatrix}
\,,
\label{eq: currents}
\end{equation}%
and for the electrical torque $\tau_e$
\begin{equation}
\tau_e = \frac{\partial W_e}{\partial\theta} = -i_{\alpha\beta}^{\top}L_m i_f\begin{bmatrix}
-\sin(\theta)\\ \cos(\theta)
\end{bmatrix}
\,.	
\label{eq: electrical torque}
\end{equation}
By using identity \eqref{eq: currents} in equation \eqref{eq: fluxes}, we express the stator dynamics in terms of current as
\begin{equation}
L_s \dot {(i_{\alpha\beta})}=-R i_{\alpha\beta}-v_{\alpha\beta}-\dot \theta L_m i_f
\begin{bmatrix}
-\sin(\theta)\\ \cos(\theta)
\end{bmatrix} 
\,.
\label{eq: stator current}
\end{equation}
Note that we can identify the electromotive force (EMF) in the machine as the last term in \eqref{eq: stator current}. As a summary, we rewrite \eqref{eq: SM equations} as follows 
\begin{subequations}%
	\label{eq: sync-gen}
	\begin{align}
	\dot \theta &= \omega
	\\
	M \dot\omega &=-D \omega+\tau_m+i_{\alpha\beta}^{\top}L_m i_f\begin{bmatrix}
	-\sin(\theta)\\ \cos(\theta)
	\end{bmatrix} 
	\\
	C \dot v_{\alpha\beta}&= -i_{load}+i_{\alpha\beta}
	\\
	L_s (\dot i_{\alpha\beta})&=-R i_{\alpha\beta}- v_{\alpha\beta} -\omega L_m i_f
	\begin{bmatrix}
	-\sin(\theta)\\ \cos(\theta)
	\end{bmatrix}
	\,. 
	\end{align}%
\end{subequations}%
Observe the similarities between the converter model \eqref{eq: inverter dynamics} and the SM model \eqref{eq: sync-gen}, where the dynamics of DC circuit can be seen as analogous to the rotor mass dynamics.
Notice that, structurally, the electrical torque and electromotive force here play a role similar to $i_x$ and $v_x$ in the converter model \eqref{eq: inverter dynamics}.

\section{The synchronous machine matching control}

\label{sec: matching control}
In this section, we propose a control scheme for the modulation signal $m_{\alpha\beta}\in\real^2$ in \eqref{eq: inverter dynamics}, which matches the closed-loop dynamics of the converter to the dynamics of the SM in \eqref{eq: sync-gen}.

The first step is to introduce the {\em virtual angle} $\theta_{v}$ to resemble to the rotor angle of the SM and assign to it the following dynamics
\begin{equation}
\label{eq: virtual angle}
\boxed{
	\dot \theta_v= \eta\, v_{dc}
}
\;,
\end{equation}
where $\eta>0$ is a constant gain to be specified. For example, a reasonable choice would be the ratio between the nominal AC frequency and the DC voltage reference since this choice induces the correct oscillation in the electrical domain.

The second step in control design is to assign a sinusoidal modulation scheme according to the following map $m_{\alpha\beta}: \mycircle^1\to\mycircle^1_\mu\left\{x\in\real^2: \norm{x}_{2}=\mu\right\}$, such that

\begin{equation}
\boxed{
	m_{\alpha\beta}=\mu \begin{bmatrix}
	-\sin(\theta_v)\\ \cos(\theta_v)
	\end{bmatrix}}
\;,
\label{eq:modulation signal}
\end{equation}
where $\theta_v\in\mycircle^1$ is an angle to be specified with the frequency $\dot\theta$ as determined in \eqref{eq: virtual angle}, while the gain $\mu\in]0,1]$ is constant and represents an amplitude for the modulation sinusoid.

By using \eqref{eq:modulation signal}, we can now write $i_x$ and $v_x$ as:
\begin{equation}
\label{eq: closed loop i_x and v_x}
i_{x} = \frac{1}{2} i_{\alpha\beta}^\top\,\mu\begin{bmatrix} -\sin(\theta_v)\\ \cos(\theta_v) \end{bmatrix},\,
v_{x} =\frac{1}{2} v_{dc}\,\mu\begin{bmatrix} -\sin(\theta_v)\\ \cos(\theta_v) \end{bmatrix}
\,.
\end{equation}

We now complete the comparison between the generator model \eqref{eq: sync-gen} and the closed-loop converter model \eqref{eq: inverter dynamics} under the control scheme \eqref{eq:modulation signal}, \eqref{eq: virtual angle}. For this purpose, we identify the average switch voltage $v_x$ with a {\em virtual electromotive force} by defining the following relation 
\begin{equation}
\mu=-2\eta L_m i_f
\,.
\label{eq: mu}
\end{equation}
By means of \eqref{eq: electrical torque}, \eqref{eq: closed loop i_x and v_x} and \eqref{eq: mu}, we identify the DC-side average switching current $i_x$ with a {\em virtual electrical torque} by defining:
\begin{equation}
\tau_{e,v}=\frac{1}{\eta}\cdot \frac{1}{2}i_{\alpha\beta}^{\top}\,\mu\, \begin{bmatrix}
-\sin(\theta_v)\\ \cos(\theta_v)
\end{bmatrix}=\frac{1}{\eta} i_x
\,.
\label{eq: vir-elec-torque}%
\end{equation}%
Next, we denote the {\em virtual angular velocity} by $\omega_v=\eta\,v_{dc}$ and rewrite the equivalent closed-loop model for the DC/AC converter after dividing by $\eta^2$ to relate $\tau_{e,v}$, as in \eqref{eq: vir-elec-torque}
\begin{subequations}
	\begin{align}
	\dot \theta_v & =\omega_v 
	\\
	\frac{C_{dc}}{\eta^2}\dot{\omega}_{v} &= -\frac{G_{dc}}{\eta^2} \omega_v +  i_{dc}/{\eta} - \frac{1}{\eta} i_x
	\\
	L \dot {(i_{\alpha\beta})} &= -R i_{\alpha\beta} - v_{\alpha\beta}+ \frac{1}{2\eta}\omega_{v}m_{\alpha\beta}
	\\
	C \dot v_{\alpha\beta} &= -G_{g}v_{\alpha\beta}+i_{\alpha\beta}
	\,.
	\end{align}%
	\label{eq: equivalent machine model}%
\end{subequations}%
By attributing proper units to $\eta$, we can now identify
$C_{dc}/\eta^2,\, G_{dc}/\eta^2$, and $i_{dc}/\eta$ respectively with the mechanical
inertia typically $3$ orders of magnitude (in p.u.) less than the inertia of a SM, a significant gain in equivalent damping factor, and driving torque of an equivalent SM.

\subsubsection{Matching vs. virtual emulation}
Observe that the structural equivalence of the closed-loop dynamics \eqref{eq: equivalent machine model} to those of a SM is of purely physical nature as opposed to virtual, such as
in the works of \cite{HB-TI-YM:14,SDA-SJA:13,VK-SDH-HKZ:11}. 
In recent works, the behavior of the SM is emulated in software, i.e with virtual storage elements which further provide set points for the inner-loop of the converter control, for which a time-scale separation is assumed. In comparison, we use the physical storage already present in the DC capacitor, which is reflected in the size of the equivalent inertia $\frac{C_{dc}}{\eta^2}$ and equivalent damping factor $\frac{G_{dc}}{\eta^2}$, typically 3. order of magnitude less than the inertia of a SM.

\subsubsection{Virtual adaptive oscillator interpretation}
	\label{remark: interpretation of closed loop}
	By defining $\xi\in\real^2$ as a controller state and $m_{\alpha\beta}$ as an output, we can rewrite the controller \eqref{eq:modulation signal} and \eqref{eq: virtual angle} as the nonlinear dynamic feedback oscillator
	\begin{equation}
	\dot \xi  = \eta\, v_{dc} \begin{bmatrix} 0 & -1 \\1 & 0 \end{bmatrix} \xi
	\,, m_{\alpha\beta}=\mu\,\xi
	\,,
	\label{eq: controller dynamics}
	\end{equation}
	where $\norm{\xi(0)}_2 = 1$. As depicted in Figure \ref{fig: closed-loop-system}, we can interpret the emulation control \eqref{eq:modulation signal},\eqref{eq: virtual angle} as an oscillator with constant amplitude $\norm{m(0)}_2=\mu$ and state-dependent frequency $\omega_v=\eta\,v_{dc}$ in feedback with the DC/AC converter dynamics \eqref{eq: inverter dynamics}. This control strategy structurally resembles the classic {\em proportional resonant control} \cite{RT-FB-LM-LPC:06} with the difference that the frequency of the oscillator \eqref{eq: controller dynamics} actually adapts to the DC voltage which again reflects the grid state. 
	
	\oprocend
	\begin{figure}[h!]
		\centering
		\resizebox{9cm}{5cm}{
		\begin{tikzpicture}[auto, node distance=2cm,>=latex']
		\node [input, name=input] {};
		\node [sum, right of=input] (sum) {};
		\node [input, above of=sum, node distance=0.5cm] (current) {current};
		\node [block, right of=sum, node distance=4.7cm] (system) 
		{$\begin{matrix}
			C_{dc}\dot v_{dc} =-G_{dc}v_{dc} + i_{dc}-\frac{1}{2}m_{\alpha\beta}^{\top}i_{\alpha\beta}\\
			L \dot i_{\alpha\beta} =-R i_{\alpha\beta} + \frac{1}{2}m_{\alpha\beta} v_{dc}-v_{\alpha\beta} \\
			C \dot v_{\alpha\beta} =-i_{load}+i_{\alpha\beta}
			\end{matrix}$};
		\node [output, right of=system, node distance=4.7cm] (output) {};
		\node [block, below of=system, node distance=4cm] (controller) {
			$\begin{matrix}
			\dot \xi=\eta\, v_{dc}\begin{bmatrix}
			0 & 1 \\ 
			-1 & 0
			\end{bmatrix}\xi\\
			\end{matrix}$
		};
		\node [smallblock, left of=controller, node distance=3cm] (gain2) {$\mu$};
		\node [smallblock, right of=controller, node distance=3cm] (gain1) {$\eta$};
		\draw [draw,->] (current) -- node {$(i_{dc}, -i_{load}\,)$} ($(system.west) - (0cm, -0.5cm)$);
		\draw [draw,->] ($(system.east) - (0cm, -0.5cm)$) -- node {$\,(v_{dc}, v_{\alpha\beta})$} ($(output.west) - (0cm, -0.5cm)$) ;
		\draw [->] (sum) -- node {$m_{\alpha\beta}$} (system);
		\draw [-] (system) -- node [name=y] {$v_{dc}$}(output);
		\draw [->] (output) |- (gain1);
		\draw [->] (gain1.west) |- (controller);
		\draw [->] (controller.west) |- (gain2);
		\draw [->] (gain2.west) -| node[pos=0.99] {$-$} 
		node [near end] {} (sum);
		\end{tikzpicture}}
		\caption{Closed-loop system comprising the converter dynamics \eqref{eq: inverter dynamics} and the controller dynamics \eqref{eq: controller dynamics}.}
		\label{fig: closed-loop-system}
	\end{figure}
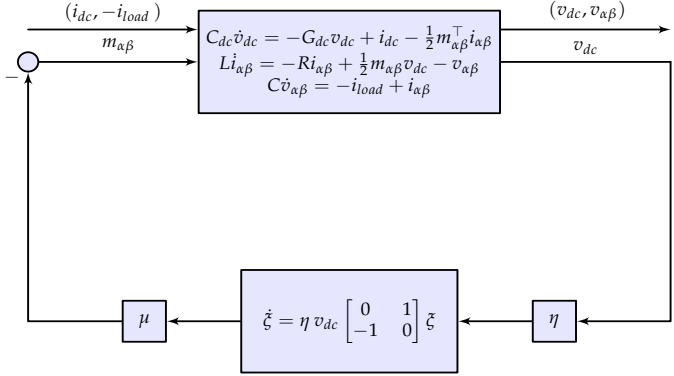

\begin{remark}[Dynamics of $\gamma$- component ]
	We now consider the $\gamma$-component of the converter AC signals.
	By construction of the modulation vector, we have $v_{x,\gamma}=0$. If we assume a balanced load, then $i_{load,\gamma}=0$. We are left with the following asymptotically stable dynamics for the $\gamma$-component:
	\begin{subequations}
		\begin{align}
		L \dot {(i_{\gamma})} &= -R i_{\gamma} - v_{\gamma}
		\\
		C \dot v_{\gamma} &= i_{\gamma}
		\,.
		\end{align}	%
		\label{eq: gamma component}%
	\end{subequations}%
	
	Since \eqref{eq: gamma component} is an asymptotically stable dynamical system, the omission of the $\gamma$-component is well-justified.
\end{remark}

{\em As a summary} we proposed a novel control strategy for grid-forming converters in low-inertia power grids. Our strategy was inspired by the identification of the structural similarities between the three-phase DC/AC converter and the synchronous machine model, mainly between the DC-cap and the rotor dynamics. We explicitly matched these models through matching control, so that they became structurally equivalent. Compared to standard emulation of virtual synchronous machines (VSM), our controller relies solely on readily available DC-side measurements and takes into account the natural DC and AC storage elements, which are usually neglected. As a result our controller is generally faster and less vulnerable to delays and measurement inaccuracies. We provided a virtual adaptive oscillator interpretation of our controller.

We will next present various plug-and-play properties of the DC/AC converter in closed-loop fashion, which we illustrate in the next section. 


\chapter{Plug and play properties of the matching controller} 
\label{sec: properties}
For large-scale power network applications, key requirements are plug-and-play properties that the DC/AC converter should  possess independently on the number and type of the devices connected to the grid. A typical decentralized stability and robustness certificate is passivity \cite{FZO+13}, and a typical control requirement for grid-forming units is droop behavior \cite{FD-JWSP-FB:14a} trading-off power injection with the voltage amplitude and frequency. In the following, we investigate, using the proposed matching controller, plug-and-play properties for the closed-loop system \eqref{eq: inverter dynamics},  \eqref{eq:modulation signal}, and \eqref{eq: virtual angle}.

\section{Voltage terminal dynamics}
In view of studying the droop behavior of the voltage at the terminals of the modulation block $v_x$, we dedicate this section to derive equivalent circuit dynamics induced by the matching controller in $(\alpha\beta)$- coordinates.

\begin{proposition}[Dynamics of the AC voltage at the output of the modulation block]
The dynamics of the AC voltage $v_x$, at the output of the modulation block can be expressed in $(\alpha\beta)$- domain as

\begin{equation}
{\frac{4C_{dc}}{\mu^2}}\, \dot v_{x}= \frac{2i_{dc}^{}}{\mu}\frac{v_{x}}{\sqrt{v_{x}^{\top}v_{x}}} -\Upsilon \, i_{\alpha\beta}+\Big({-\frac{4G_{dc}}{\mu^2} I_2 + \frac{8C_{dc}\eta}{\mu_{}^3 }\sqrt{v_x^{\top}v_x}J_2}\Big) v_{x}
\,,
\label{eq: voltage terminal dynamic}
\end{equation}
where $\omega_x=\dot \theta_x,\, J_2\in\mbb R^{2\times 2}$ the rotation matrix with angle $\frac{\pi}{2}$, $I_2$ is the identity matrix in $\mbb R^2$ and the projection matrix $\Upsilon\in\mbb R^{2\times 2}$, defined as

\begin{equation*}
\Upsilon=\frac{v_{x}\,v_{x}^{\top}}{v_{x}^{\top}\,v_{x}}
\,.
\label{eq: projection matrix}
\end{equation*}
\end{proposition}

\begin{proof}
We express the terminal voltage $v_x$ in polar coordinates by
\[
v_x=\frac{1}{2}m_{\alpha\beta} v_{dc}=\frac{1}{2}\mu v_{dc}\begin{bmatrix}
-\sin(\theta_x) \\ \cos(\theta_x)
\end{bmatrix}=\hat {v}_x \begin{bmatrix}
-\sin(\theta_x) \\ \cos(\theta_x)
\end{bmatrix}
\,,
\]
where $\hat {v}_x$ the amplitude and $\theta_x$ the phase angle of $v_x$ are defined as follows

\begin{subequations}
\begin{align*}
\hat {v}_x &=\frac{1}{2} \mu v_{dc}
\\
\,
\theta_x &=\theta_v
\,,
\end{align*}
\end{subequations}   
with the corresponding dynamics  

\begin{subequations}
\begin{align}
\label{eq: polar r_x}
\dot {\hat v}_x & = \frac{1}{2} \mu \dot v_{dc} 
= \frac{\mu}{2C_{dc}}(i_{dc}-\frac{m_{\alpha\beta}^{\top}}{2}i_{\alpha\beta})-\frac{G_{dc}}{C_{dc}} \hat {v}_x
\\
\,
\dot \theta_x &=\omega_x = \frac{2\eta}{\mu}r_x
\,,
\label{eq: omega to r}
\end{align}
\end{subequations}   

where we apply $v_{dc}= 2\hat{v}_x/\mu$.

We can rewrite the dynamics of the voltage $v_x$ as 

\begin{equation*}
\dot v_x = \frac{d}{dt}\hat {v}_x \begin{bmatrix}
-\sin(\theta_x) \\ \cos(\theta_x)
\end{bmatrix} = \dot {\hat v}_x \begin{bmatrix}
-\sin(\theta_x) \\ \cos(\theta_x)
\end{bmatrix}+ \begin{bmatrix}
-\cos(\theta_x) \\ -\sin(\theta_x)
\end{bmatrix} \dot \theta_x \hat {v}_x 
\,.
\end{equation*}

We substitute $\dot {\hat v}_x$ using \eqref{eq: polar r_x} and we can write:
\begin{subequations}
\label{eq: polar dymanics}
\begin{align*}
\dot v_x & = \bigg(\frac{\mu}{2C_{dc}}(i_{dc}^{*}-\frac{m_{\alpha\beta}^{\top}}{2}i_{abc})-\frac{G_{dc}}{C_{dc}} r_x\bigg)\begin{bmatrix}
-\sin(\theta_x) \\ \cos(\theta_x)
\end{bmatrix} +  \begin{bmatrix}
-\cos(\theta_x) \\ -\sin(\theta_x)
\end{bmatrix}\dot\theta_x \hat {v}_x 
\\
\,
& = \frac{1}{2C_{dc}}\bigg(i_{dc}^{*}-\frac{m_{\alpha\beta}^{\top}}{2}i_{\alpha\beta}\bigg)m_{\alpha\beta}-\frac{G_{dc}}{C_{dc}} v_{x}+  \begin{bmatrix}
-\cos(\theta_x) \\ -\sin(\theta_x)
\end{bmatrix}\dot\theta_x \hat {v}_x 
\\
\,
& = \frac{ m_{\alpha\beta}^{\top}m_{\alpha\beta}}{4C_{dc}} \bigg(\frac{2m_{\alpha\beta}i_{dc}}{m_{\alpha\beta}^{\top}m_{\alpha\beta}}-\frac{m_{\alpha\beta}m_{\alpha\beta}^{\top}}{m_{\alpha\beta}^{\top}m_{\alpha\beta}}i_{\alpha\beta}\bigg)-\frac{G_{dc}}{C_{dc}} v_x + \begin{bmatrix}
-\cos(\theta_x) \\ -\sin(\theta_x)
\end{bmatrix}\dot\theta_x \hat {v}_x 
\,,    
\end{align*}
\end{subequations}
where 
\begin{subequations}
\begin{align*}
m_{\alpha\beta}^{\top}m_{\alpha\beta}& =\mu^2\\
\frac{m_{\alpha\beta}m_{\alpha\beta}^{\top}}{m_{\alpha\beta}^{\top}m_{\alpha\beta}} & =\frac{v_x v_x^{\top}}{v_x^{\top}v_x}
\,,
\end{align*}
\end{subequations}
and it yields the following
\begin{equation*}
\frac{4C_{dc}}{\mu^2}\,\dot v_x=\frac{2i_{dc}}{\mu}\begin{bmatrix}
-\sin(\theta_x) \\ \cos(\theta_x)
\end{bmatrix}-\frac{v_x\,v_x^{\top}}{v_x^{\top}v_x}i_{\alpha\beta}-\frac{4G_{dc}}{\mu^2} v_x + \frac{4C_{dc}}{\mu^2 }\begin{bmatrix}
-\cos(\theta_x) \\ -\sin(\theta_x)
\end{bmatrix}\dot\theta_x \hat {v}_x 
\,.
\end{equation*}

Due to $(\alpha\beta)$- transformation, we use the fact that 
\[
\begin{bmatrix}
-\cos(\theta_x) \\ -\sin(\theta_x)
\end{bmatrix} = J_2 \begin{bmatrix}
-\sin(\theta_x) \\ \cos(\theta_x)
\end{bmatrix}
\,,
\]
where $J_2\in\mbb R^{2\times 2}$ is the rotation matrix with angle $\frac{\pi}{2}$. We deduce the following

\begin{equation}
\frac{4C_{dc}}{\mu^2}\, \dot v_{x}= {\frac{2 i_{dc}}{\mu}}\begin{bmatrix}
-\sin(\theta_x) \\ \cos(\theta_x)
\end{bmatrix}
-\frac{v_{x}\,v_{x}^{\top}}{v_{x}^{\top}\,v_{x}}i_{\alpha\beta}+\Big({-\frac{4G_{dc}}{\mu_{}^2} I_2 + \frac{4C_{dc}}{\mu_{}^2 }\omega_x J}_{}\Big) v_{x}
\,.
\end{equation}
Additionally, we use the fact that: 
\begin{align}
\omega_x=\frac{2\eta}{\mu}\hat {v}_x=\frac{2\eta}{\mu}\sqrt{v_x^{\top}v_x}
\,,
\end{align}
and that: 
\begin{equation}
\begin{bmatrix}
-\sin(\theta_x) \\ \cos(\theta_x)
\end{bmatrix}=\frac{v_x}{\hat {v}_x}=\frac{v_x}{\sqrt{v_x^{\top}v_x}}
\,.
\end{equation}
It yields for the terminal voltage $v_x$ the dynamics described in \eqref{eq: voltage terminal dynamic}
\end{proof}

\subsubsection{Interpretation of $v_x$ dynamics}

\begin{equation}
\underbrace{\frac{4C_{dc}}{\mu^2}}_{C'}\, \dot v_{x}= \underbrace{\frac{2i_{dc}}{\mu}\frac{v_{x}}{\sqrt{v_{x}^{\top}v_{x}}}}_{i^{'}}-
\underbrace{\frac{v_{x}\,v_{x}^{\top}}{v_{x}^{\top}\,v_{x}}i_{\alpha\beta}}_{i_{l}^{'}}+\Big(\underbrace{-\frac{4G_{dc}}{\mu^2} I_2 + \frac{8C_{dc}\eta}{\mu^3 }\sqrt{v_x^{\top}v_x}J_2}_{G'}\Big) v_{x}
\,.
\label{eq: new DC dynamic}
\end{equation}
The dynamics of the voltage $v_x$ can be interpreted as equivalent AC circuit dynamics resulting from merging the modulation block into the DC circuit, where $i_{L}^{'}$ is the projected current, after applying the projection matrix $\Upsilon\in\mbb R^{2\times2}$ 
\begin{equation}
\Upsilon=\frac{v_{x}\,v_{x}^{\top}}{v_{x}^{\top}\,v_{x}}
\,.
\end{equation}
It has an equivalent capacitor $C^{'}$, an equivalent AC current source $i^{'}$ and an equivalent conductance matrix $G^{'}$. It has nonlinear AC- time variant parameters, i.e. in dependency of $v_{dc}$ representing the DC capacitor voltage and of $v_{x}$ representing the voltage at the output of the modulation block.

Furthermore, the dynamics of the voltage $v_x$ can be interpreted as an equation relating $i_{\alpha\beta}$ as input coming from the AC circuit to the output $v_x$ of the equivalent circuit represented by two blocks: a DC circuit and a modulation block as depicted in Figure \ref{fig: equivalent-circuit}. 

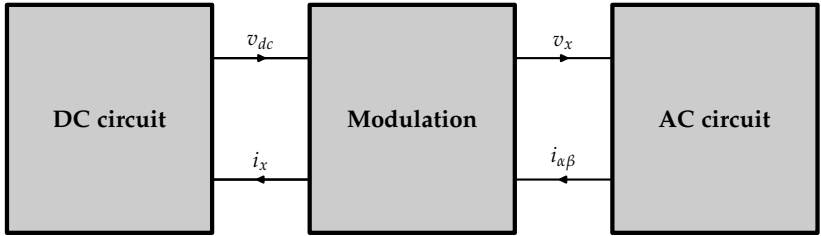
\begin{figure}[h!]
\begin{center}
{\begin{tikzpicture}
\coordinate (SolarCellP) at (0.0,0);
\coordinate (ChopperP) at (4,0);
\coordinate (InverterP) at (8,0);
\coordinate (Equivalent) at (2,0);
\begin{scope}[every node/.style={draw,ultra thick, fill=black!20, rectangle,text width=2.5cm, minimum height=3cm, text badly centered}]
\node[name=pv] at (SolarCellP) {\bfseries DC circuit};
\node[name=chopper] at (ChopperP) {\bfseries Modulation};
\node[name=inverter] at (InverterP) {\bfseries AC circuit};
\end{scope}

\draw
(pv.30) -- (chopper.150)
(pv.330) -- (chopper.210);

\draw (pv.30) to [short, i>=$v_{dc}$] (chopper.150);
\draw (pv.330) to [short, i<=$i_{x}$] (chopper.210);
\draw (chopper.30) to [short, i>=$v_x$] (inverter.150);
\draw (chopper.330) to [short, i<=$i_{\alpha\beta}$] (inverter.210);

\end{tikzpicture}}
\end{center}
\caption{Interpretation of $v_x$ dynamics}
\label{fig: equivalent-circuit}
\end{figure}

\section{Analysis of the power flow}
\subsection{Power injection at the output of the modulation block}
We now provide an investigation of the proposed controller by studying the steady-state power flow at the output of the modulation block.
Since we are interested in the steady state power injection, we introduce the following definition \cite{CA-DG-FD:16}.

\begin{definition}[Steady-state of DC/ AC signal]
	\label{definition: non-trivial-steady-state}
	An AC signal $z_{\alpha\beta}(t)\in\real^2$ is said to be in (synchronous and balanced) steady state, when it satisfies the following differential equation for some  nonzero synchronous frequency $\omega_s\in\real$:
	\[
	\dot z_{\alpha\beta} = \omega_s \begin{bmatrix}0 & -1 \\ 1 & 0\end{bmatrix} z_{\alpha\beta}
	\,.
	\]
	Similarly, a DC signal $z_{dc}(t)\in\real$ is said to be in steady-state when it satisfies the differential equation $\dot z_{dc} = 0$.
	\label{def: ss}
\end{definition}
In the remainder of this section we prove that there exists a relationship between active power at the output of the modulation block and the voltage amplitude and frequency at steady state of the voltage terminal $v_x$.

\begin{assumption}[Feasibility] In the remainder, we assume that a non-trivial steady state exists for all DC and AC signals.
\end{assumption}

Based on the Definition \ref{def: ss}, we state the following theorem
\\
\begin{theorem}[Active and reactive power]
	Consider the converter model \eqref{eq: inverter dynamics} together with the controller \eqref{eq:modulation signal}, \eqref{eq: virtual angle}. Assume that all DC and AC signals are in  steady state as described in Definition~\ref{definition: non-trivial-steady-state} with synchronous frequency $\omega_s$. Let $P_{x}$ denote the active power flowing out of the average switching voltage node $v_x$ and let $\hat {v}_{x}$ and $\omega_{x}$ be its amplitude and frequency, then the following holds:
	\begin{subequations}
		\begin{align}
		\hat {v}_x &=\frac{\mu}{4G_{dc}}(i_{dc}+\sqrt{i_{dc}^{2}-4G_{dc}P_x})
		\\
		\omega_x &=\frac{\eta}{2G_{dc}}(i_{dc}+\sqrt{i_{dc}^{2}-4G_{dc}P_x})
		\,,
		\end{align}
		\label{eq: amplitude and frequency}%
	\end{subequations}%
	with $\omega_x=\omega_v=\omega_s$ by assumption.
	Moreover, there is no influence of reactive power $Q_{x}$ on the amplitude and frequency $(r_{x},\omega_{x})$.
	\label{thm: anal-sol}
\end{theorem}

\begin{proof} 
	Consider the dynamics of the DC circuit as described in \eqref{eq: DC voltage} at steady state, i.e, when $\dot v_{dc}=0$
	\begin{subequations}
		\begin{align}
		0 &=-G_{dc}v_{dc}+i_{dc}-i_x 
		\\
		0 &=-G_{dc}v_{dc}^2+i_{dc}v_{dc}-i_xv_{dc} 
		\,,
		\label{eq: steady state vdc}
		\end{align}
	\end{subequations}
	
	where we multiply by $v_{dc}$ the second equation.
	The active power at the output of the modulation block is given by 
	\[
	P_x=v_x^{\top}i_{\alpha\beta}=\frac{1}{2}m_{\alpha\beta}^{\top}v_{dc}i_{\alpha\beta}=i_x v_{dc}
	\,.
	\]
	We multiply \eqref{eq: steady state vdc} by $v_{dc}$ to obtain the quadratic expression relating $P_x$ and $v_{dc}$ at steady state.
	\begin{equation*}
	v_{dc}=\frac{i_{dc}+\sqrt{i_{dc}^{2}-4G_{dc}P_x}}{2 G_{dc}}
	\,.
	\label{eq: v_dc at ss}
	\end{equation*}
	
	Note that the amplitude $\hat {v}_x$ and frequency $\omega_x$ at the output of the modulation block can be expressed as: 
	\begin{equation*}
	\hat{v}_x =\frac{1}{2} \mu\, v_{dc},\;
	\omega_x=\eta\, v_{dc}
	\,.
	\label{eq: polar v_x }%
	\end{equation*}%
	\eqref{eq: amplitude and frequency} follows directly from \eqref{eq: v_dc at ss} and \eqref{eq: polar v_x }.
\end{proof} 

Equations \eqref{eq: amplitude and frequency} relate the active power $P_x$ flowing out at the output of the modulation block and the corresponding the amplitude $\hat {v}_x$ and frequency $\omega_x$ of the voltage $v_x$ at steady state.

Only active power can influence the amplitude and frequency at the output of the modulation block. These results can be justified by the implemented control behavior which basically takes as input the DC capacitor voltage as its major element. DC circuit can only be affected by active power as described in equation \eqref{eq: v_dc at ss} therefore there is no influence of reactive power on the voltage at the output of the modulation block.

In the following, we characterize the voltage and frequency droop slopes induced by our emulation controller \eqref{eq:modulation signal}, \eqref{eq: virtual angle} at a particular steady state of the switching voltage $v_x$ written in terms of its amplitude $\hat {v}_{x}$ and frequency $\omega_x$. Here, the droop slopes $d_{\hat {v}_{x}}$ and $d_{\omega_{x}}$ describe the locally linear sensitivity relating the active power injection $P_{x}$  and its steady-state voltage amplitude $\hat{v}_x$ and frequency $\omega_x$.

\begin{corollary}[Droop coefficients] Around a steady-state operating point $(\hat {v}_x,\omega_x)$, the following active power droop coefficients are identified
	\begin{equation*}
	d_{\hat {v}_x}=-\frac{8G_{dc}}{\mu^2}\hat {v}_x+\frac{2i_{dc}}{\mu},\;
	d_{\omega_x}=-\frac{2G_{dc}}{\eta^2}\omega_x+\frac{i_{dc}}{\eta}
	\,.
	\label{eq: droop slopes}
	\end{equation*}
\end{corollary}

\begin{proof}
	From equations \eqref{eq: amplitude and frequency}, the active power $P_x$ expression at steady state can be given as a function of $\hat {v}_x$ and $\omega_x$
	\begin{equation*}
	P_{x} = \frac{-4G_{dc}}{\mu^{2}} \hat {v}_x^{2}+ \frac{2i_{dc}}{\mu}\hat {v}_x
	= \frac{-G_{dc}}{\eta^{2}} \omega^{2}_x+ \frac{i_{dc}}{\eta} \omega_x
	\,.
	\label{eq: power equations}
	\end{equation*}
	
	By linearizing equation \eqref{eq: power equations} around the operating point $(\hat {v}_x, \omega_x)$, we find the droop slopes in \eqref{eq: droop slopes} from the following equations
	\begin{subequations}
		\begin{align*}
		\left.\frac{\partial P_{x}}{\partial \hat {v}_x}\right|_{\hat {v}_x} &= -\frac{8G_{dc}}{\mu^2} \hat {v}_x+\frac{2i_{dc}}{\mu}
		\\
		\left.\frac{\partial P_{x}} {\partial \omega_x}\right|_{ \omega_x} &=-\frac{2G_{dc}}{\eta^2} \omega_x+\frac{i_{dc}}{\eta}
		\,.
		\end{align*}
	\end{subequations}
\end{proof}

No influence of the reactive power on the amplitude and frequency is captured at steady state associated to zero droop slopes.

\begin{corollary}[Maximal active power]
	The maximal active power $\bar P_x$  that can be delivered at the output of the modulation block is 
	\begin{equation*}
	\bar P_x =\frac{i_{dc}^{2}}{4G_{dc}}
	\,.
	\label{eq: max active power}
	\end{equation*}
\end{corollary}

\begin{proof}
	The maximum $\bar P_x $ of the parabolic equation \eqref{eq: power equations} describing $P_{x}$ is attained when ${dP_x}/{d\hat {v}_x}=0$ or ${dP_x}/{d\omega_x}=0$. Equivalently, we obtain
	\[	
	\bar P_x =\frac{-4G_{dc}}{\mu^{2}} \bar v_x^{2}+ \frac{2i_{dc}}{\mu} \bar v_x 
	=\frac{i_{dc}^{2}}{4G_{dc}}
	\,,
	\]
	where $\bar v_x $ is the maximal amplitude and $\bar \omega_x$ is the maximal frequency defined by 
	\begin{subequations}
		\begin{align}
		\bar v_x &=\frac{\mu i_{dc}}{4G_{dc}}
		\\
		\bar \omega_x &=  \frac{\eta i_{dc}}{2G_{dc}}
		\label{eq: max-amp-freq}
		\end{align}
	\end{subequations}
\end{proof}	
The maximal deliverable AC active power is naturally constrained by the maximal DC power in accordance with the maximum power transfer theorem \cite{WH-JK-SD:12} stating that, to obtain maximum external power from a source with a finite internal resistance, the resistance of the load must equal the resistance of the source. 
\\
In fact, at steady state the DC/AC converter can be interpreted by its resistive load comprising the DC conductance $G_{dc}$ and the equivalent load conductance $\tilde G_{load}$ as depicted in Figure \ref{fig: maximum power} resulting in the maximal power load described in \eqref{eq: max active power}.
\begin{figure}[h!]
	\centering
	\begin{center}
		{\begin{circuitikz}[american voltages]
				\draw
				(0,0) to[american current source, i>=$i_{dc}$](0,2) 
				(1.3,2) to[R,l=$G_{dc}$] (1.3,0) 
				(2.5,2) to[R,l=$\tilde G_{load}$] (2.5,0) 
				(0,0) -- (2.5,0)
				(2.5,2) -- (0,2);
			\end{circuitikz}}
		\end{center}
		\caption{DC circuit with its DC resistive                                                                 conductance $G_{dc}$ and the equivalent load conductance $\tilde G_{load}$}
		\label{fig: maximum power}
	\end{figure}
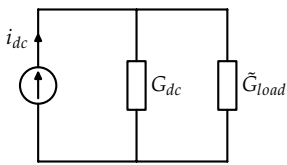

\subsubsection{Simulation results} 
		
For the following case study, we consider the  converter introduced previously yielding {\em nominal} values  \footnote{We refer DC and AC quantities as {\em nominal} when they are in the steady-state induced by an open-circuit operation with $i_{load} = 0A$.}, $\omega_{ref}=2 \pi 50\,rad/s,\,  \hat {v}_{ref}= 165V$ and nominal DC voltage of $v_{dc, ref} = 1000V$. In order to obtain the desired nominal values $\hat v_{ref}$ and $\omega_{ref}$, we choose the controller gains as: 
	\begin{equation*}
	\label{eq:modulation equations}
	\eta = \frac{\omega_{ref}}{v_{dc, ref}}=0.3142\,rad/sV,\, 
	\mu =\frac{2 \hat {v}_{ref}}{v_{dc, ref}}=0.33
	\,.
	\end{equation*}
	\begin{figure}[h!]
			\centering
			\includegraphics[scale=0.3]{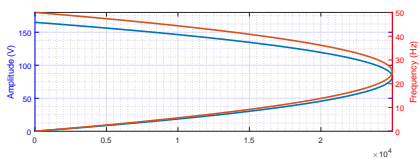}
			\caption{Analytical solutions at the output of the modulation block}
			\label{fig: analytical sol}	
	\end{figure}
		
Figure \ref{fig: analytical sol} represents the analytical curves found in \eqref{eq: amplitude and frequency}. The analytical solutions describing active power in dependency of the amplitude and frequency match the experimental results as shown in Figure \ref{fig: Power in simplified model}.
We then run simulations according to a time varying and balanced load profile acting on the converter described by step changes starting from $t_s=0.3 s$ in either load conductance  or susceptance (Figures \ref{fig: resistive load}, \ref{fig: reactive load}). We notice that reactive power has no effect on steady state response at the at the output of the modulation block.
A magnified version of Figure~\ref{fig: analytical sol}, near the nominal, is shown in Figure \ref{fig: Power in simplified model}, where we overlaid the analytic curves \eqref{eq: amplitude and frequency} with values from numerical experiments for initial condition $v_{dc}(0)=0$. Observe the nearly linear droop characteristics at the operating points.
		\begin{figure}[h!]
			\centering
			\subfloat[Reactive power at the output of the modulation block]{\includegraphics[scale=0.3]{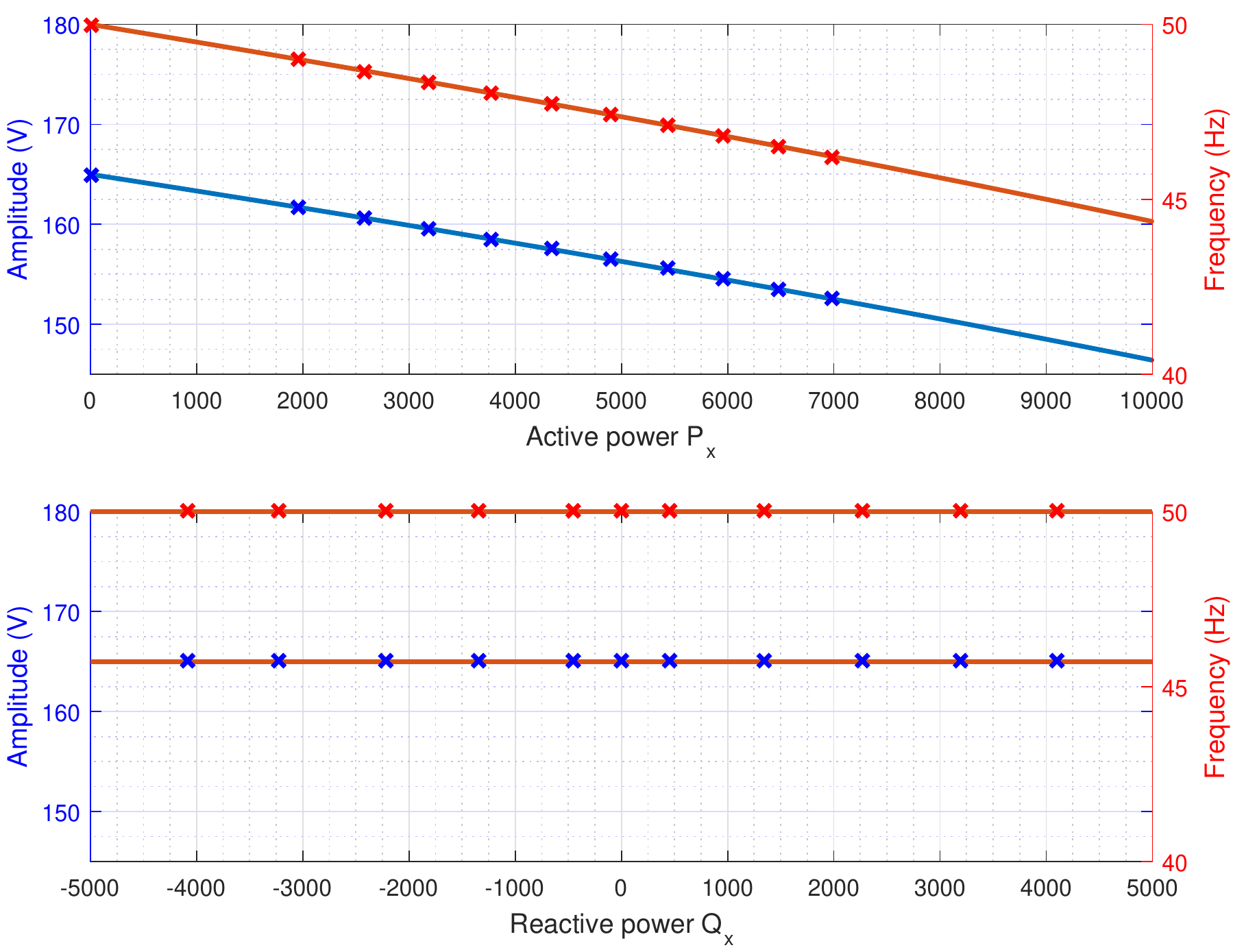}}
			\caption{Power at the output of the modulation block}
			\label{fig: Power in simplified model}	
		\end{figure}
		
		\begin{figure}[h!]
			\centering
			\subfloat[Resistive load profile and voltage $v_x$]{\includegraphics[scale=0.3]{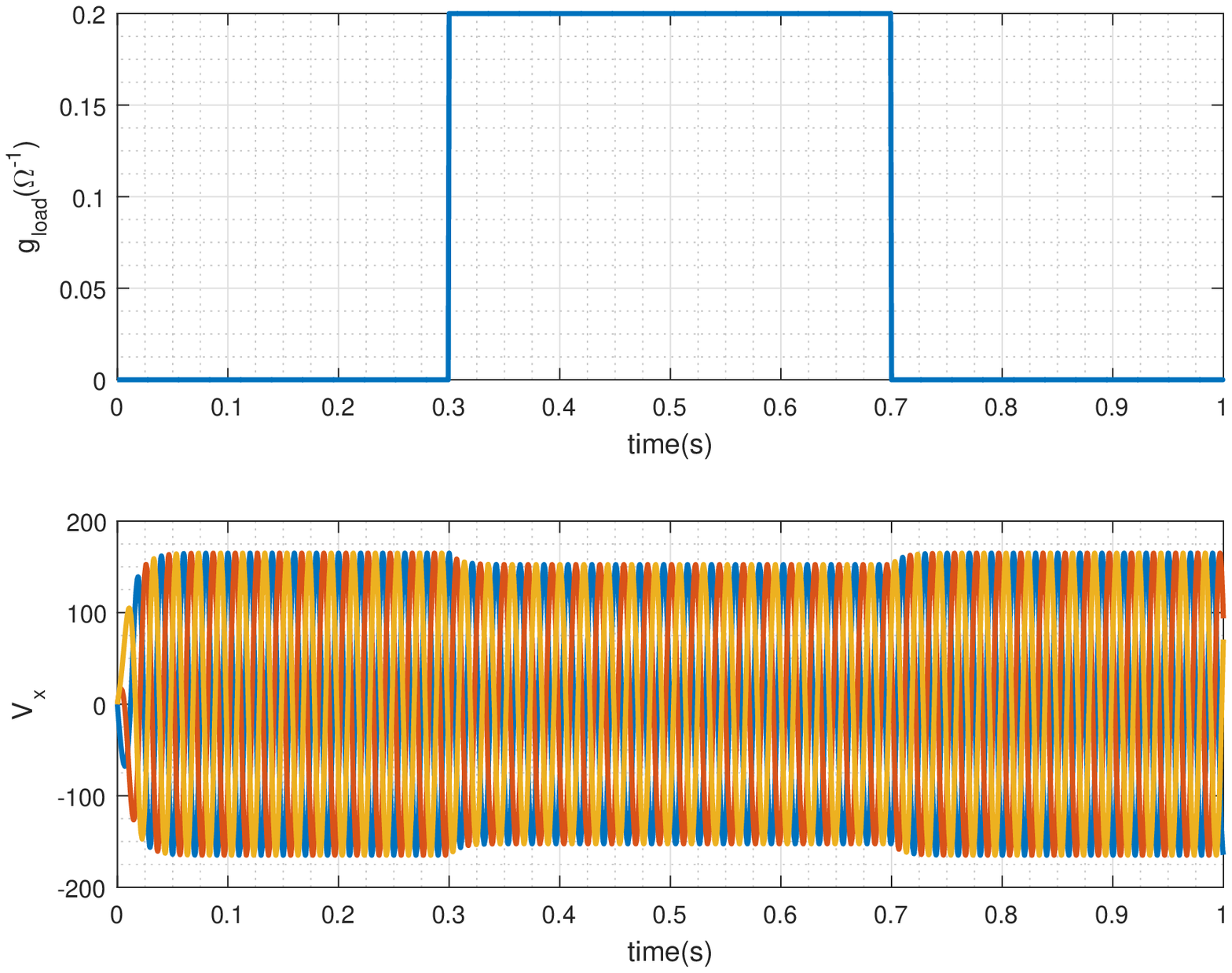}}
			\,
			\subfloat[Amplitude and Frequency of $v_x$]{\includegraphics[scale=0.3]{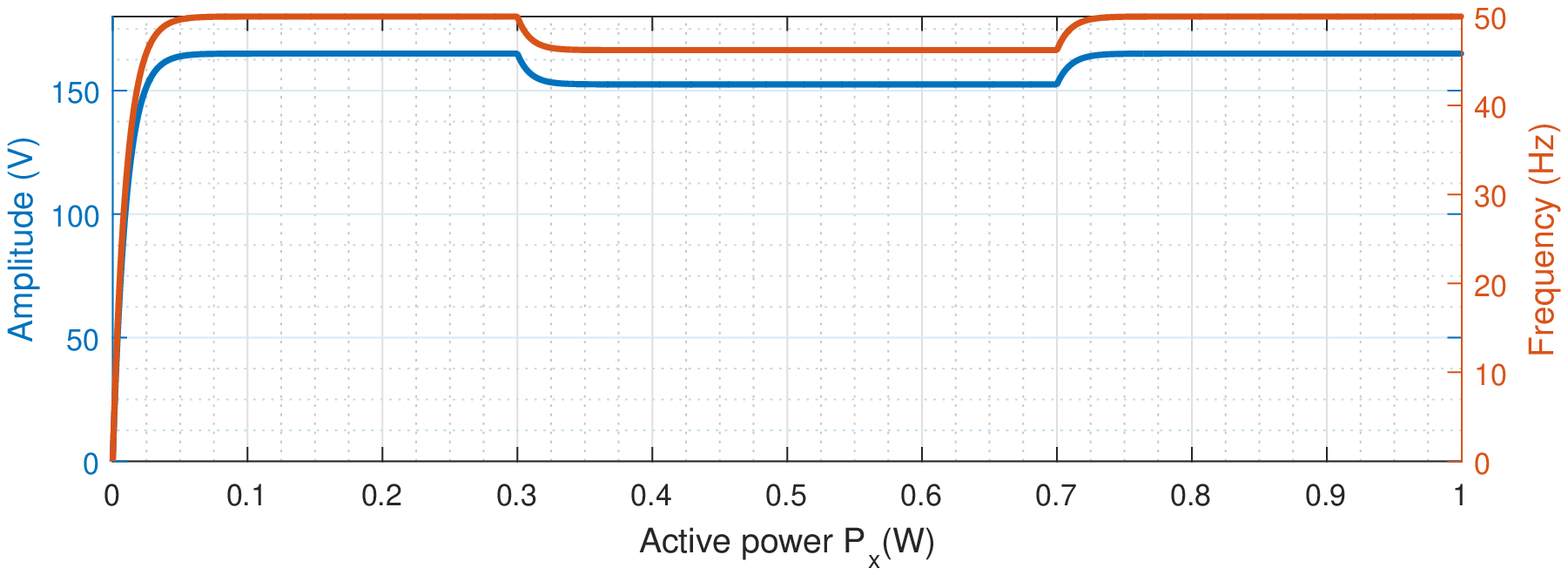}}
			\caption{Time domain simulation of the modulation voltage terminal with resistive load.}
			\label{fig: resistive load}
		\end{figure}
		
		\begin{figure}[h!]
			\centering
			\subfloat[Reactive load profile and voltage $v_x$]{\includegraphics[scale=0.3]{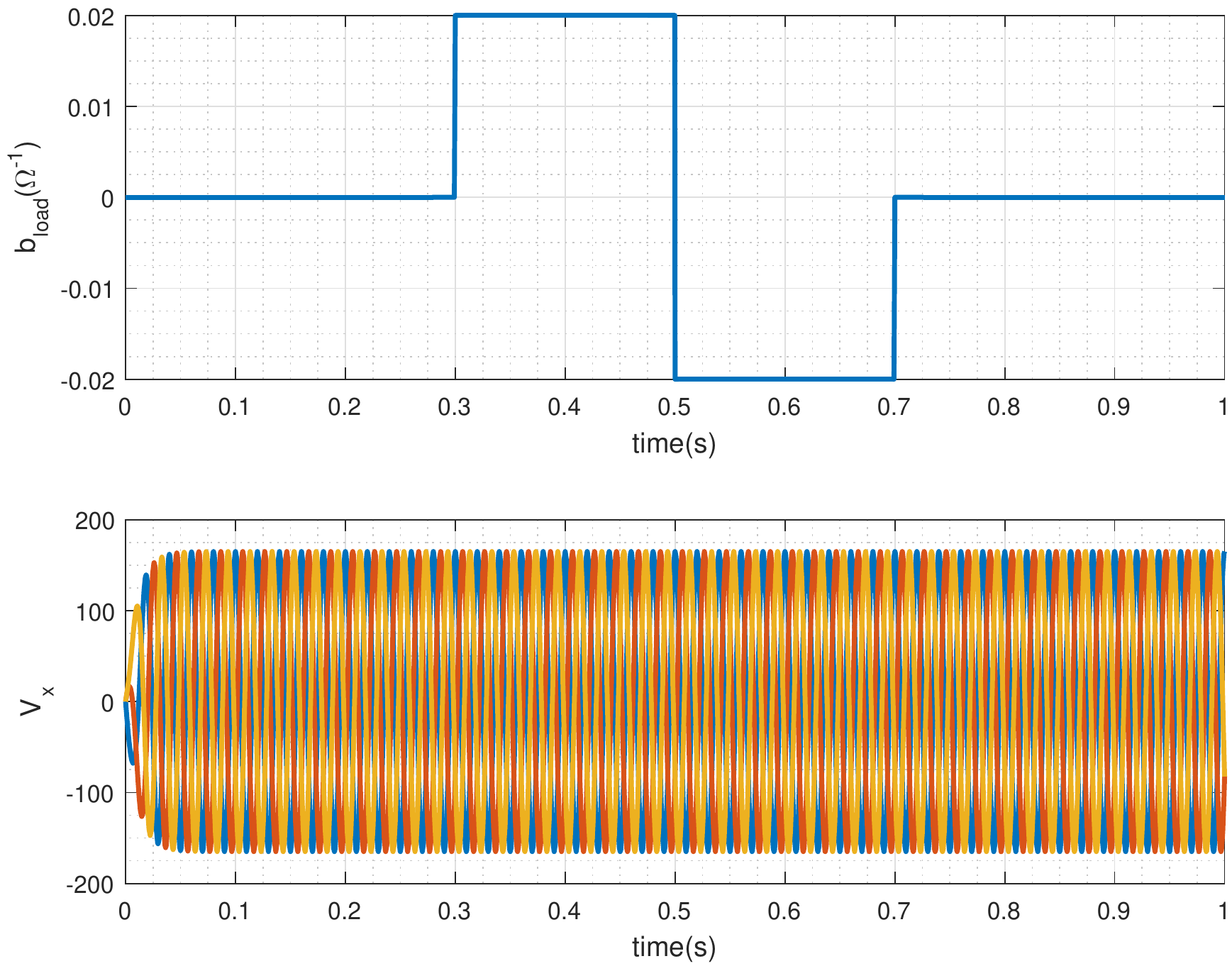}}
			\,
			\subfloat[Amplitude and Frequency of $v_x$]{\includegraphics[scale=0.3]{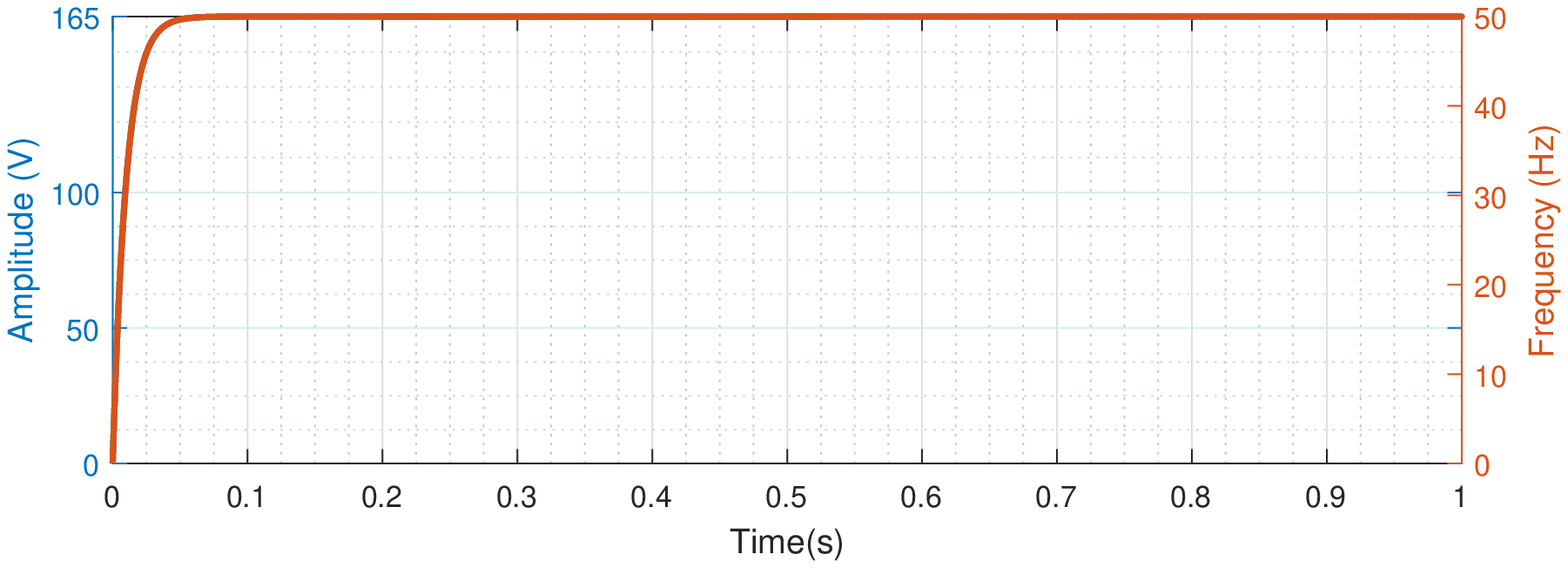}}
			\caption{Time domain simulation of the modulation voltage terminal with resistive load.}
			\label{fig: reactive load}
		\end{figure}
	
\begin{remark}[Parametric Sensitivity]
We investigate through different simulations the effect of slightly varying each of the controller parameters $\mu,\,\eta$ as well as the DC current source $i_{dc}$ as depicted in Figure \ref{fig: parameter effect}. 
There is a trade-off between the amplitude $\mu$ and the maximal amplitude $\bar {v}_x$ and the frequency gain $\eta$ and the maximal frequency $\bar {\omega}_x$ as described in \eqref{eq: max-amp-freq}. A change in $i_{dc}$ affects both maximal frequency and amplitude of $v_x$.

\begin{figure}[h!]
			\subfloat[Effect of change in modulation amplitude $\mu$ on the voltage amplitude $\hat {v}_x$]{\includegraphics[scale=0.3]{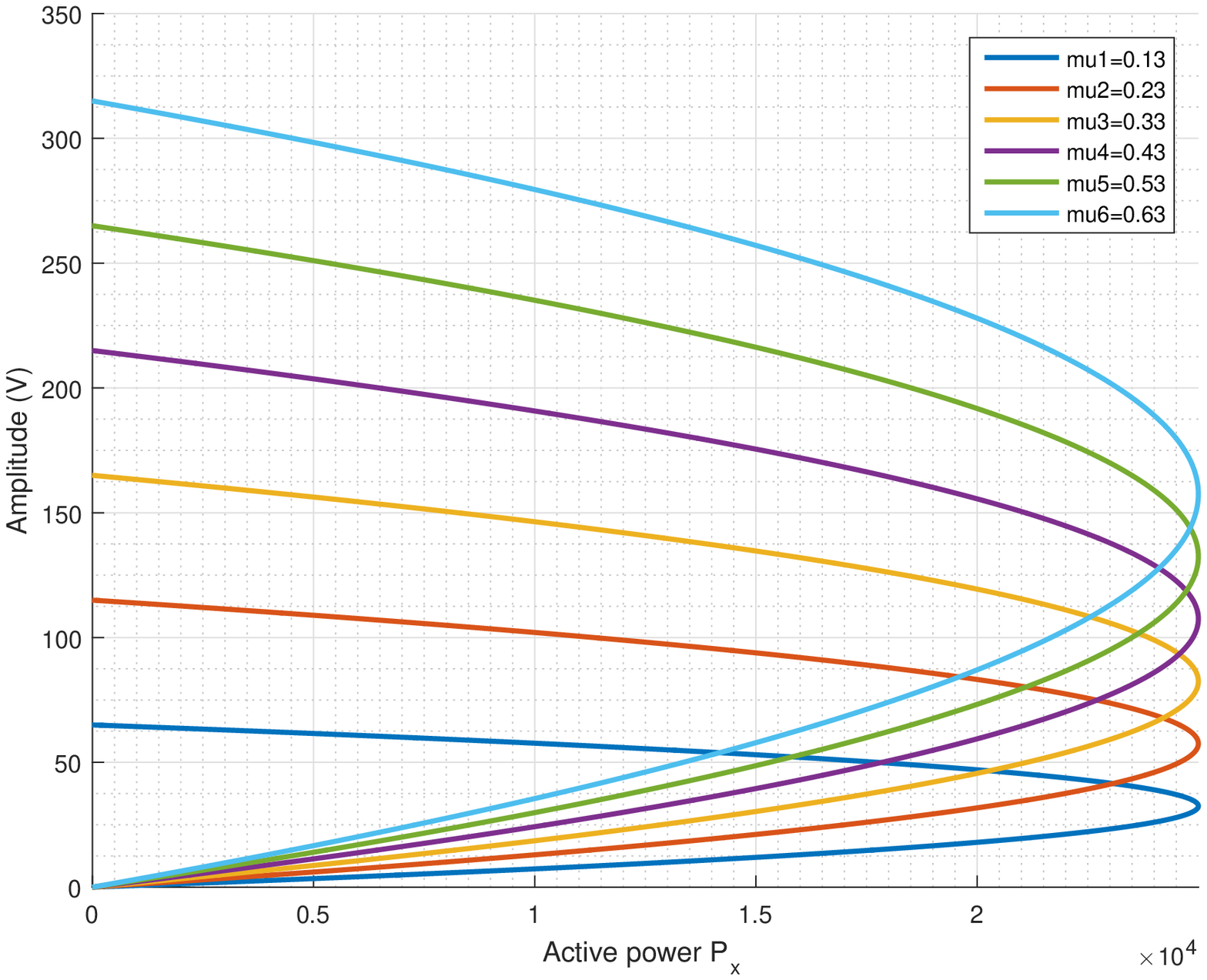}}
			\subfloat[Effect of a change in $\eta$ on the voltage terminal amplitude $\hat {v}_x$]{\includegraphics[scale=0.3]{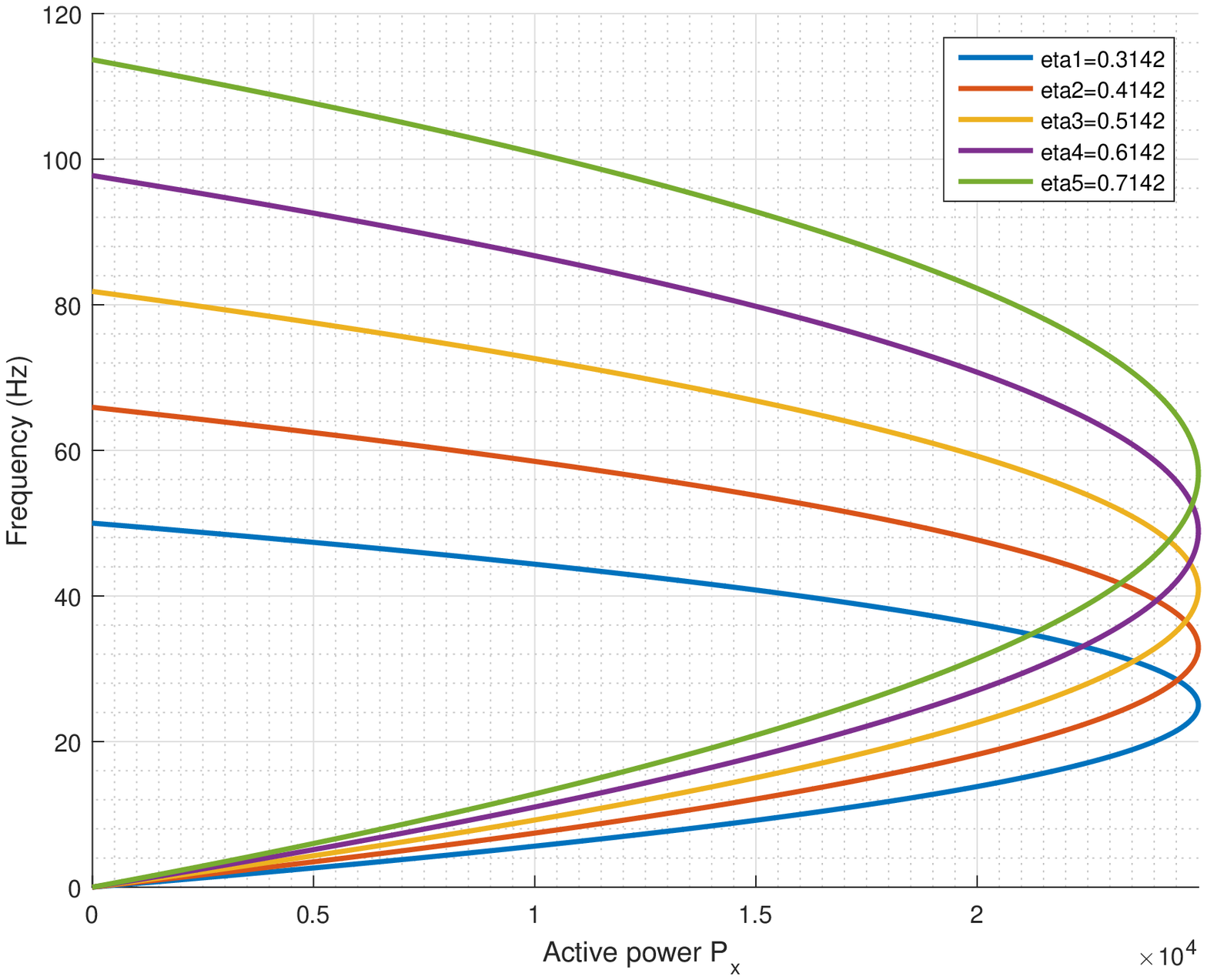}}
			\\
			\subfloat[Effect of a change in $i_{dc}$ on the voltage amplitude $\hat {v}_x$ ]{\includegraphics[scale=0.3]{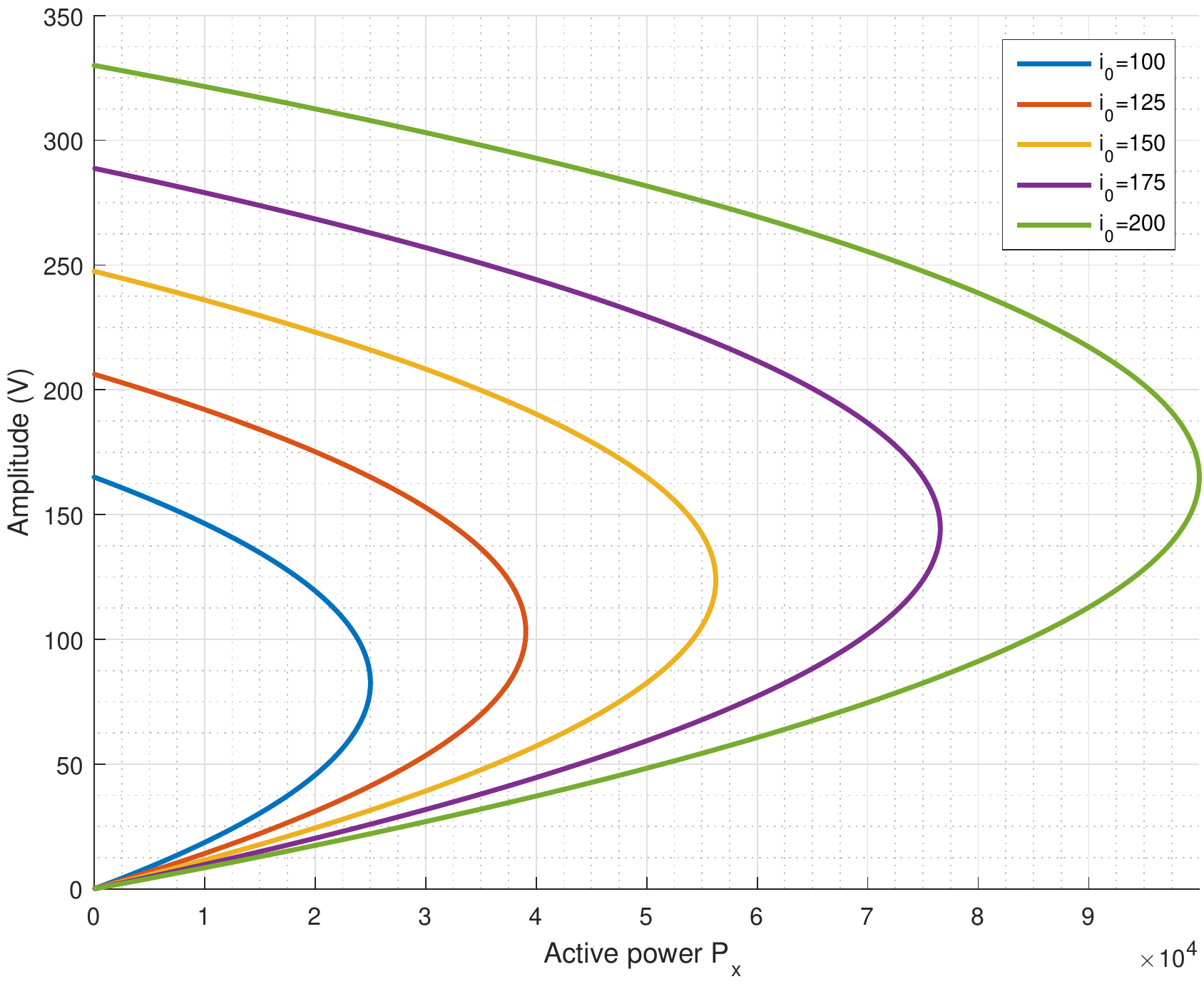}}
			\subfloat[Effect of change in $i_{dc}$ on the frequency $\omega_x$]{\includegraphics[scale=0.3]{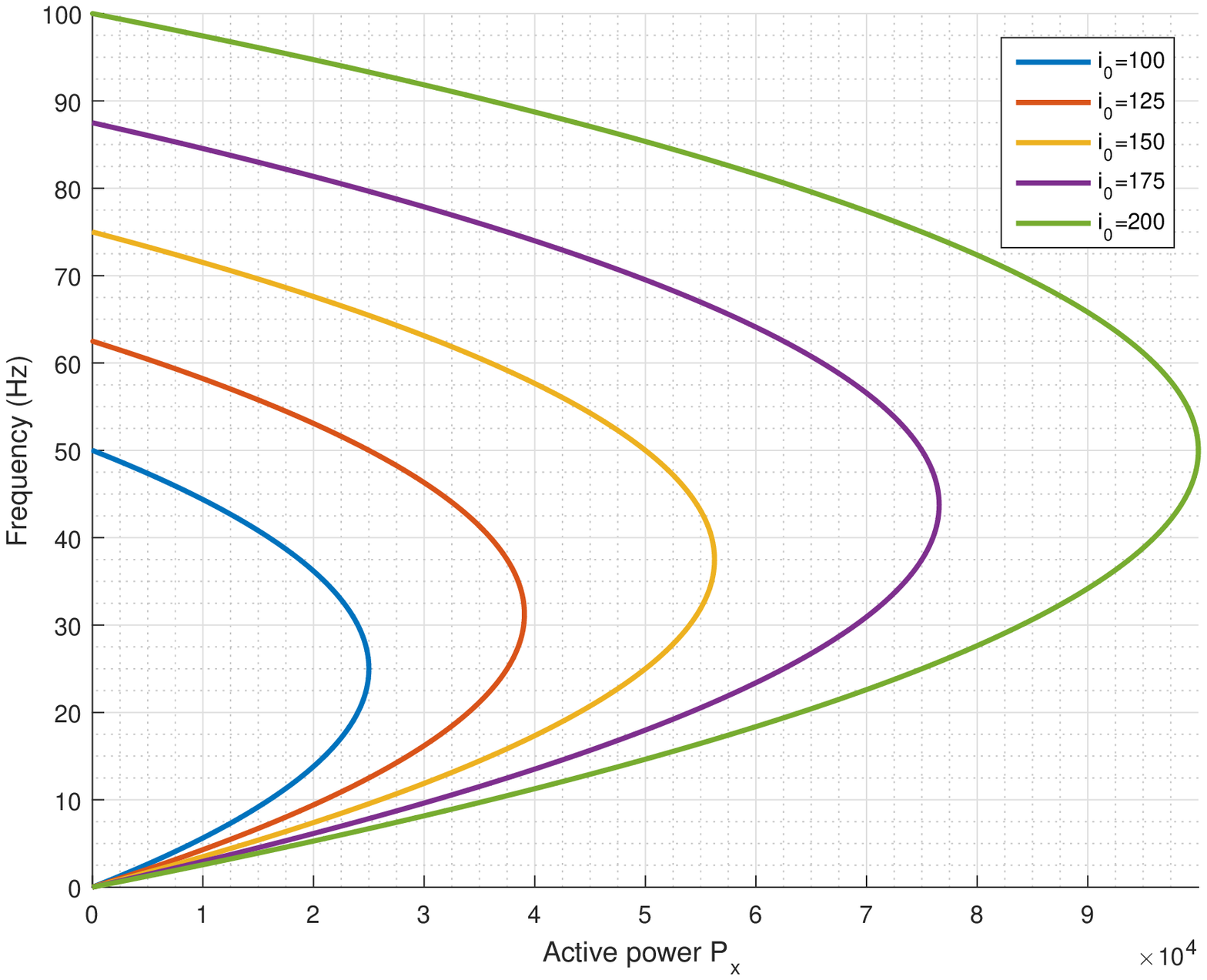}}
			\caption{Effect of the variation of the matching controller gains $\mu,\,\eta$ and the current source $i_{dc}$}
			\label{fig: parameter effect}	
\end{figure}
\end{remark}
\subsection{Analysis of power flow at the filter node}

We state the following main theorem

\begin{theorem}[Steady State Power balance at the filter node]
	We assume all AC quantities are synchronized at the same frequency $\omega_s$ at steady state. Active and reactive power $P_{load},\,Q_{load}$ at the filter node (i.e., after the converter AC circuit) are described as
	\begin{equation}
	\begin{bmatrix}
	P_{load} \\ Q_{load}
	\end{bmatrix}
	=
	\begin{bmatrix}
	P_{x}\\ Q_{x}
	\end{bmatrix}
	+
	\begin{bmatrix}
	-R \hat {v}_{\ell}^{2}\\\omega_s L \hat {v}_{\ell}^{2}+\omega_s C \hat {v}^{2}
	\end{bmatrix}
	\,,
	\label{eq: power balance at filter node}
	\end{equation}
	with $\hat {v}_\ell$ is the amplitude of $i_{\alpha\beta}$ and $\hat {v}$ is the amplitude of $v_{\alpha\beta}$. This relation describes the apparent power balance in the converter at steady state. 
\end{theorem}

\begin{proof}
	We consider $(\alpha\beta)$- frame and define the instantaneous active and reactive power going out of the output of the modulation block as a function of the voltage $v_{x}$ and the inductance current $i_{\alpha\beta}$ as follows 
	\begin{equation}
	\begin{bmatrix}  P_x \\ Q_x \end{bmatrix} = \underbrace{\begin{bmatrix} v_{x,\alpha}  &  v_{x,\beta } \\ v_{x, \beta }  & -v_{x,\alpha}\end{bmatrix}}_{V_x}  \begin{bmatrix} i_{\alpha } \\ i_{\beta }\end{bmatrix}
	\,,
	\label{eq: instantaneous power}
	\end{equation}
	where $V_x \in \mbb R^{2\times 2}$ is the matrix with terms representing the components of the vector $v_x$.
	We now derive an expression which characterizes the power balance after the RLC filter. 
	We define the polar coordinates of the voltage across the capacitor $v_{\alpha\beta}$ and the inductance current $i_{\alpha\beta}$ as
	
	\begin{subequations}
		\begin{align}
		v_{\alpha\beta} &= \hat v \begin{bmatrix}
		-\sin(\theta) \\ \cos(\theta)
		\end{bmatrix}
		\\
		\,
		i_{\alpha\beta} &= \hat {v}_\ell  \begin{bmatrix}
		-\sin(\theta_\ell) \\ \cos(\theta_\ell)
		\end{bmatrix}
		\,.
		\end{align}
	\end{subequations}

	We recall the dynamics of the capacitor voltage and inductance current introduced in~\eqref{eq: AC dynamics}.
	\\
	AC circuit dynamics in $(\alpha\beta)$- frame are given by:
	\begin{subequations}
		\begin{align}
		\label{eq: AC capacitor}
		C \dot v_{\alpha\beta} &= -i_{load}+i_{\alpha\beta}
		\\
		\,
		L \dot {i_{\alpha\beta}} &= -R i_{\alpha\beta} + \frac{1}{2}m_{\alpha\beta} v_{dc} - v_{\alpha\beta}
		\,.
		\label{eq: AC inductor}
		\end{align}
		\label{eq: AC dynamics alpha-beta}
	\end{subequations}
	The instantaneous active and reactive power at the load node can be written as follows:
	\begin{equation}
	\label{eq: rotating frame output}
	\begin{bmatrix}  P_{load} \\ Q_{load} \end{bmatrix} = \underbrace{\begin{bmatrix} v_{\alpha}  &  v_{\beta} \\ v_{\beta }  & -v_{\alpha }\end{bmatrix}}_{V_{c}}  \begin{bmatrix} i_{ {load},\alpha} \\ i_{load,\beta}\end{bmatrix}
	\,,
	\end{equation}
	where the matrix $V_c \in \mbb R^{2\times 2}$.

	We assume further that there are balanced sinusoidal steady state solutions to the equations described in \eqref{eq: AC dynamics alpha-beta}, which exhibit harmonic oscillations synchronous at a non zero steady state frequency $\omega_s=\dot \theta$.
	After multiplication of \eqref{eq: AC capacitor} with the matrix $V_c$, we use the following relationship at steady state
	\begin{subequations}
		\begin{align}
		V_c \dot v_{\alpha\beta} &= V_c \omega_s J_2 v_{\alpha\beta}\\
		&=\omega_s \hat {v}^2\begin{bmatrix} \cos\theta  & \sin\theta   \\ \sin\theta  & -\cos\theta \end{bmatrix}\begin{bmatrix} -\sin\theta \\ \cos\theta \end{bmatrix}\\
		&=\begin{bmatrix}
		0\\ -\omega \hat {v}^2
		\end{bmatrix}
		\,.
		\end{align}
	\end{subequations}
	
	The dynamics simplify to:
	\begin{equation}
		C\begin{bmatrix}
		0\\ -\hat{v}_{\alpha\beta}^{2} \omega_s
		\end{bmatrix}
		=-\begin{bmatrix}
		P_{load}\\ Q_{load} 
		\end{bmatrix}+ V_c i_{\alpha\beta}
		\,.
		\label{eq: capacitor ss}
	\end{equation}
	
	In order to identify the second term $V_c\, i_{\alpha\beta}$, we redefine it as follows: 
	\begin{equation*}
	V_c \, i_{\alpha\beta}=\begin{bmatrix} v_{\alpha}  &  v_{\beta} \\ v_{\beta} & -v_{\alpha}\end{bmatrix}  \begin{bmatrix} i_{\alpha} \\ i_{\beta}\end{bmatrix}
	=\underbrace{\begin{bmatrix} i_{\alpha}  &  i_{\beta} \\ -i_{\beta}  &  i_{\alpha}\end{bmatrix}}_{I_\ell}  \begin{bmatrix} v_{\alpha} \\ v_{\beta}\end{bmatrix}
	\,,
	\end{equation*} 
	where the matrix $I_\ell \in \mbb R^{2\times 2}$.

	We multiply \eqref{eq: AC inductor} from the left with the matrix $I_\ell$ and get the following expression at steady state: 
	
	\begin{equation}
	L\begin{bmatrix}
	0\\ -\omega_s \hat {v}_{\ell}^{2}
	\end{bmatrix}=
	\begin{bmatrix}
	-R \hat {v}_{\ell}^{2} \\ 0
	\end{bmatrix} +
	\begin{bmatrix}
	P_x\\ Q_x
	\end{bmatrix}
	- I_\ell v_{\alpha\beta}
	\,,
	\label{eq: inductance ss}
	\end{equation}
	with $\omega_s=\dot \theta_\ell$, where we make use of the following relationship at steady state:
	
	\begin{subequations}
		\begin{align*}
		I_\ell \dot i_{\alpha\beta} &= I_\ell \omega_s J_2 i_{\alpha\beta}\\
		&=\omega_s \hat {v}_{\ell}^2 \begin{bmatrix} \cos\theta_\ell  & \sin\theta_\ell   \\ \sin\theta_\ell  & -\cos\theta_\ell \end{bmatrix}\begin{bmatrix} -\sin\theta_\ell \\ \cos\theta_\ell \end{bmatrix}\\
		&=\begin{bmatrix}
		0\\ -\omega_s \hat {v}_{\ell}^2
		\end{bmatrix}
		\end{align*}
		\,.
	\end{subequations}
	
	We now add the equations 
	\begin{equation*}
	L\begin{bmatrix}
	0\\ -\omega_s \hat {v}_{\ell}^2
	\end{bmatrix}+C\begin{bmatrix}
	0\\ -\hat v^{2} \omega_s
	\end{bmatrix}=\begin{bmatrix}
	-R \hat v_{\ell}^2 \\ 0
	\end{bmatrix}+
	\begin{bmatrix}
	P_x\\ Q_x
	\end{bmatrix}-
	\begin{bmatrix}
	P_{load}\\ Q_{load}
	\end{bmatrix}
	\,.
	\end{equation*} 
	
	After combining \eqref{eq: inductance ss} with \eqref{eq: capacitor ss}, we derive the power balance equation at load node:
	\begin{equation*}
	\begin{bmatrix}
	P_{load}\\ Q_{load}
	\end{bmatrix}=
	\begin{bmatrix}
	P_x \\ Q_x
	\end{bmatrix}+
	\begin{bmatrix}
	-R \hat {v}_{\ell}^2\\ L \omega_s \hat v_{\ell}^{2}+C \omega_s  \hat {v}^{2}
	\end{bmatrix}
	\,.
	\end{equation*}
\end{proof}
The equation \eqref{eq: power balance at filter node} reveals the power flow between the node at the output of the modulation block and the node at the RLC filter, such that there exist losses of active and reactive power in the linear AC circuit. Nevertheless these power losses are small due to the small values of the electrical AC components.

\section{Steady state analysis of the DC/AC converter for a constant load}

\subsection{Synchronization of AC signals}
\label{subsec: synchro-ac}
\begin{definition}[Inductive and capacitive load]
	For an inductive reactive load as depicted in Figure \ref{fig: ind-load}, we can write the inductive current $i_l$ as follows
	\begin{equation*}
		i_l := j(\frac{-1}{\omega L})v_l=j b v_l
		\,, 
	\end{equation*} 
	with $b<0$ for an inductive power $Q=-b v_l^2>0$. It holds for the capacitor current $i_c$ the following
	\begin{equation*}
		i_c := j\omega Cv_c = j b v_c
		\,,
	\end{equation*}
	with $b>0$ for a capacitive power $Q=-b v_l^2<0$ as depicted in Figure \ref{fig: cap-load}. $v_l$ is the voltage across the inductor and $v_c$ that across the capacitor. 
	
	\begin{figure}
		\begin{center}
			\begin{circuitikz}[american voltages]
				\draw
				(0,0) to [short,i=$i_{l}$] (1,0) 
				to [L, l=$L$] (2,0); 
			\end{circuitikz}
		\end{center}
		\caption{ Representation of an inductive load with the voltage $v_l$}
		\label{fig: ind-load}
	\end{figure}
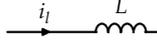

	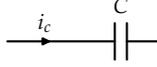
\begin{figure}
		\begin{center}
			\begin{circuitikz}[american voltages]
				\draw
				(0,0) to [short,i=$i_{c}$] (1,0) 
				to [C, l=$C$] (2,0); 
			\end{circuitikz}
		\end{center}
		\caption{Representation of a capacitive load with the voltage $v_c$}
		\label{fig: cap-load}
	\end{figure}

	\label{def: reactive load}
\end{definition}

\begin{theorem}
	We suppose that the inverter is interfaced with a constant load impedance defined as $G_{load}\in\real^{2\times 2}$ by
	
	\begin{equation}
	G_{load}=
	\begin{bmatrix}
	g & -b \\
	b & g 
	\end{bmatrix}
	\,,
	\label{eq: load-imp}
	\end{equation}
	with $g>0$ resistive and $b\in\real, b\neq -C\omega_s$ reactive load.
	At steady state, all AC signals synchronize at the same frequency $\omega_s$ 
	\begin{equation*}
	{\dot \theta=\dot \theta_x=\dot \theta_l=\omega_s}
	\,.
	\end{equation*}
\end{theorem}

\begin{proof}
	We define the following AC signals at steady state by its amplitudes and angular velocities.
	\begin{equation*}
	v_{\alpha\beta} = \hat {v} \begin{bmatrix}
	-\sin(\theta) 
	\\ 
	\cos(\theta)
	\end{bmatrix},\,\,
	i_{\alpha\beta} = \hat {v}_\ell \begin{bmatrix}
	-\sin(\theta_\ell)
	\\
	\cos(\theta_\ell)
	\end{bmatrix}
	,\,\,
	v_x= \hat {v}_x \begin{bmatrix}
	-\sin(\theta_x)
	\\
	\cos(\theta_x)
	\end{bmatrix}
	\,,
	\end{equation*}
	with $\hat v,\hat {v}_\ell, \hat {v}_x>0,\, \omega _x=\omega_v$, $v_{\alpha\beta}$ is the capacitor voltage, $i_{\alpha\beta}$ is the inductance current and $v_x$ is the voltage terminal at the output of the switching block.
	
	We suppose that the inverter is interfaced with a constant load impedance defined as $G_{load}\in\real^{2\times 2}$ as follows
	
	\begin{align*}
	G_{load}=
	\begin{bmatrix}
	g & -b \\
	b & g 
	\end{bmatrix}
	\,,
	\end{align*}
	with $g>0$ and $b\in\real, b\neq -C\omega_s$ resistive, respectively reactive load as introduced in Definition \eqref{def: reactive load}.

	We first examine the capacitor equation assuming harmonics at steady state synchronous at the frequency $\dot\theta=\omega_s$ and express it in terms of the above defined signals and we get:
	\begin{equation}
	C \hat v \omega_s \begin{bmatrix}
	-\cos(\theta) 
	\\ 
	-\sin(\theta)
	\end{bmatrix}=
	\hat v_\ell \begin{bmatrix}
	-\sin(\theta_\ell) 
	\\ 
	\cos(\theta_\ell)
	\end{bmatrix}-g\, \hat v \begin{bmatrix}
	-\sin(\theta) 
	\\ 
	\cos(\theta)\end{bmatrix} - b\, \hat v\begin{bmatrix}
	-\cos(\theta) 
	\\ 
	-\sin(\theta)
	\end{bmatrix}
	\,.
	\label{eq: cap-eq}
	\end{equation} 
	
	We multiply from the left by $\begin{bmatrix} -\cos(\theta) & -\sin(\theta)\end{bmatrix}$ to get the following equation 
	
	\begin{equation*}
	C \hat v\, \omega_s = \hat v_\ell \sin(\theta_\ell-\theta)
	\,,
	\end{equation*}
	
	and that
	\begin{equation}
	{\sin(\theta_l-\theta) = \frac{C \hat v \omega_s + b\hat v}{\hat v_\ell}}
	\,. 
	\label{eq: sinlc}
	\end{equation}
	
	We multiply now from the left by $\begin{bmatrix} -\sin(\theta) & \cos(\theta)\end{bmatrix}$ to get the following equation 
	
	\begin{equation*}
	0 = \hat {v}_\ell \cos(\theta_\ell-\theta) - g \hat {v} 
	\,.
	\end{equation*}
	and we deduce that 
	\begin{equation}
	{\cos(\theta_\ell-\theta) = g \frac{\hat v}{\hat {v}_\ell} \neq 0}
	\,.
	\end{equation}
	
	We drive \eqref{eq: sinlc} with respect to the time to get
	\begin{equation*}
	(\dot{\theta}-\dot{\theta_\ell}) \cos(\theta_\ell-\theta)=0
	\,.
	\end{equation*}
	
	Since $\cos(\theta_\ell-\theta)\neq 0$, we deduce that 
	
	\begin{equation}
	\dot\theta=\dot{\theta_\ell}=\omega_s
	\,.
	\label{eq: c=l}
	\end{equation}
	
	Next, we rewrite the inductance equation in terms of the polar coordinates of the above defined AC signals
	\begin{equation}
	L \omega_s \hat {v}_\ell \begin{bmatrix}
	-\cos(\theta_\ell) 
	\\ 
	-\sin(\theta_\ell)
	\end{bmatrix} = -R \hat {v}_\ell  \begin{bmatrix}
	-\sin(\theta_\ell) 
	\\ 
	\cos(\theta_\ell)
	\end{bmatrix}-\hat v \begin{bmatrix}
	-\sin(\theta) 
	\\ 
	\cos(\theta)
	\end{bmatrix} + \hat {v}_x \begin{bmatrix}
	-\sin(\theta_x) 
	\\ 
	\cos(\theta_x)
	\end{bmatrix}
	\,.
	\label{eq: induc-eq}
	\end{equation}
	
	We now multiply with the vector $\begin{bmatrix}
	-\cos(\theta) -\sin(\theta)
	\end{bmatrix}$ from the left. It yields that
	
	\begin{equation*}
	0=-R \hat{v}_\ell + \hat {v}_x\cos(\theta_x-\theta_\ell)
	\,,
	\end{equation*}
	
	and it follows that
	
	\begin{equation}
	{\cos(\theta_x-\theta_\ell)=\frac{R \hat {v}_\ell}{r_x}>0}
	\,.
	\end{equation}
	
	If we multiply from the left with the vector $\begin{bmatrix}
	-\cos(\theta_\ell) -\sin(\theta_\ell)
	\end{bmatrix}$, it holds
	
	\begin{equation*}
	L \hat {v}_\ell \omega_\ell= -\hat {v}\sin(\theta-\theta_\ell)+\hat {v}_x\sin(\theta_x-\theta_\ell)
	\,.
	\end{equation*}
	
	We get the following
	
	\begin{equation}
	{\sin(\theta_x-\theta_l)=\frac{L\hat {v}_\ell\omega_\ell+\hat v\sin(\theta-\theta_\ell)}{\hat {v}_x}}
	\,.
	\end{equation}
	
	We now differentiate with respect to the time and get
	
	\begin{equation*}
	\frac{d}{dt}(\sin(\theta_x-\theta_\ell))=(\dot\theta_x-\dot{\theta_\ell}) \cos(\theta_x-\theta_\ell)
	\,.
	\end{equation*}
	Since $\cos (\theta_x-\theta_\ell)\neq 0$, we get
	
	\begin{equation}
	\dot\theta_x=\dot\theta_\ell=\omega_s
	\,.
	\label{eq: x=l}
	\end{equation}
	
	By combining \eqref{eq: x=l} and \eqref{eq: c=l}, we get
	\begin{equation*}
	\dot\theta_x=\dot\theta_\ell=\dot{\theta}=\omega_s
	\,.
	\end{equation*}
\end{proof}

\subsubsection{Case studies of different loads}
	
	When choosing the purely inductive load to be $b=b_{cri}=-C\omega_s$ and $g=0$, using the power balance equation in \eqref{eq: power balance at filter node}, the DC/AC inverter can deliver the inductive reactive power
	
	\begin{equation*}
	Q_{cri}=-b_{cri} \hat v^2= C \omega_s \hat v^2
	\,.
	\end{equation*}
	From \eqref{eq: cap-eq}, if we set $g=0$, we can derive the following relationships depending on the reactive load, in case it is under-critical ($b<-C\omega_s$), respectively over- critical ($b>-C\omega_s$)  
	
	\begin{equation*}
	\hat {v}_\ell = \hat v|C\omega_s+b|
	\,, 
	\end{equation*}
	and the following holds
	
	\begin{equation*}
	\hat {v}_\ell=\sqrt{i_{\alpha}^2+i_{\beta}^2}=0
	\,,
	\end{equation*}
	
	so that we conclude that the inductance current is zero, when choosing this critical inductive load.

	It is noteworthy that at this step, using the inductance equation we have also 
	\begin{equation*}
	{\hat {v}_x = \hat {v},\, \theta_x=\theta}
	\,.
	\end{equation*}
	
	The voltage across the capacitor and at the output of the modulation block are the same and therefore synchronize in angle and amplitude.
	Due to the presence of the capacitor in AC circuit, we further consider  under-critical ($b<b_{cri}$), respectively over- critical ($b>b_{cri}$).  
	
	In case $g\neq0$ and $b=b_{cri}$, then it holds
	
	\begin{equation*}
	\sin(\theta_\ell-\theta)=0
	\,,
	\end{equation*}
	and it holds
	\begin{equation*}
	{\theta_\ell=\theta} 
	\,,
	\end{equation*}
	and that
	
	\begin{equation*}
	{g=\frac{\hat v}{\hat {v}_\ell}}
	\,.
	\end{equation*}
	The inductance current and capacitor voltage synchronize in angle.
	
	In case of $g=0$ and $b\neq b_{cri}$. In case of a purely non-critical reactive load, we use the capacitance equation to get
	\begin{equation*}
	\cos(\theta_\ell-\theta)=0
	\,,
	\end{equation*}
	and we have 
	\begin{equation*}
	{\theta_\ell-\theta=\frac{\pi}{2}}
	\,.
	\end{equation*}

	In general, for a $b\neq b_{cri}$ and $g\neq0$ holds
	\begin{equation*}
	{0<\theta_\ell-\theta < \frac{\pi}{2}}
	\,.
	\end{equation*}

\subsubsection{Limits on current amplitude}
	In the case of an open-circuit operation, i.e $b=0$ and $g=0$, the amplitude $\hat {v}_\ell>0$ can be expressed from \eqref{eq: sinlc} as
	\begin{equation*}
	\hat v_{\ell,open}=b\hat v + C\,\hat v\,\omega_s= C\,\hat v\,\omega_s
	,\,
	Q_{x, open}=-L C^2\omega_s^3\hat v^2- \hat v^2C\omega_s
	\,,
	\end{equation*} 
	corresponding to the experimental values 
	\begin{equation*}
	\hat v_{\ell,open} = 0.518V 
	,\, \, \,
	Q_{x, open} = -128VAR
	\,.
	\end{equation*} 
	This is in accordance with the intuition that in an open-circuit, the capacitor to the ground can be interpreted as a capacitive load such that $Q_{x, open}<0$. 
	\\
	The active power can be expressed as 
	\begin{equation*}
	P_x=\hat {v}_x\,\hat {v}_\ell\, cos(\theta_x-\theta_\ell)
	\,,
	\end{equation*}
	using the inductance equation in \eqref{eq: induc-eq}, we can show that when no active load is present, the active power at the output of the modulation block corresponds to 
	\begin{equation}
	P_x=R\,\hat v_\ell^2
	\,,
	\label{eq: act-power}
	\end{equation} 
	such that the active power $P_x$ at open circuit is:
	\begin{equation*}
	P_{x, open} = R (C \hat v\omega_s)^2
	\,,
	\end{equation*}
	and the maximal current amplitude, is identified as
	\begin{equation}
	\bar v_{\ell}=\sqrt{\frac{\bar P_{x}}{R}}=\frac{i_{dc}}{2\sqrt{G_{dc}R}}
	\,.
	\label{eq: rl-max}
	\end{equation}
	By applying \eqref{eq: act-power} and corresponding experimentally to 
	\begin{equation*}
	\bar {v}_{\ell}=500A
	\,.
	\end{equation*}
	
	One can interpret this result by saying that the DC/AC converter in open-circuit is naturally resistive and capacitive.

\subsection{Analysis of purely constant reactive load in steady state}

In this section, we aim to identify the characteristic curves of the inverter at steady state relating  reactive load at the filter node $Q_{load}$ to the amplitude of capacitor voltage $\hat {v}_\ell$ and the inductance current $\hat v$. We identify certain limits on the relevant signals induced by the presence of purely reactive load in the DC/AC converter.

\begin{assumption}
We consider $(\alpha\beta)$- framework such that all AC signals are balanced. At steady state, the reactive load is given by the constant susceptance matrix $B_{load}\in\real^{2\times 2}$ defined by
\begin{equation*}
B_{load}=\begin{bmatrix}
0 & -b \\
b & 0
\end{bmatrix}=bJ_2
\,,
\end{equation*}
$b<b_{cri}$ for under-critical, $b>b_{cri}$ for over-critical load with $b_{cri}=-C\omega_s$.
We further assume that no active power is acting on the DC/AC converter, i.e $P_{load}=0$.
\end{assumption}

\begin{corollary}[Reactive load to the current amplitude]
	
	The relationship between the purely non-critical reactive load, i.e $b\neq b_{cri},\, g=0$ and the AC current amplitude is described by	
	\begin{equation}
	b(\hat {v}_\ell)=\pm\left(\frac{\hat {v}_\ell}{\sqrt{\frac{\mu^2}{16G_{dc}^2}\left(i_{dc}+\sqrt{i_{dc}^2-4G_{dc}R\hat {v}_\ell^2}\right)^2-R^2\hat {v}_\ell^2}-L\hat {v}_\ell\omega_s}-C\omega_s\right)
	\,.
	\label{eq: anal-b-to-rl}
	\end{equation}
\end{corollary}

\begin{proof}
	
	It follows from the case study that for $g=0,\, b\neq b_{cri}$ it holds that
	\begin{equation*}
	\theta-\theta_\ell=\frac{\pi}{2}
	\,.
	\end{equation*}

	Let us consider the inductance equation at steady state as introduced previously with $g=0$ and $b\neq b_{cri}$ and rewrite the main results obtained there as 
	
	\begin{subequations}
		\begin{align*}
		\cos(\theta_x-\theta_\ell) &=R\frac{\hat v_\ell}{\hat {v}_x}
		\\
		\sin(\theta_x-\theta_\ell) &=\frac{L\,\hat v_\ell \omega_s + \hat v}{\hat {v}_x}
		\,.
		\end{align*}
	\end{subequations}
	Using the fact that, $\cos(\phi)^2+\sin(\phi)^2=1,\forall\phi\in\mycircle$, we have 
	
	\begin{subequations}
		\begin{align}
		1-\cos(\theta_x-\theta_l)^2 &=\sin(\theta_x-\theta_l)^2
		\\
		1-\left(R\frac{\hat v_\ell}{\hat v_x}\right)^2 &=\left(\frac{L\hat v_\ell\omega_s+\hat v}{\hat v_x}\right)^2
		\\
		\hat {v}_x^2 - R^2\hat v_\ell^2 &=(L\hat v_\ell\omega_s+\hat v)^2
		\,.
		\label{eq: rx-rl-2}
		\end{align}
	\end{subequations}

	Using the fact that $P_x=R\hat v_\ell^2$, we plug it in DC circuit equation at steady state, in order to get
	
	\begin{subequations}
		\begin{align*}
		v_{dc}(P_x) &= \frac{i_{dc}+\sqrt{i_{dc}^2-4G_{dc}P_x}}{2G_{dc}}
		\\
		v_{dc}(\hat {v}_l) &= \frac{i_{dc}+\sqrt{i_{dc}^2-4G_{dc}R\hat v_\ell^2}}{2G_{dc}}
		\,.
		\end{align*}
	\end{subequations}
	
	Using the definition of $\hat {v}_x={\mu v_{dc}}/{2}$, we have
	
	\begin{equation}
	\hat {v}_x=\frac{\mu}{2}\frac{i_{dc}+\sqrt{i_{dc}^2-4G_{dc}R\hat {v}_\ell^2}}{2G_{dc}}
	\,.
	\label{eq: rx-rl}
	\end{equation}
	
	We apply \eqref{eq: rx-rl} in \eqref{eq: rx-rl-2} and after simplification using  $\hat {v}=\hat {v}_\ell/|b+C\omega_s|$
	
	\begin{equation*}
	|b+C\omega_s|=\frac{\hat {v}_\ell}{\sqrt{\frac{\mu^2}{16G_{dc}^2}\left(i_{dc}+\sqrt{i_{dc}^2-4G_{dc}R\hat {v}_\ell^2}\right)^2-R^2\hat {v}_\ell^2}-L \hat {v}_\ell\omega_s}
	\,.
	\end{equation*}

	We study respectively the over- and under critical cases. Thus,
	
	\begin{equation*}
	b(\hat {v}_\ell)=\pm\left(\frac{\hat {v}_\ell}{\sqrt{\frac{\mu^2}{16G_{dc}^2}\left(i_{dc}+\sqrt{i_{dc}^2-4G_{dc}R\hat {v}_\ell^2}\right)^2-R^2\hat {v}_\ell^2}-L\hat {v}_\ell\omega_s}-C\omega_s\right)
	\,.
	\end{equation*}
	
	And we deduce the expression of the reactive load at the filter node in function of the amplitude of the inductance current as
	
	\begin{equation*}
	Q_{load}=-b(\hat {v}_\ell)\hat {v}^2=\mp\left(\frac{\hat {v}_\ell \hat {v}^2}{\sqrt{\frac{\mu^2}{16G_{dc}^2}\left(i_{dc}+\sqrt{i_{dc}^2-4G_{dc}R\hat {v}_\ell^2}\right)^2-R^2\hat {v}_\ell^2}-L\hat {v}_\ell\omega_s}-C\omega_s\hat {v}^2\right)
	\,,
	\end{equation*}
	
	which is a function of both amplitudes of the inductance current $\hat {v}_\ell$ and of the capacitor voltage $\hat {v}$.
\end{proof}

\subsubsection{Simulation results}
	We simulate the DC/AC converter within an operating range corresponding to $Q_{load}\in[-5000,\,5000]VAR$.
	We plot the characteristic curve describing the dependency of the amplitude $\hat {v}_\ell$ to the load $b\in\real,\,b\neq -C\omega_s$. The general solution is shown in Fig. \ref{fig: sol-b-rl}.
	
	\begin{figure}[h!]
		\subfloat[Analytical solution found in \eqref{eq: anal-b-to-rl}]{\includegraphics[scale=0.3]{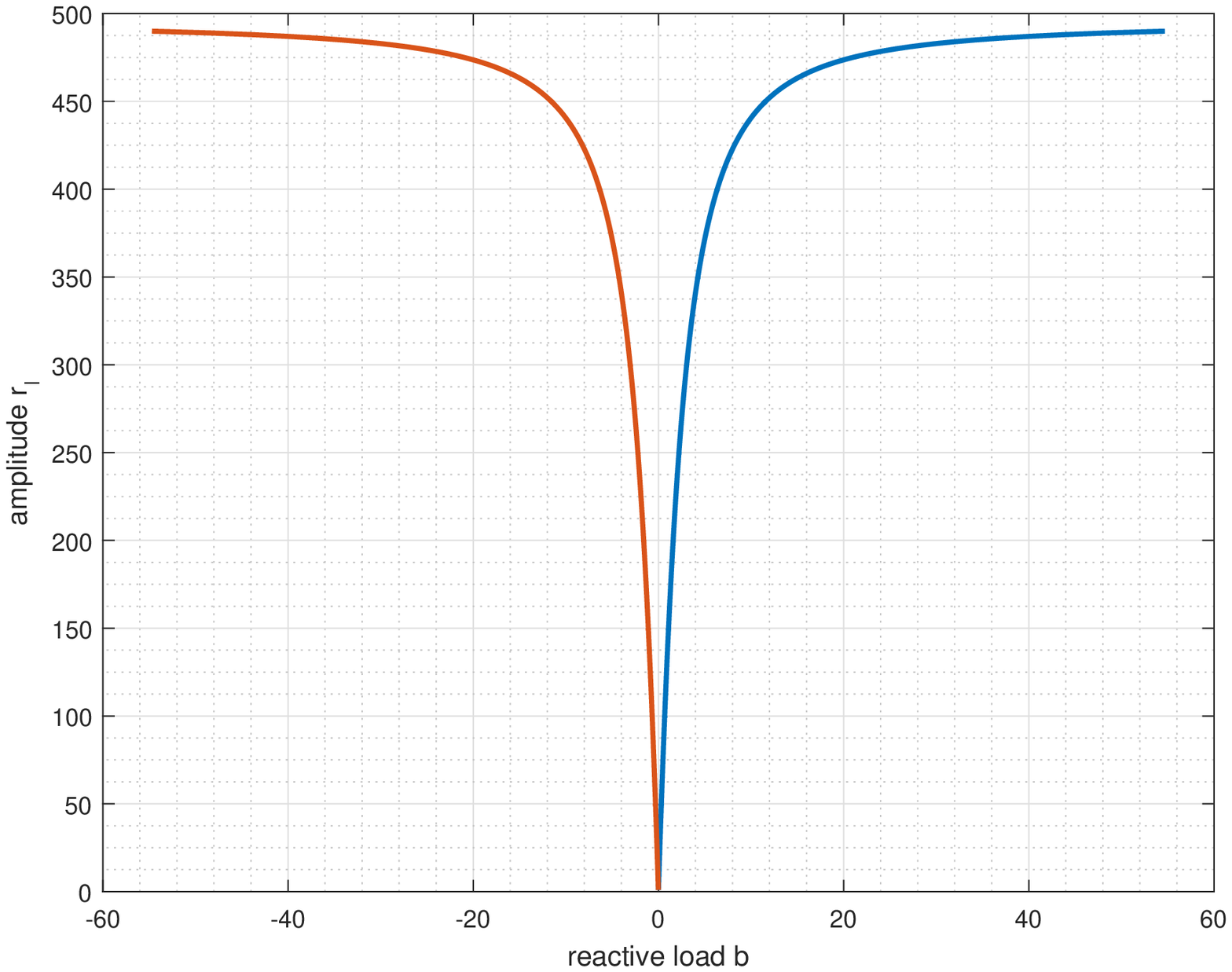}}
		\,
		\subfloat[Analytical solution found in \eqref{eq: anal-b-to-rl} within the operating range of reactive power corresponding to ]{\includegraphics[scale=0.3]{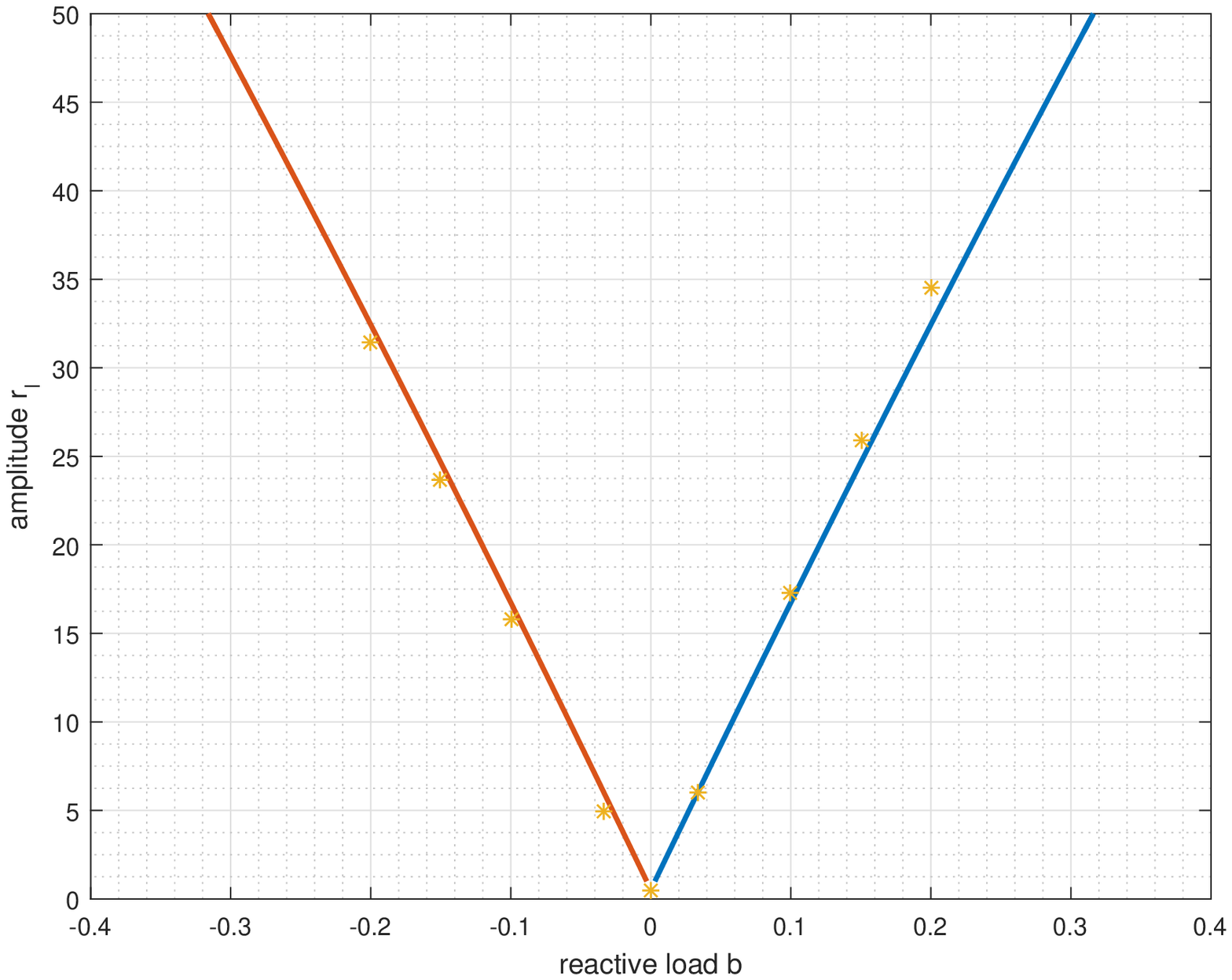}}
		\caption{Steady state simulations of the analytical solution as described in \eqref{eq: anal-b-to-rl} and around the operation range}
		\label{fig: sol-b-rl}
	\end{figure}
	The experiments confirm the analytical solution within the range of our study, where the maximal current amplitude $\bar v_\ell$ is as found in \eqref{eq: rl-max}.

\subsubsection{Amplitude ratio and reactive load}
	The relationship of the purely non-critical reactive load, i.e $b\neq b_{cri}, g=0$ to the amplitude ratio $\kappa=\hat v_\ell/\hat v$ of the AC circuit is defined by 
	\begin{equation}
	\kappa=|C\omega_s+b|
	\,.
	\label{eq: amp-ratio}
	\end{equation}

	From \eqref{eq: cap-eq} and after setting $g=0$, we write the capacitor equation as
	
	\begin{equation*}
	C \hat v \omega_s \begin{bmatrix}
	-\cos(\theta) 
	\\ 
	-\sin(\theta)
	\end{bmatrix}=
	\hat v_\ell \begin{bmatrix}
	-\sin(\theta_\ell) 
	\\ 
	\cos(\theta_\ell)
	\end{bmatrix} - b\, \hat v\begin{bmatrix}
	-\cos(\theta) 
	\\ 
	-\sin(\theta)
	\end{bmatrix}
	\,.
	\end{equation*}

	We can derive the following relationships depending on the reactive load, in case it is under-critical ($b<b_{cri}$), respectively over-critical ($b>b_{cri}$)  
	\begin{equation*}
	\hat v_\ell=\hat v|C\omega_s+b|
	\,.
	\end{equation*}
	
	We investigate the relationship between the amplitude variables $\hat v_\ell$ and $\hat v$ and the reactive load $b\in\real, b\neq-C\omega_s$. We define the amplitude ratio $\kappa={\hat {v}_\ell}/{\hat v}$ such that 
	\begin{equation*}
	\kappa= \frac{\hat {v}_\ell}{\hat v}=|C\omega_s+b|
	\,.
	\end{equation*}
	Depending on the applied load, i.e under- or over-critical, we can plot the following curves as shown in Fig.\ref{fig:bload_vs_ratio}, where the critical load $b_{cri}$ is excluded from the domain of definition and marked in green. The analytical solution describes an \eqref{eq: amp-ratio} affine function of the reactive load $b$ in function of the ratio $\kappa$ matches the experimental results.
	\begin{figure}[h!]
		\centering
		\includegraphics[scale=0.3]{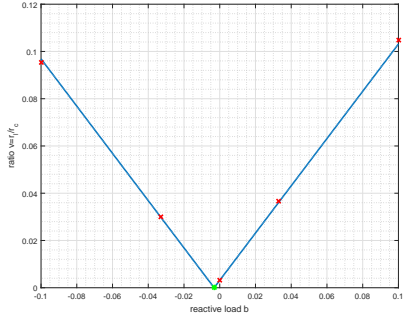}
		\caption{Reactive load $b$ versus amplitude ratio $\kappa$}
		\label{fig:bload_vs_ratio}
	\end{figure}

\section{Passivity analysis of the DC/AC converter}

Passivity is regarded here as decentralized stability certificate \cite{AJvdS:96},\,\cite{FZO+13} that allows for the converter to connect to AC and DC grids in a stable closed-loop fashion provided that those are passive as well. We investigate passivity in both open and closed-loop fashions

\subsection{Passivity analysis in open-loop fashion}

We note that for any choice of the modulation signal the system \eqref{eq: inverter dynamics} is passive with respect to the AC grid port as well as the DC port. 
\begin{lemma}[Modulation-independent passivity]
	\label{Lem: passivity open-loop}
	Consider the DC/AC converter model \eqref{eq: inverter dynamics}.  For any modulation signal $m_{\alpha\beta}$, the system is passive with respect to the input $u=\begin{bmatrix}
	i_{dc} & -i_{load}^{\top}\end{bmatrix}^{\top}$ and the output $y=\begin{bmatrix}v_{dc} & v_{\alpha\beta}^{\top}\end{bmatrix}^{\top}$.
\end{lemma}

\begin{proof}
Inspired by circuit theory, consider the positive definite storage function, $S: \mbb R^5\to \mbb R_{\geq 0}$, defined as
\begin{equation}
	S(v_{dc}, v_{\alpha\beta}, i_{\alpha\beta})= \frac{1}{2}C_{dc}v_{dc}^{2}+\frac{1}{2}C^{}v_{\alpha\beta}^{\top}v_{\alpha\beta}+\frac{1}{2}L^{}i_{\alpha\beta}^{\top}i_{\alpha\beta}
	\,.
	\label{eq: lyapunov function}
\end{equation}
	We calculate the directional derivative of $S$ along the vector field \eqref{eq: inverter dynamics} describing the DC/AC converter dynamics as 
	
	\[
	\dot S=\begin{bmatrix}
	v_{dc} \\ v_{\alpha\beta}\\ i_{\alpha\beta}
	\end{bmatrix}^{\top} 
	\begin{bmatrix}
	-{G_{dc}} & 0 & 0 \\
	0 & 0 & 0  \\
	0 & 0 &  -R I_{2}
	\end{bmatrix}
	\begin{bmatrix}
	v_{dc} \\ v_{\alpha\beta}\\ i_{\alpha\beta}
	\end{bmatrix}+\begin{bmatrix}
	i_{dc}^{} & -i_{load}
	\end{bmatrix} 
	\begin{bmatrix}
	v_{dc} \\ v_{\alpha\beta}	
	\end{bmatrix}
	\,,
	\]
	where $I_2$ is the identity matrix in $\real^2$. The claim follows from the definition of passivity \cite{AJvdS:96}.
\end{proof}

\subsection{Passivity analysis in the closed-loop fashion}

We first note that the closed-loop dynamics \eqref{eq: inverter dynamics},  \eqref{eq:modulation signal}, \eqref{eq: virtual angle} are passive as in Lemma~\ref{Lemma: passivity open-loop} also when augmenting the storage function \eqref{eq: lyapunov function} with an additional term accounting for the dynamics of the matching controller \eqref{eq: controller dynamics}. Consider the positive definite storage function $W: \real^{7}\to \real_{\geq0}$
\begin{equation}
W(v_{dc}, v_{\alpha\beta}, i_{\alpha\beta},m_{\alpha\beta})=S(v_{dc}, v_{\alpha\beta}, i_{\alpha\beta})+\frac{1}{2}m_{\alpha\beta}^{\top}m_{\alpha\beta}
\,,
\label{eq: storage function closed loop}
\end{equation}

where $S(v_{dc}, v_{\alpha\beta}, i_{\alpha\beta})$ is defined in \eqref{eq: lyapunov function}. The derivative along trajectories of the closed loop \eqref{eq: inverter dynamics}, \eqref{eq: controller dynamics} is given by
\begin{subequations}
	\begin{align*}
	\dot W &=\dot S+m_{\alpha\beta}^{\top}\dot m_{\alpha\beta}
	\,
	\\
	&=\dot S + \eta \, v_{dc}\, m_{\alpha\beta}^{\top}\, \begin{bmatrix}0 & -1 \\ 1 & 0\end{bmatrix} \, m_{\alpha\beta} 
	\\&\leq 
	\begin{bmatrix}
	i_{dc} & -i_{load}
	\end{bmatrix} 
	\begin{bmatrix}
	v_{dc} \\ v	
	\end{bmatrix}
	\,.
	\end{align*}
	\label{eq: closed loop passivity}%
\end{subequations}%
Hence the closed loop \eqref{eq: inverter dynamics},\eqref{eq: controller dynamics} (and thus also \eqref{eq: inverter dynamics},  \eqref{eq:modulation signal}, \eqref{eq: virtual angle}) is passive with input $(i_{dc}, -i_{load})$ and output $(v_{dc}, v_{\alpha\beta})$.

\begin{lemma}[Closed loop passivity]
	\label{Lemma: passivity open-loop}
	Consider the DC/AC converter model \eqref{eq: inverter dynamics} with the modulation control \eqref{eq: controller dynamics}. The closed-loop system \eqref{eq: inverter dynamics},\eqref{eq: controller dynamics}  is passive with respect to the input $u=\begin{bmatrix}
		i_{dc} & -i_{load}^{\top}\end{bmatrix}^{\top}$ and output $y=\begin{bmatrix}v_{dc} &  v_{\alpha\beta}^{\top}\end{bmatrix}^{\top}$.
\end{lemma}

{\em As a summary}, different plug-and-play properties have been investigated considered so far as key requirements for a networked viewpoint. The analysis of power flow at the switching as well as the filter node reveals the effect of power injection  on amplitude and frequency of AC quantities and takes into account the presence of RLC filter at the converter terminals reflected in the power balance equation. AC signals synchronize in frequency under arbitrary constant load.

Next, we investigated the passivity of the DC/AC converter in open-loop, which turns out to be passive with respect to DC and AC inputs. Passivity is preserved even in closed-loop fashion, i.e after introducing the matching control. Indeed, passivity with  with respect to the AC grid port serves as a decentralized stability certificate when interconnecting the inverter with a passive AC power grid model. Whereas, passivity with respect to the DC port $(i_{dc},v_{dc})$ serves as a main starting point for passivity-based control design, which will be discussed and analyzed via high-level control.

\chapter{Stability analysis of the DC/AC converter in the closed loop fashion}
\label{sec: stability anal}

This stability analysis is similar to the study of equilibria for a single generator investigated in the work of Caliskan and Tabuada in \cite{YS-PT:14}.

\section{Stability analysis of the closed-loop system in $(dq0)$ Frame} 

\begin{assumption}[Constant load impedance]
	In the following, the load is described by its constant impedance matrix $G_{load}\in\real^{2\times 2}$, with
	
	\begin{equation}
	G_{load}=\begin{bmatrix}
	g & 0 
	\\
	0 & g
	\end{bmatrix}
	\,,
	\label{eq:const-load}
	\end{equation}
	where $g>0$ represents the resistive load.
\end{assumption}

We showed previously in \ref{subsec: synchro-ac} that at steady state and for the constant load impedance matrix $G_{load}$ described in \eqref{eq:const-load}, all AC signals synchronize at the same nonzero frequency $\omega_s=\eta v_{dc,s}, \eta>0$.

After a transformation using the matrix $T_{dq}$ as defined in \eqref{eq: trafo to dq} using the angle of transformation $\gamma=\theta_v(t), t>0$ with $\dot\gamma=\eta v_{dc}$. The modulation signal $m_{\alpha\beta}$ can be described in $dq0$- frame by $m_{dq}$ defined as
\begin{equation*}
m_{dq}= \mu\, T_{dq}(\theta_v)\begin{bmatrix}
-\sin(\theta_v)\\ \cos(\theta_v)
\end{bmatrix}=\mu\begin{bmatrix}
0 \\ 1
\end{bmatrix}
\,.
\end{equation*}

The DC/AC converter can be expressed in $dq$- frame, after a transformation using the matrix $T_{dq}$ using the angle $\gamma(t)=\theta_v(t),\,t>0$ as the following: 
\begin{subequations}
	\begin{align}
	C_{dc} \dot v_{dc} &=-G_{dc} v_{dc} + i_{dc} - \frac{\mu}{2} \begin{bmatrix}
	0\\1 \end{bmatrix}^{\top} i_{dq}
	\\
	L \dot {i_{dq}} &=-(L\eta v_{dc} J_2 +R)\, i_{dq} + \frac{\mu}{2} \begin{bmatrix}
	0 \\ 1 
	\end{bmatrix} v_{dc} - v_{dq}\\
	C\, \dot v_{dq} &=-(C\eta v_{dc} J_2 + G_{load})\,v_{dq} + i_{dq}
	\,,
	\end{align}
	\label{eq: inv-dq}
\end{subequations}
with $\dot \gamma=\dot \theta_v=\eta v_{dc}$.

At steady state holds $\dot i_{dq}=\dot v_{dq}=0$ following from the Definition \ref{def:eq-dq}. Moreover, it holds for the steady state frequency $\dot\theta_s=\omega_s=\eta\, v_{dc,s}$ and we can express the system at steady state as

\begin{subequations}
	\begin{align}
	0 &=-G_{dc} v_{dc,s} + i_{dc} - \frac{\mu}{2} \begin{bmatrix}
	0\\1 \end{bmatrix}^{\top} i_{dq,s}
	\\
	0 &=-(L\omega_s J_2 +R) i_{dq,s} + \frac{\mu}{2} \begin{bmatrix}
	0 \\ 1 
	\end{bmatrix} v_{dc,s} - v_{dq,s}
	\\
	0 &=-(C\omega_s J_2 + G_{load})\,v_{dq,s}+i_{dq,s}
	\,,
	\end{align}
	\label{eq: ss-inv-dq}
\end{subequations}

where we perform a transformation using the matrix $T_{dq}$ with transformation angle $\gamma=\theta_v,\,\dot \theta_v=\eta v_{dc,s}$.

\subsubsection{Uniqueness of the equilibrium in $(dq0)$- frame}
	Note that we define a steady state of the DC/AC converter as a point in $\real^5$ in the rotating $dq0$- frame resulting from solving the equations \eqref{eq: ss-inv-dq}.
	Solving \eqref{eq: ss-inv-dq}  reveals that DC/AC converter possesses five equilibria, where one is uniquely real and all others are complex. 
	For the given choice of parameters and an arbitrary choice of the current source $i_{dc}$ and load impedance matrix $G_{load}\in\real^{2\times2}$, there is a unique voltage $v_{dc,s}\in\real$, inducing a unique frequency $\omega_s=\eta v_{dc,s}\in\real$ for the DC/AC converter at steady state. We further consider the unique real equilibrium $\begin{bmatrix} v_{dc,s} & v_{dq,s} & v_{dq,s}\end{bmatrix}\in\real^5$.
\vspace{2em}

We define the following error coordinates 
\begin{equation*}
\tilde v_{dc}=v_{dc}-v_{dc,s},\,\tilde i_{dq} = i_{dq} - i_{dq,s},\,\tilde v_{dq} = v_{dq} - v_{dq,s}
\,,
\end{equation*} 
and the corresponding state-error vector 
\begin{equation*}
e=\begin{bmatrix} \tilde v_{dc}\\ \tilde i_{dq}\\ \tilde v_{dq} \end{bmatrix}\in\real^5
\,.
\end{equation*}

By subtracting \eqref{eq: inv-dq} from \eqref{eq: ss-inv-dq}, we get the following equations in error coordinates of the inductor and the capacitor, where we define $\omega=\dot \gamma=\eta v_{dc}$ 
\begin{subequations}
	\begin{align*}
	L \dot{\tilde i}_{dq}
	&=-RI_2 \tilde i_{dq} - L\omega J_2 i_{dq}+L\omega_sJ_2 i_{dq,s}+\frac{\mu}{2} \begin{bmatrix}
	0 \\ 1 
	\end{bmatrix} \tilde v_{dc} + \tilde v_{dq}
	\\
	&=-RI_2 \tilde i_{dq} - L (\omega-\omega_s) J_2 i_{dq,s}-L\omega J_2 (i_{dq}-i_{dq,s})+\frac{\mu}{2} \begin{bmatrix}
	0 \\ 1 
	\end{bmatrix} \tilde v_{dc} + \tilde v_{dq}
	\\
	&=-(L\omega J_2 + RI_2) \tilde i_{dq} - L\eta\, \tilde{v}_{dc} J_2 i_{dq,s}+\frac{\mu}{2} \begin{bmatrix}
	0 \\ 1 
	\end{bmatrix} \tilde{v}_{dc} + \tilde{v}_{dq}
	\\
	C\, \dot{\tilde{v}}_{dq} &=-G_{load}\,\tilde {v}_{dq}-C\omega J_2v_{dq}+C\omega_s J_2 v_{dq,s}+i_{dq}
	\\
	&=-(C\omega J_2 + G_{load})\,\tilde {v}_{dq}-C \eta\, \tilde v_{dc} J_2\, v_{dq,s} + \tilde i_{dq}
	\,,
	\end{align*}
\end{subequations}

where we add and subtract $-L\omega J_2 i_{dq,s}$, respectively $-C\omega J_2 v_{dq,s}$ to get again the error coordinate $\tilde i_{dq}$, respectively $\tilde v_{dq}$ and use the fact that $\omega-\omega_s=\eta(v_{dc}-v_{dc,s})=\eta\, \tilde v_{dc}$.

The error dynamics of the DC-circuit are described by
\begin{equation*}
C_{dc}\dot{\tilde v}_{dc}=-G_{dc}\tilde v_{dc}-\frac{\mu}{2}\begin{bmatrix}
0\\ 1
\end{bmatrix}^{\top}\tilde i_{dq}
\,.
\end{equation*} 

As a summary, we can write the DC/AC converter in error dynamics as 
\begin{subequations}
	\begin{align}
	C_{dc}\dot{\tilde v}_{dc}&=-G_{dc}\tilde v_{dc}-\frac{\mu}{2}\begin{bmatrix}
	0\\ 1
	\end{bmatrix}^{\top}\tilde i_{dq}
	\\
	L \dot{\tilde i}_{dq} &=-(L\omega J_2 + RI_2) \tilde i_{dq} - L\eta\, \tilde{v}_{dc} J_2 i_{dq,s}+\frac{\mu}{2} \begin{bmatrix}
	0 \\ 1 
	\end{bmatrix} \tilde{v}_{dc} + \tilde{v}_{dq}
	\\
	C\, \dot{\tilde{v}}_{dq}&=-(C\omega J_2+G_{load})\,\tilde {v}_{dq}-C \eta\, \tilde v_{dc} J_2\, v_{dq,s} + \tilde i_{dq}
	\,.
	\end{align}
	\label{eq: sys-err-dq}
\end{subequations}

We are now ready to define the positive definite, radially unbounded Lyapunov candidate $\tilde W:\real^{5}\to\real$ by
\begin{equation*}
\tilde W=\frac{1}{2}C_{dc}\tilde v_{dc}^{2}+\frac{1}{2}C^{}\tilde v_{dq}^{\top}\tilde v_{dq}+\frac{1}{2}L\tilde i_{dq}^{\top}\tilde i_{dq}
\,,
\end{equation*}

and calculate its time derivative along the closed-loop trajectories of \eqref{eq: sys-err-dq}

\begin{subequations}
	\begin{align*}
	\dot{\tilde W}&=C_{dc}\tilde v_{dc}\dot{\tilde v}_{dc}+C^{}{\tilde v}_{dq}^{\top}\dot{\tilde v}_{dq}+L\tilde i_{dq}^{\top}\dot{\tilde i}_{dq}
	\\
	&=\tilde v_{dc}(-G_{dc}\tilde v_{dc}-\frac{\mu}{2}\begin{bmatrix}
	0\\ 1
	\end{bmatrix}^{\top}\tilde i_{dq})
	+\tilde v_{dq}^{\top}\left(-(C\omega J_2+G_{load})\,\tilde {v}_{dq}-C \eta\, \tilde v_{dc} J_2\, v_{dq,s} + \tilde i_{dq}\right)\\
	&+\tilde i_{dq}^{\top}\left(-(L\omega J_2+RI_2) \tilde i_{dq} - L\eta\, \tilde{v}_{dc} J_2 i_{dq,s}+\frac{\mu}{2} \begin{bmatrix}
	0 \\ 1 
	\end{bmatrix} \tilde{v}_{dc} + \tilde{v}_{dq}\right)
	\,.
	\end{align*}
\end{subequations}

We simplify the expression further to get the following quadratic form equation using the fact that $-\tilde i_{dq}^{\top}L\omega J_2\tilde i_{dq}=0$ and $-\tilde v_{dq}^{\top}L\omega J_2\tilde v_{dq}=0$

\begin{subequations}
	\begin{align*}
	\dot{\tilde W}
	&=-G_{dc}\tilde v_{dc}^{2} - \tilde v_{dq}^{\top} G_{load}\,\tilde {v}_{dq}-C \eta\,\tilde v_{dq}^{\top} \tilde v_{dc} J_2\, v_{dq,s}-\tilde i_{dq}^{\top}RI_2 \tilde i_{dq} - L\eta\,\tilde i_{dq}^{\top} \tilde{v}_{dc} J_2 i_{dq,s}
	\\
	&=\begin{bmatrix}
	\tilde v_{dc}\\ \tilde v_{dq} \\ \tilde i_{dq}
	\end{bmatrix}^{\top}
	\,
	\begin{bmatrix}
	-G_{dc} & -\frac{C\eta}{2}(J_2v_{dq,s})^{\top} & -\frac{L\eta}{2}(J_2i_{dq,s})^{\top}
	\\
	-\frac{C\eta}{2}J_2v_{dq,s} & -G_{load} & 0\\
	-\frac{L\eta}{2}J_2i_{dq,s} & 0 & -RI_2
	\end{bmatrix}
	\,
	\begin{bmatrix}
	\tilde v_{dc}\\ \tilde v_{dq} \\ \tilde i_{dq}
	\end{bmatrix}\\
	&= e^{\top} P e < 0
	\,.
	\end{align*}
\end{subequations}
The matrix $P\in\real^{5\times 5}$ is negative definite under the necessary condition corresponding to the chosen Lyapunov function that 
\begin{equation}
{RC^2\eta^2\norm{v_{dq}}_2^2+g L^2\eta^2\norm{i_{dq}}_2^2< 4RG_{dc} g}
\,.
\label{eq: neg-def-cond}
\end{equation}
The necessary condition described in \eqref{eq: neg-def-cond} is derived from evaluating the principal minors of the matrix $P$ and setting the block-wise necessary conditions for its negative definiteness.  

If \eqref{eq: neg-def-cond} is satisfied, the system states converges to the set of equilibria S and the origin is globally asymptotically stable for the error system .

We deduce that the DC/AC converter as described in \eqref{eq: inv-dq} converges to S, where S defines a set of globally and  asymptotically stable equilibria under the sufficient condition \eqref{eq: neg-def-cond} in the rotating frame $dq0$, corresponding to the following steady state locus $S$
\begin{equation*}
S=\{e\in\real^3|e=0\}=\left\{\begin{bmatrix}
v_{dc}\\ i_{\alpha\beta}\\ v_{\alpha\beta}
\end{bmatrix}\in\real^{5},\dot v_{dc,s}=0, \dot v_{\alpha\beta,s}=J_2\omega_sv_{\alpha\beta,s}, \dot i_{\alpha\beta,s}=J_2\omega_s i_{\alpha\beta,s} \right\}
\,.
\end{equation*} 

\begin{remark}[Condition for convergence]
	Generally, the condition \eqref{eq: neg-def-cond} is only sufficient and not necessary for global asymptotic convergence of the DC/AC converter to the steady state locus, since it depends on the choice of the Lyapunov function $\tilde W$. 
\end{remark}

\section{Stability analysis using internal model principle} 

We consider the $(\alpha\beta)$- frame again in this section. We drop the $(\alpha\beta)$ index for AC signals.
\begin{assumption}
	The transients of the DC circuit are ignored. The DC circuit is assumed to be at steady state, i.e $\dot v_{dc}=0, v_{dc}=v_{dc,s}$. 
\end{assumption}
The voltage at the output of the modulation block is defined by
\begin{equation*}
v_x=\frac{1}{2} m v_{dc,s}
\,.
\end{equation*}
It yields for the dynamics of the matching controller that: 
\begin{equation*}
\dot m= \eta v_{dc,s} J_2 m
\,,
\end{equation*} 
exhibiting the harmonic oscillations synchronous at $\omega_s=\eta v_{dc,s}$.
As a consequence, the dynamics of voltage of the modulation block can be written as
\begin{equation*}
\dot v_x=\frac{1}{2} v_{dc,s} \dot m=  \eta v_{dc,s}  J_2 v_x
\end{equation*}

We consider the following AC system with input $v_x$ 
\begin{subequations}
	\begin{align}
	L\dot i &=-R i+v_x-v
	\\
	C\dot v &=-G_{load} v + i
	\\
	\dot v_x &= \eta v_{dc,s}  J_2 v_x
	\,.
	\end{align}
	\label{eq: red-ss}
\end{subequations}

An interpretation of the system described in \eqref{eq: red-ss} is that of an exogenous system resulting from merging the DC circuit with the modulation block, which exhibits harmonics and drives the AC circuit as depicted in \ref{fig: ORP-sys}.

\begin{figure}[h!]
	\centering
		\begin{tikzpicture}[auto, node distance=3cm,>=latex']
		\node [input, name=input] {};
		\node [block, node distance=4cm] (controller) {
			$\begin{matrix}
			L \dot i =-R i_{dq} + v_x - v \\
			C \dot v =-G_{load}v + i
			\end{matrix}$
		};
		\node [smallblock, left of=controller, node distance=5cm] (gain1) 
		{$\dot v_{x}=\eta v_{dc,s}J_2 v_x$};
		\draw [->] (gain1.east) |- (controller);
		\end{tikzpicture}
	\caption{Harmonics described by the dynamics of $v_x$ at steady state driving the AC dynamics}
	\label{fig: ORP-sys}
\end{figure}
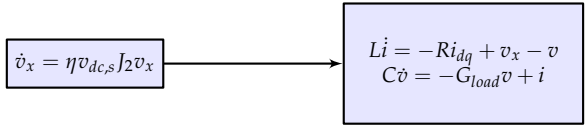

We now formulate a state space representation of the DC/AC converter
\begin{equation*}
\begin{bmatrix}
\dot i \\ \dot v \\ \dot v_x
\end{bmatrix}=\begin{bmatrix}
\frac{-R}{L}I_2 & -\frac{1}{L}I_2 & \frac{1}{L}I_2 \\
\frac{1}{C}I_2 & -\frac{1}{C}G_{load} & 0 \\
0 & 0 & J_2 \omega_s
\end{bmatrix}
\begin{bmatrix}
i \\ v\\ v_x
\end{bmatrix}
\,.
\end{equation*}

By introducing $x=[i,v]^{\top}\in\real^2, u=v_x$, we end up with 

\begin{equation*}
\dot x=Ax+Du,\; \dot u=Su
\,,
\end{equation*}
where
\begin{equation*}
A= \begin{bmatrix}
\frac{-R}{L}I_2 & -\frac{1}{L}I_2 \\
\frac{1}{C}I_2 & -\frac{1}{C}G_{load} 
\end{bmatrix}
,\,
D=\begin{bmatrix}
\frac{I_2}{L} \\ 0
\end{bmatrix}
,\, S=J_2\omega_s
\,.
\end{equation*}

By showing that the eigenvalues of A are negative, we can show that the system matrix is Hurwitz and conclude from this fact that 
\begin{equation*}
\sigma(A)\cap \sigma (S)=\emptyset
\,.
\end{equation*} 

According to the internal model principle, $\exists F\in\real^{4\times4}$ satisfying the Sylvester equation

\begin{equation}
AF-FS=-D
\,.
\label{eq: syl-eq}
\end{equation}

We now identify the matrix $F$ by considering the driven AC system at steady state described by
\begin{subequations}
	\begin{align*}
	\dot i_s &=J_2\omega_si_s=\frac{1}{L}(-R i_s+v_x-v_s)
	\\
	\dot v_s &=J_2\omega_s v_s=\frac{1}{C}(-G_{load}v_s+i_s)
	\,,
	\end{align*}
\end{subequations}

and we get 
\begin{subequations}
	\begin{align*}
	i_s &=(L\omega_sJ_2+RI_2+(C\omega_sJ_2+G_{Load})\inv)\inv v_x
	\\
	&=(L\omega_sJ_2+RI_2+N^{-1})\inv v_x
	\\
	&=M\inv v_x
	\\
	v_s &=(C\omega_sJ_2+G_{load})\inv i_s
	\\
	&=N\inv i_s
	\\
	&=(MN)\inv v_x
	\,,
	\end{align*}
\end{subequations}

where $M,N\in\real^{2\times2}$.

As a summary we derive for the steady state locus of AC circuit 
\begin{equation}
\epsilon=\{x\in\real^2, x=Fu,\, F=\begin{bmatrix}
N\inv \\ (MN)\inv
\end{bmatrix}\}
\,.
\label{eq: ss-manifold}
\end{equation}

We now prove global asymptotic convergence to this manifold by defining the error coordinates 
$\delta=x-Fu$. Looking to its dynamics
\begin{subequations}
	\begin{align*}	
	\dot \delta &=\dot x-F\dot u
	\\
	&=A(x-Fu)+(AF+D-FS)u
	\\
	&= A \delta
	\,,
	\end{align*}
\end{subequations}
where we take into account that $F$ satisfies the Sylvester equation described in \eqref{eq: syl-eq}. 
This shows that the manifold described in \eqref{eq: ss-manifold} is globally asymptotically stable, since A is Hurwitz.
Global statement follows from choosing the Lyapunov candidate $V(\delta)=\frac{1}{2}\delta^2$.

{\em As a summary}, we transformed DC/AC converter dynamics into $dq0$- frame, considered so far as a usual working frame for synchronous machines. We conducted our stability analysis by considering error dynamics of the DC/AC converter and defining an appropriate Lyapunov function. The convergence to the set of equilibria is guaranteed under sufficient conditions.
Under the assumption of no transients of DC circuit dynamics, one can apply the internal model principle by presenting the DC/AC converter as a system where the DC circuit together with the modulation block is an exogenous system driving the AC circuit.

\chapter{High-level control architectures}

\label{sec: high-level-ctrl}
Our matching controller can be regarded as an inner loop that structurally equivalences a converter and a SM model. Based on this inner loop, further outer-loop controls can be added, e.g the equivalent of PSS or governor control to regulate frequencies and to tightly control currents, or to induce extra inertia and damping in the system. We dedicate this section to exploit the degrees of freedom reflected in the design of $i_{dc},\,\mu,\,\eta$ considered so far as a constant in our control approach.

\section{Amplitude tracking} 

\begin{assumption}[Non-zero current amplitude $\hat {v}_\ell$]
We assume that the load does not render the inductor current zero. In the case of a constant impedance load as defined in \eqref{eq: load-imp} we exclude the purely critical load, where $g=0$ and $b=b_{crit}$.
\end{assumption}

We propose in this section to design a controller which is able to asymptotically track a desired value of the current amplitude $\hat {v}_{\ell,ref}$. The reference for the current amplitude can be generated for example from an upper controller which tracks a given amplitude of the capacitor voltage $\hat {v}_{ref}$.  

We first consider the inductance equation in open-loop defined as follows

\begin{equation*}
L \dot{i}_{\alpha\beta}=-Ri_{\alpha\beta}+v_x-v_{\alpha\beta}
\,.
\end{equation*}

By defining AC signals in polar coordinates 

\begin{equation}
i_{\alpha\beta}=\hat {v}_\ell\begin{bmatrix}
-\sin(\theta_\ell) \\ \cos(\theta_\ell)
\end{bmatrix}
,\,
v_{\alpha\beta}=\hat {v}\begin{bmatrix}
-\sin(\theta) \\ \cos(\theta)
\end{bmatrix}
,\,
v_x=\hat {v}_x\begin{bmatrix}
-\sin(\theta_x) \\ \cos(\theta_x)
\end{bmatrix}
\,.
\label{eq: ac-polar}
\end{equation}
We apply \eqref{eq: ac-polar} to the inductor equation and by defining $\hat {v}_{\ell}$ as 

\begin{equation*}
\hat {v}_{\ell}=\sqrt{i_{\alpha\beta}^{\top}i_{\alpha\beta}}
\,,
\end{equation*}

and multiplying with the vector $i_{\alpha\beta}^{\top}$, we have
\begin{subequations}
\begin{align*}
L i_{\alpha\beta}^{\top}\dot i_{\alpha\beta} &=-Ri_{\alpha\beta}^{\top}i_{\alpha\beta}+i_{\alpha\beta}^{\top} v_x-i_{\alpha\beta}^{\top} v_{\alpha\beta} \\
L \hat {v}_{\ell} \dot {\hat {v}}_{\ell} &=-R\hat {v}_{\ell}^2 + \hat {v}_{\ell} \hat {v}_x \cos(\theta_x-\theta_\ell)-\hat {v}_\ell \hat {v} \cos(\theta-\theta_\ell)
\,,
\end{align*}
\end{subequations}
where we use the fact that $ \hat {v}_\ell \dot \hat {v}_\ell=i_{\alpha\beta}^{\top} \dot i_{\alpha\beta}$. Finally, we arrive at 
\begin{equation}
L\dot {\hat {v}}_\ell=-R\hat {v}_\ell+\hat {v}_x \cos(\theta_x-\theta_\ell)-\hat v\cos(\theta-\theta_\ell)
\,.
\label{eq:contr-il}
\end{equation}

In order to track the reference current amplitude $\hat {v}_{\ell, ref}$, we choose to place the poles of the closed loop system as the following
\begin{subequations}
\begin{align*}
L\dot {\hat {v}}_\ell &=K_p(\hat {v}_{\ell, ref} - \hat {v}_{\ell})+K_i\int(\hat {v}_{\ell, ref} - \hat {v}_{\ell})dr 
\\
\dot e_l &=\frac{-K_p}{L} e_l + \frac{-K_i}{L}\int e_l\, dr=-\lambda_1 e_l-\lambda_2 \int e_l 
\,,
\end{align*}
\end{subequations}
where $\lambda_1, \lambda_2\in\real$.
\\
This induces the following choice of the control input $\hat {v}_{x,ref}$ using feedback linearization as follows

\begin{equation*}
{\hat {v}_{x,ref}= \frac{-K_p e_\ell-K_i\int e_\ell+\hat v\cos(\theta_c-\theta_\ell)+R \hat {v}_\ell}{\cos(\theta_x-\theta_\ell)}}
\,,
\end{equation*}

which is a well-defined reference due to
\begin{equation*}
\cos(\theta_x-\theta_\ell)=R\frac{\hat {v}_\ell}{\hat {v}_x}\neq 0
\,.
\end{equation*} 

A globally defined and smooth version of this reference would be 
\begin{equation*}
{\hat {v}_{x,ref}= -K_pe_\ell-K_i\int e_\ell + \hat v}
\,.
\end{equation*}

In order to track the reference $\hat {v}_{x,ref}$, we design the controller of the amplitude $\dot{\mu}$ as follows based on the definition of $\hat {v}_x$

\begin{subequations}
\begin{align*}
\hat {v}_x &=\frac{1}{2} \mu v_{dc}
\\
\dot {\hat {v}}_x &=\frac{1}{2}\dot \mu v_{dc}+\frac{1}{2}\mu \dot v_{dc}
\,. 
\end{align*}
\end{subequations}

In closed loop, we would like to have the following
\begin{subequations}
\begin{align*}
\dot {\hat {v}}_x &=\lambda_x ({\hat {v}}_{x,ref}- {\hat {v}}_x)
\\
\dot e_x &= -\lambda_x e_x
\,.
\end{align*}
\end{subequations}
where $\lambda_x>0$.
Using exact feedback linearization, we define the dynamics of the gain $\mu$ as follows

\begin{subequations}
\begin{align*}
\dot \mu&=\frac{-2 \lambda_x}{v_{dc}} e_x -\frac{1}{2} \mu \dot v_{dc} 
\\
\dot \mu&=\frac{-2 \lambda_x}{v_{dc}} e_x + \frac{1}{2v_{dc}}\mu^2 \cos(\theta_l-\theta_m)-\mu \frac{i_{dc}^*}{v_{dc}} + \mu \frac{G_{dc}}{v_{dc}}
\,.
\end{align*}
\end{subequations}	 
A simpler well-defined version of this control law is

\begin{equation*}
{\dot \mu=\frac{-2 \lambda_x}{v_{dc}} e_x=\frac{-K_x}{v_{dc}} e_x}
\,.
\end{equation*}

such that $\mu\in[0,1]$.

In summary, we consider the following closed-loop system for tracking a desired current amplitude $\hat {v}_{\ell,ref}$ as in Figure \ref{fig: controller for r_l}

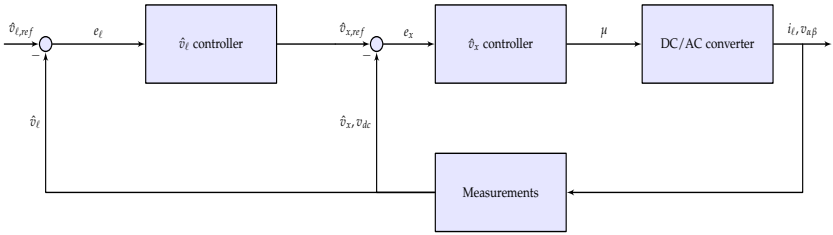
\begin{figure}[h!]
\begin{center}
\resizebox{11cm}{3cm}{
		\begin{tikzpicture}[auto, node distance=3cm,>=latex']
		\node [input, name=input] {};
		\node [sum, right of=input] (sum) {};
		\node [block, right of=sum] (controller) {$\hat {v}_{x}$ controller};
		\node [block, right of=controller, node distance=5cm] (system) {DC/AC converter};
		\node [block, left of=sum, node distance=4cm] (controller2) {$\hat {v}_{\ell}$ controller};
		\node [sum, left of=controller2, node distance=4cm] (sum2) {};
		\node [input, name=input2 , left of=sum2, node distance=1cm] {};
		
		\draw [->] (controller) -- node[name=u] {$\mu$} (system);
		\node [output, right of=system] (output) {};
		\node [block, below of=controller, node distance=3cm] (measurements) {
			Measurements
		};
		
		\draw [draw,->] (input) -- node {$\hat {v}_{x,ref}$} (sum);
		\draw [draw,->] (input2) -- node {$\hat {v}_{\ell,ref}$} (sum2);
		\draw (controller2) -- (input);
		\draw [->] (sum) -- node {$e_x$} (controller);
		\draw [->] (sum2) -- node {$e_\ell$} (controller2);
		\draw [->] (system) -- node [name=y] {$i_\ell, v_{\alpha\beta}$}(output);
		\draw [->] (y) |- (measurements);
		\draw [->] (measurements) -| node[pos=0.99] {$-$} 
		node [near end] {$\hat {v}_x, v_{dc}$} (sum);
		\draw [->] (measurements) -| node[pos=0.99] {$-$} 
		node [near end] {$\hat {v}_{\ell}$} (sum2);
		\end{tikzpicture}}
\caption{Control architecture for tracking a current  amplitude reference $\hat {v}_{\ell,ref}$}
\label{fig: controller for r_l}
\end{center}
\end{figure}

\begin{remark}[Controllability of DC/AC converter and time scale separation]
We consider again the equation \eqref{eq:contr-il} and check the current amplitude for controllability. Due to the fact that 
\begin{equation*}
\cos(\theta_x-\theta_\ell)=R\frac{\hat {v}_\ell}{\hat {v}_x}\neq 0
\,.
\end{equation*} 
The current amplitude $\hat {v}_\ell$ is controllable via setting the amplitude of the voltage at the output of the modulation block $\hat {v}_x$ to a given reference $\hat {v}_{x,ref}$. 
In order for $\hat {v}_x$ to track $\hat {v}_{x,ref}$, we make use of the input $1>\mu>0$. This assumes there is a times scale separation between the $\hat {v}_\ell$- controller and the $\hat {v}_x$-controller.

It yields for $\mu$-controller that $v_{dc}(t)\neq 0, \forall t>0$, at the time when the $\mu$- controller is acting on the DC/AC converter, which implies in return that there exists a time scale separation of the $\mu$- controller and the DC/AC converter. The $\mu$- controller is designed to be slow enough in comparison to the DC/AC converter dynamics in order to fulfill our control objectives.  
\end{remark}

\begin{remark}[Placement of closed-loop poles and closed-loop stability]
We choose the pole for the Proportional (P) and Integral (I) parts for the outer controller as follows,
\begin{equation*}
\ddot e_\ell=-\lambda_1 \dot e_\ell- \lambda_2 e_\ell
\,.
\end{equation*}	
We choose $\lambda_1,\,\lambda_2>0$ to satisfy critical damping such that the closed loop system has a double eigenvalue at $\lambda_0<0$ with 
\begin{equation*}
\lambda_0=\frac{-\lambda_2\pm\sqrt{\lambda_1^2 - 4\lambda_2}}{2}<0
\,.
\end{equation*}	
Using $\lambda_1^2=4\lambda_2$, we have
\begin{subequations}
\begin{align*}
\lambda_1 &=-2\lambda_0
\\
\lambda_2 &=\lambda_0^2
\,.
\end{align*}
\end{subequations}
We place the closed-loop poles such that the inner-controller responsible for tracking $r_{x, ref}$ is at least ten times faster than the outer controller tracking the desired current amplitude $r_{l,ref}$.
\begin{equation*}
|\lambda_x|> 10|\lambda_0|
\,.
\end{equation*}
\end{remark}

\subsubsection{Tracking a given reference for a voltage capacitor amplitude $\hat {v}_{ref}$}
We assume that the DC/AC converter is interfaced with a resistive load with admittance $g>0$. This allows to deduce that 
\begin{equation*}
cos(\theta_\ell-\theta)=g\frac{\hat v}{\hat {v}_\ell}
\,,
\end{equation*}
with $\hat v,\, \hat {v}_\ell>0$.
Given a reference amplitude for the capacitor voltage $\hat {v}_{ref}$, our aim is to design a controller able to track this given reference.

For this purpose, we rewrite the capacitor equation as follows
\begin{equation*}
C\dot v_{\alpha\beta}=-i_{load} + i = K_{pc} (\hat v-\hat v_{ ref})+ K_{ci} \int (\hat v-\hat {v}_{ref})
\,.
\end{equation*}

We define $\dot {\hat v} \hat v=v_{\alpha\beta}^{\top}\dot v_{\alpha\beta}$ to get
\begin{equation*}
C\dot {\hat v}=-\hat {v}_{load} \cos(\theta-\theta_{load})+ \hat {v}_\ell \cos(\theta-\theta_\ell)
\,,
\end{equation*}
where $i_{load}=\hat {v}_{load}\begin{bmatrix}
-\sin(\theta_{load}) & \cos(\theta_{load})
\end{bmatrix}^{\top}$.
\\
The error dynamics of the closed-loop system in function of the error $e_c=\hat {v}_{ref} - \hat{v}$ can be expressed as

\begin{equation*}
\dot e_c=\frac{-K_{cp}}{C} e_c+ \frac{-K_{ci}}{C} \int e_c=-\lambda_{c1}\dot e_c-\lambda_{c2} e_c
\,.
\end{equation*}

The gains $K_{cp}>0$ and $K_{ci}>0$ can be chosen analog to the current amplitude controller.
In order for the amplitude $\hat v$ to follow the desired reference $\hat {v}_{ref}$, we set the following desired amplitude $\hat v_{\ell,ref}$ defined as
\begin{equation*}
\hat v_{\ell,ref}=\frac{\hat {v}_{load}\cos(\theta-\theta_{load})-K_{cp}e_c-K_{ci}\int e_c}{\cos(\theta_\ell-\theta)}
\,.
\end{equation*}

A smooth version of this reference is the following: 
\begin{equation*}
{\hat {v}_{\ell,ref}=\hat {v}_{load}-K_{cp} e_c -K_{ci}\int e_c}
\,.
\end{equation*}
We can now track this given reference using the previously described controller cascade for tracking a given inductance amplitude current.
The control architecture can be explained by the Figure \ref{fig: controller-for-r_c}.

\begin{figure}[h!]
\centering
\resizebox{10cm}{2cm}{
\begin{tikzpicture}[auto, node distance=3cm,>=latex']
\node [input, name=input] {};
\node [sum, right of=input] (sum) {};
\node [block, right of=sum] (controller) {$\hat {v}_x$-controller};
\node [block, right of=controller, node distance=5cm] (system) {DC/AC converter};
\node [block, left of=sum] (controller2) {$\hat {v}_\ell$-controller};
\node [block, left of=sum2] (controller3) {$\hat v$-controller};
\node [sum, left of=controller2, node distance=3cm] (sum2) {};
\node [sum, left of=controller3, node distance=3cm] (sum3) {};
\node [input, name=input2 , left of=sum2, node distance=1cm] {};
\node [input, name=input3 , left of=sum3, node distance=1cm] {};
\node [sum, left of=controller3, node distance=3cm] (sum3) {};
\node [input, name=input3 , left of=sum3, node distance=1cm] {};
\draw [->] (controller) -- node[name=u] {$\mu$} (system);
\node [output, right of=system] (output) {};
\node [block, below of=u, node distance=3cm] (measurements) {Measurements};

\draw [draw,->] (input) -- node {$\hat {v}_{x,ref}$} (sum);
\draw (controller3) -- (input2);
\draw [draw,->] (input2) -- node {$\hat {v}_{\ell,ref}$} (sum2);
\draw [draw,->] (input3) -- node {$\hat {v}_{ref}$} (sum3);
\draw [->] (sum) -- node {$e_x$} (controller);
\draw [->] (sum2) -- node {$e_\ell$} (controller2);
\draw [->] (sum3) -- node {$e_c$} (controller3);
\draw [->] (system) -- node [name=y] {$i_l, v_{\alpha\beta}$}(output);
\draw [->] (y) |- (measurements);
\draw [->] (measurements) -| node[pos=0.99] {$-$} 
node [near end] {$\hat {v}_x, v_{dc}$} (sum);

\draw [->] (measurements) -| node[pos=0.99] {$-$} 
node [near end] {$\hat {v}_\ell$} (sum2);
\draw [->] (measurements) -| node[pos=0.99] {$-$} 
node [near end] {$\hat v$} (sum3);
\end{tikzpicture}}
\caption{Control architecture for tracking a desired amplitude $\hat {v}_{ref}$}
\label{fig: controller-for-r_c}
\end{figure}

\subsubsection{Simulation results}
We choose the following parameter values for the cascaded controllers in order to follow a given reference for the current amplitude $\hat {v}_{\ell,ref}=20 A$.
\begin{equation*}
K_{p}=0.15 H/s,\,K_{i}=11.25H/s,\, K_x=2\cdot 10^5 H/s,\,\lambda_x= 10^5 s^{-1},\,\lambda_0=-100 s^{-1}
\,.
\end{equation*}

\begin{figure}[h!]
\centering
\includegraphics[scale=0.3]{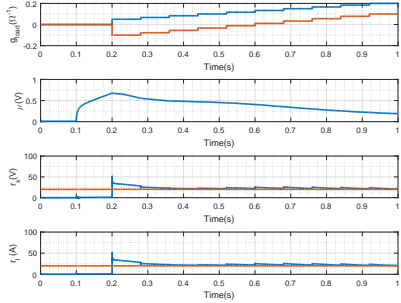}
\caption{Time domain simulation for tracking a given inductance current amplitude with increasing steps in the purely resistive and reactive load starting from $t=0.3s$.}
\label{fig: mu-ctrl-resis}
\end{figure}

\section{Frequency tracking}
\label{sec:controller-for-frequency-tracking} 

In the following, we present two different approaches to track a given reference frequency $\omega_{ref}$. The first is based on linear PID control inspired by the passivity analysis conducted earlier in this work and drawing upon governor control in synchronous machines. This passive control method preserves stability in closed-loop fashion. Second, we use nonlinear control tools like feedback linearization to track a desired frequency by using the frequency gain $\eta$ considered so far as constant to be specified.

\subsection{Frequency tracking using the DC current source}

By taking a closer look into the dynamics of the DC circuit we have the following                                                                                           
\begin{subequations}
\begin{align*}
C_{dc}\dot v_{dc}=-G_{dc}v_{dc}+i_{dc}-i_x
\,.
\end{align*}
\end{subequations}

Our aim is to design a controller using the DC circuit input $i_{dc}$ which fulfills the objective of tracking of a desired frequency $\omega_{ref}$ reduced to tracking of a desired voltage via $v_{dc, ref}=\eta^{-1}\omega_{ref}$, which can be formulated as follow
\begin{subequations}
\begin{align*}
C_{dc}\dot v_{dc} &=-(G_{dc}+K_{p,dc}) (v_{dc,ref}-v_{dc})- K_{i, dc}\int (v_{dc,ref}-v_{dc})-K_{d,ic}\dot e_{dc}
\\
(C_{dc}+K_{d,ic})\dot e_{dc} &=-(G_{dc}+K_{p,dc})e_{dc} - K_{i, dc}\int e_{dc}
\,,
\end{align*}
\end{subequations}
with $e_{dc}=v_{dc,ref}-v_{dc}$. 
As a consequence, an adequate choice of $i_{dc}$ would be
\begin{equation*}
i_{dc}=-K_{p,dc}e_{dc}-K_{i, dc}\int e_{dc}-K_{d,ic}\dot e_{dc}
\,.
\end{equation*}

The dynamics of the closed-loop system can be expressed by

\begin{equation*}
\dot e_{dc}=\frac{-(G_{dc}+K_{p,dc})}{C_{dc}+K_{d,ic}} e_{dc}+\frac{-K_{i,dc}}{C_{dc}+K_{d,ic}} \int e_{dc}=-\gamma_1 e_{dc}-\gamma_2\int e_{dc}
\,,
\end{equation*}
with $\gamma_1,\gamma_2>0$.

We choose the closed-loop eigenvalues such that the DC circuit is asymptotically stable with double negative eigenvalues by $\gamma_0<0$ in closed-loop fashion 
\begin{subequations}
\begin{align*}
\gamma_1 &=-2 \gamma_0
\\
\gamma_2 &=\gamma_0^2
\,.
\end{align*}
\end{subequations}
\begin{proposition}[Frequency control via the current source $i_{dc}$]
	Consider the DC/AC converter as described in \eqref{eq: inverter dynamics}. 
	For a given frequency $\omega_{ref}=\eta^{-1}v_{dc,ref}$, we consider the tracking  PID controller by making the intuitive choice of the current source $i_{dc}$ considered as constant so far as
	\begin{equation}
	i_{dc}=-K_{p,dc}e_{dc}-K_{i, dc}\int e_{dc}-K_{d,ic}\dot e_{dc}
	\,,
	\label{eq: idc-ctrl}
	\end{equation}
	with $K_{p,dc},\, K_{i, dc},\, K_{d,dc},\,\eta>0$.
	The DC/AC converter is asymptotically stable using this frequency controller and converges at steady state to the desired frequency.
\end{proposition}
\begin{proof}
We consider the DC/AC converter as described in $dq0$- frame \eqref{eq: inv-dq} after performing a transformation $T_{dq}$ with an angle $\gamma=\theta_v$ as defined in \eqref{eq: trafo to dq}. Define the DC/AC converter in error coordinates as follows with 
\begin{equation*}
\tilde v_{dc}=-e_{dc}=v_{dc}-v_{dc,ref}\,,\tilde i_{dq} = i_{dq} - i_{dq,s}\,,\tilde v_{dq} = v_{dq} - v_{dq,s}
\,.
\end{equation*} 
We now plug in \eqref{eq: idc-ctrl} such that 

\begin{subequations}
\begin{align*}
C_{dc} \dot v_{dc} &=-G_{dc} v_{dc}-K_{p,dc}\tilde v_{dc} -K_{i, dc}\int \tilde v_{dc} -K_{d,ic}\dot {\tilde v}_{dc} - \frac{\mu}{2} \begin{bmatrix}
0\\1 \end{bmatrix}^{\top} i_{dq}
\\
L \dot {i}_{dq} &=-(L\omega J_2 +R)\, i_{dq} + \frac{\mu}{2} \begin{bmatrix}
0 \\ 1 
\end{bmatrix} v_{dc} - v_{dq}\\
C\, \dot v_{dq} &=-(C\omega J_2 + G_{load})\,\tilde {v}_{dq}-C \eta\, \tilde v_{dc} J_2\, v_{dq,s} + \tilde i_{dq}
\,,
\end{align*}
\end{subequations}
with $\omega=\eta v_{dc}$.

We have at steady state after performing a transformation to $dq0$- frame with the angle $\omega_s=\dot\gamma=\dot\theta_v=\eta v_{dc,ref}, \eta>0$
\begin{subequations}
\begin{align*}
0 &=-G_{dc} v_{dc,ref} - \frac{\mu}{2} \begin{bmatrix}
0\\1 \end{bmatrix}^{\top} i_{dq,s}
\\
0 &=-(L\omega_s J_2 +R) i_{dq,s} + \frac{\mu}{2} \begin{bmatrix}
0 \\ 1 
\end{bmatrix} v_{dc,s} - v_{dq,s}
\\
0 &=-(C\omega_s J_2 + G_{load})\,v_{dq,s}+i_{dq,s}
\,.
\end{align*}
\end{subequations}

The system can be expressed in error coordinates by introducing the state $\tilde\xi=\int \tilde {v}_{dc}$ as

\begin{subequations}
\begin{align}
(C_{dc}+K_{d,dc}) \dot {\tilde v}_{dc} &=-(G_{dc}+K_{p,dc}) \tilde v_{dc} - K_{i, dc}\tilde \xi+ \frac{\mu}{2} 
\begin{bmatrix}
0\\ 1 
\end{bmatrix}^{\top} \tilde i_{dq}
\\
L \dot {\tilde i}_{dq} &=-(L\omega J_2 + RI_2) \tilde i_{dq} - L\eta\, \tilde{v}_{dc} J_2 i_{dq,s}+\frac{\mu}{2} \begin{bmatrix}
0 \\ 1 
\end{bmatrix} \tilde{v}_{dc} + \tilde{v}_{dq} \\
C\, \dot {\tilde v}_{dq} &=-(C\omega J_2 + G_{load})\,\tilde v_{dq}+ \tilde i_{dq}
\\
\dot {\tilde \xi} &=\tilde v_{dc}
\,.
\end{align}
\label{eq: err-idc-ctrl}
\end{subequations}

We are now ready to define the positive definite, radially unbounded Lyapunov candidate with $\tilde W:\real^{6}\to\real$ by
\begin{equation*}
\tilde W=\frac{1}{2}(C_{dc}+K_{d,dc})\tilde v_{dc}^{2}+\frac{1}{2}K_{i,dc}\tilde \xi^{\top}\tilde \xi+\frac{1}{2}C^{}\tilde v_{dq}^{\top}\tilde v_{dq}+\frac{1}{2}L\tilde i_{dq}^{\top}\tilde i_{dq}
\,,
\end{equation*}
where $\dot {\tilde \xi}=\tilde v_{dc}$ and calculate its time derivative along the closed-loop trajectories of \eqref{eq: err-idc-ctrl}.
We end up with 
\begin{subequations}
\begin{align*}
\dot{\tilde W}
&=-(G_{dc}+K_{p,dc})\tilde v_{dc}^{2}-K_{i, dc}\tilde\xi \tilde v_{dc} + K_{i, dc}\tilde\xi \tilde v_{dc}-\tilde v_{dq}^{\top} G_{load}\,\tilde {v}_{dq}-C \eta\,\tilde v_{dq}^{\top} \tilde v_{dc} J_2\, v_{dq,s}
\\
& -\tilde i_{dq}^{\top}RI_2 \tilde i_{dq} - L\eta\,\tilde i_{dq}^{\top} \tilde{v}_{dc} J_2 i_{dq,s}
\\
&=-(G_{dc}+K_{p,dc})\tilde v_{dc}^{2} - \tilde v_{dq}^{\top} G_{load}\,\tilde {v}_{dq}-C \eta\,\tilde v_{dq}^{\top} \tilde v_{dc} J_2\, v_{dq,s}-\tilde i_{dq}^{\top}RI_2 \tilde i_{dq} - L\eta\,\tilde i_{dq}^{\top} \tilde{v}_{dc} J_2 i_{dq,s}
\\
&=\begin{bmatrix}
\tilde v_{dc}\\ \tilde v_{dq} \\ \tilde i_{dq} \\ \tilde \xi
\end{bmatrix}^{\top}
\,
\begin{bmatrix}
-(G_{dc}+K_{p,dc}) & -\frac{C\eta}{2}(J_2v_{dq,s})^{\top} & -\frac{L\eta}{2}(J_2i_{dq,s})^{\top} & 0 
\\
-\frac{C\eta}{2}J_2v_{dq,s} & -G_{load} & 0 & 0
\\
-\frac{L\eta}{2}J_2i_{dq,s} & 0 & -RI_2 & 0
\\	
0 & 0 & 0 & 0 	
\end{bmatrix}
\,
\begin{bmatrix}
\tilde v_{dc}\\ \tilde v_{dq} \\ \tilde i_{dq} \\ \tilde\xi
\end{bmatrix}\\
&= e^{\top} Q e \leq 0
\,.
\end{align*}
\end{subequations}

The matrix $Q\in\real^{6\times 6}$ is negative semi-definite  for the slightly modified condition found in \eqref{eq: neg-def-cond} with
\begin{equation*}
{RC^2\eta^2\norm{v_{dq}}_2^2+g L^2\eta^2\norm{i_{dq}}_2^2< 4R(G_{dc}+K_{p,dc})g}
\,.
\end{equation*}
\end{proof}

\subsubsection{Toward providing Additional inertia to the system}
	
By looking into the closed-loop system after introducing the control law of $i_{dc}$ as in \eqref{eq: idc-ctrl}, it holds for the capacitor dynamics
	
	\begin{equation*}
	C_{dc}\dot{\tilde {v}}_{dc}=-G_{dc} v_{dc}-K_{p,dc}\tilde {v}_{dc}-K_{i,dc} \int \tilde {v}_{dc} - K_d \dot {\tilde {v}}_{dc}-\frac{1}{2} i_{\alpha\beta}^{\top} m_{\alpha\beta}
	\,. 
	\end{equation*}	
	
	Now we multiply with $\eta>0$ to yield ${\tilde\omega}_v=\eta \tilde v_{dc}$ and we have for the frequency error dynamics
	
	\begin{equation*}
	{(C_{dc}+K_d) \dot{\tilde\omega}_v = -G_{dc} \omega_v -K_{p,dc}\tilde {\omega}_{v}-K_{i,dc} \int \tilde {\omega}_{v} -\eta i_x}
	\,.
	\end{equation*}
	
	Finally, we divide by $\eta^2$ to get
	
	\begin{equation}
	{\frac{(C_{dc}+K_d)}{\eta^2} \dot{\tilde\omega}_v = -\frac{G_{dc}}{\eta^2} \omega_v -\frac{K_{p,dc}}{\eta^2}\tilde {\omega}_{v}-\frac{K_{i,dc}}{\eta^2} \int \tilde {\omega}_{v} -\tau_{e,v}}
	\,.
	\end{equation}

If the measurements of the DC capacitor voltage are available, the presence of D-controller contribute to the increase of {\em synthetic} inertia of the DC/AC converter. This is advantageous since in traditional power systems, inertia and rotating masses provide/absorb energy in case of frequency deviation contributing to system damping through their rotational inertia \cite{BKP-SB-FD:15}. In the network case, where multiple inverters are operating and connected to the grid, optimal virtual inertia placement problem has been already addressed in \cite{BKP-SB-FD:15}.

\subsubsection{About the natural P-controller in DC circuit}
	We take $K_{p,dc}=0$. 
	There is a natural P-part for the PID controller, which consists in the resistance $G_{dc}$ of the DC circuit. According to our approach, based on critical damping for placement of the closed loop system eigenvalues $\lambda_0<0$, it holds
	\[
	\lambda_0=-\frac{G_{dc}}{2(C_{dc}+K_{d,dc})}
	\,. 
	\]
	

\subsubsection{Simulation results}

For the PID controller of $i_{dc}$, we use the following controller parameters values
\begin{equation*} 
\lambda_0=-100s^{-1},\,K_{p,dc}=0.3H/s,\,K_{i,dc}=20H/s,\, K_{d,dc}=0.001H/s
\,,
\end{equation*}

tracking the desired frequency $f_{ref}=50$ Hz. For the simulation purpose we implement the following PID with the filter represented in the Laplace-domain with $\mathcal{L}\{e(t)\}=E(s)$

\begin{equation*}
U(s)=K_{p,dc} E(s)+ K_{i,dc} \frac{1}{s}E(s) + K_{d,dc} \underbrace{\frac{Ns}{s+N}E(s)}_{Z(s)}
\,,
\end{equation*}
where D-part of the PID controller goes through a low pass (LP) filter with the cut-off frequency $N>0$. The dynamics of the filtered error $\mathcal{L}^{-1}\{Z(s)\}=z(t)$ in time-domain can be described by

\begin{equation*}
\dot z=N(K_{d,dc}\dot e_{dc}-z)
\,.
\end{equation*}

We choose $N=10s^{-1}$ and initialize the DC voltage at $v_{dc}(0)=1000V$ as well as the DC current source $i_{dc}(0)=100A$ and plot the following the curves of the frequency tracking in closed loop fashion PID controller as shown in Figure \ref{fig: i_dc_ctrl_freq}.

\begin{figure}[h!]
\centering
\subfloat[Simulations with PID controller with $P=K_{p,dc}$]{\includegraphics[scale=0.3]{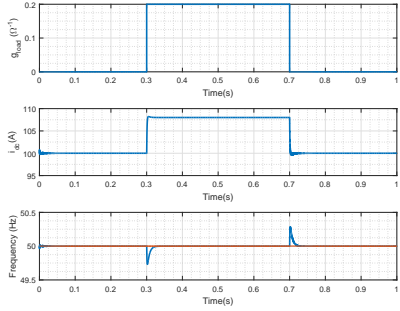}}
\caption{Time domain simulations of $i_{dc}$ controller tracking a given frequency $f_{ref}=\omega_{ref}/(2\pi)$ using the current source $i_{dc}$.}
\label{fig: i_dc_ctrl_freq}
\end{figure}

\subsection{Frequency tracking using the frequency gain}
\begin{assumption}[Non-zero DC voltage]
We assume that $v_{dc}(0)$ is non zero and take the value of the nominal DC voltage $v_{dc}(0)=v_{dc,ref}=i_{dc}/G_{dc}$
\end{assumption}

Given a desired frequency $\omega_{ref}=\eta_s v_{dc, s}$, we intend to design a controller able to track this given frequency using the gain $\eta$, considered so far as a real positive constant.
\begin{proposition}[Frequency control via the gain $\eta$]
 In order to track  the desired frequency $\omega_{ref}$, we propose the following controller
\begin{equation}
\dot\eta= \frac{\tau}{v_{dc}} (\omega_{ref}-\omega)-\frac{\eta}{v_{dc}}\dot v_{dc}
\,,
\label{eq: eta-law}
\end{equation}
with $\tau>0$.
The closed-loop system  converges asymptotically to the desired set of equilibria. 
\end{proposition}

\begin{proof}
We define the dynamics of the frequency $\omega=\eta v_{dc}$ as follows by using the control law \eqref{eq: eta-law}

\begin{subequations}
\begin{align*}
\dot{\omega} &=\dot\eta v_{dc}+\eta\dot v_{dc}
\\
\dot{\omega} &= {\tau} (\omega_{ref}-\omega)
\,.
\end{align*}
\end{subequations}

After performing a transformation into $dq0$- frame using the matrix $T_{dq0}$ as defined in \eqref{eq: trafo to dq} with angle $\gamma=\theta_v,\,\dot{\gamma}\dot\theta_v=\eta v_{dc}$. 
Let us consider the DC/AC converter as described previously incremented by the state $\omega\in\real$ as follows

\begin{subequations}
\begin{align}
\dot\omega & = -{\tau} (\omega-\omega_{ref})
\\
C_{dc} \dot v_{dc} &=-G_{dc} v_{dc} +i_{dc} - \frac{\mu}{2} \begin{bmatrix}
0\\1 \end{bmatrix}^{\top} i_{dq}
\\
L \dot {i}_{dq} &=-(L\eta v_{dc} J_2 +R)\, i_{dq} + \frac{\mu}{2} \begin{bmatrix}
0 \\ 1 
\end{bmatrix} v_{dc} - v_{dq}\\
C\, \dot v_{dq} &=-(C\eta v_{dc} J_2 + G_{load})\,v_{dq}+i_{dq}
\,.
\end{align}
\label{eq: closed-loop-eta}
\end{subequations}
with $\omega=\eta v_{dc}$.
At steady state, we have $\eta=\eta_s,\,v_{dc}=v_{dc,s}$ and  $\omega=\omega_{ref}=\eta_sv_{dc,s}$, where we perform a transformation using the matrix $T_{dq}$ with the angle $\gamma=\theta_v,\,\dot \theta_v=\eta_s v_{dc,s}$ at steady state 

\begin{subequations}
\begin{align*}
0 &= 0 \\
0 &=-G_{dc} v_{dc,ref} + i_{dc} - \frac{\mu}{2} \begin{bmatrix}
0\\1 \end{bmatrix}^{\top} i_{dq,s}
\\
0 &=-(L\omega_{ref} J_2 +R) i_{dq,s} + \frac{\mu}{2} \begin{bmatrix}
0 \\ 1 
\end{bmatrix} v_{dc,s} - v_{dq,s}
\\
0 &=-(C\omega_{ref} J_2 + G_{load})\,v_{dq,s}+i_{dq,s}
\,.
\end{align*}
\end{subequations}

We can write the incremented system in error coordinates as follows with $\tilde \omega=\omega-\omega_{ref}$

\begin{subequations}
\begin{align*}
\dot{\tilde \omega} & = - \tau \tilde\omega
\\
C_{dc} \dot {\tilde v}_{dc} &=-G_{dc} \tilde v_{dc} - \frac{\mu}{2} \begin{bmatrix}
0\\1 \end{bmatrix}^{\top} \tilde i_{dq}
\\
L \dot {\tilde i}_{dq} &=-(L\omega J_2 + RI_2) \tilde i_{dq} - L\, \tilde\omega J_2 i_{dq,s}+\frac{\mu}{2} \begin{bmatrix}
0 \\ 1 
\end{bmatrix} \tilde{v}_{dc} + \tilde{v}_{dq}\\
C\, \dot {\tilde v}_{dq} &=-(C\omega J_2 + G_{load})\,\tilde {v}_{dq}-C \, \tilde\omega J_2\, v_{dq,s} + \tilde i_{dq}
\,.
\end{align*}
\end{subequations}

We now define the incremented Lyapunov candidate $\tilde W:\real^{6}\to\real$ by
\begin{equation*}
\tilde W=\frac{1}{2}\tilde\omega^2 +\frac{1}{2}C_{dc}\tilde v_{dc}^{2}+\frac{1}{2}C^{}\tilde v_{dq}^{\top}\tilde v_{dq}+\frac{1}{2}L\tilde i_{dq}^{\top}\tilde i_{dq}
\,,
\end{equation*}

and calculate its derivative along the closed-loop system trajectories. 
\begin{subequations}
\begin{align*}
\dot{\tilde W}
&=\tilde\omega\dot{\tilde{\omega}}-G_{dc}\tilde v_{dc}^{2} - \tilde v_{dq}^{\top} G_{load}\,\tilde {v}_{dq}-C \,\tilde v_{dq}^{\top} \tilde \omega J_2\, v_{dq,s}-\tilde i_{dq}^{\top}RI_2 \tilde i_{dq} - L\,\tilde i_{dq}^{\top} \tilde\omega J_2 i_{dq,s}
\\
&=- \tau \tilde\omega^2 - G_{dc}\tilde v_{dc}^{2} - \tilde v_{dq}^{\top} G_{load}\,\tilde {v}_{dq}-C\,\tilde v_{dq}^{\top} \tilde\omega J_2\, v_{dq,s}-\tilde i_{dq}^{\top}RI_2 \tilde i_{dq} 	- L\,\tilde i_{dq}^{\top} \tilde\omega J_2 i_{dq,s}
\\
&=\begin{bmatrix}
\tilde \omega  \\\tilde i_{dq}  \\  \tilde v_{dq} \\ \tilde v_{dc} 
\end{bmatrix}^{\top}
\,
\begin{bmatrix}
- \tau  & -\frac{L}{2}(J_2i_{dq,s})^{\top} & -\frac{C}{2}(J_2v_{dq,s})^{\top} & 0
\\
-\frac{L}{2}J_2i_{dq,s} & -RI_2 & 0 & 0
\\
-\frac{C}{2}J_2v_{dq,s} & 0 &  -G_{load} & 0
\\
0 & 0 & 0 & -G_{dc}
\end{bmatrix}
\,
\begin{bmatrix}
\tilde \omega\\ \tilde i_{dq} \\ \tilde v_{dq} \\  \tilde v_{dc}
\end{bmatrix}\\
&= e^{\top} Z e < 0
\,,
\end{align*}
\end{subequations}
with $e=\begin{bmatrix}
\tilde\omega & \tilde i_{dq} & \tilde v_{dq} & \tilde v_{dc}
\end{bmatrix}^{\top}$.

The matrix $Z\in\real^{6\times 6}$ is negative definite under the following condition corresponding to the chosen Lyapunov function $\tilde W$: 

\begin{equation*}
{\frac{LRg^2}{4}\norm{i_{dq}}^2_2+\frac{R^2gC^2}{4}\norm{v_{dq}}_2^2\leq R^2g^2\tau}
\,.
\end{equation*}
\end{proof}

\subsubsection{High-level control and VSM}

By choosing the dynamics of $\eta$ control as follows	
\begin{equation}
{\dot{\eta}= \frac{-\tau}{v_{dc}} (\omega-\omega_{ref})+\eta \frac{G_{dc}}{C_{dc}}}
\,.
\label{eq: match-vsm}
\end{equation}

and applying \eqref{eq: match-vsm}, it holds for the dynamics of $\omega_v$

\begin{subequations}
\begin{align*}
\dot \omega_v &= \dot{\eta} v_{dc}+\eta\dot {v}_{dc}
\\
\dot {\tilde\omega}_v&= -\tau {\tilde \omega}_v+\eta \frac{G_{dc}}{C_{dc}} v_{dc}+\eta \dot v_{dc}
\\
&=-\tau {\tilde \omega}_v+\eta \frac{G_{dc}}{C_{dc}} v_{dc}+\frac{\eta}{C_{dc}} (-G_{dc}v_{dc}+i_{dc}-\frac{1}{2} m_{\alpha\beta}^{\top}i_{\alpha\beta})
\\
&=-\tau {\tilde \omega}_v+\frac{\eta^2}{C_{dc}} \left(\frac{i_{dc}}{\eta}-\frac{\mu}{2\eta} \begin{bmatrix}
-\sin(\theta_v) \\ \cos(\theta_v) \end{bmatrix}
i_{\alpha\beta}\right)
\\
&=-\tau {\tilde \omega}_v+\frac{\eta^2}{C_{dc}} (\tau_{m,v}-\tau_{e,v})
\,,
\end{align*}
\end{subequations}

where $\tilde{\omega}_v=\omega_v - \omega_{ref}$ and we have 

\begin{equation*}
{\frac{C_{dc}}{\eta^2}\dot {\tilde\omega}_v= -\frac{\tau C_{dc}}{\eta^2} {\tilde \omega}_v + \tau_{m,v}-\tau_{e,v}}
\,.
\end{equation*}

Using the controller in \eqref{eq: match-vsm}, we recover virtual synchronous machine (VSM) model control introduced earlier in this work.
In fact, we can emulate an SM equation, due to the introduction of the fictitious angle $\theta_v$, with a inertia $M=\frac{C_{dc}}{\eta^2}$ and the damping term $D={\tau C_{dc}}/{\eta^2}$. This contributes to even more enhancing the system stability via damping and frequency response after a disturbance. 

We consider the following controller

\begin{equation}
{\dot{\eta}= \frac{-\tau}{v_{dc}} (\omega-\omega_{ref})+ \eta\frac{G_{dc}}{C_{dc}}
	- \frac{J}{v_{dc}}\dot{\omega}_v}
\,.
\label{eq: mat-vsm}
\end{equation}
In closed-loop fashion, we have the following

\begin{equation*}
{\frac{C_{dc}}{\eta^2}(1+J)\dot {\tilde\omega}_v= -\frac{\tau C_{dc}}{\eta^2} {\tilde \omega}_v + \tau_{m,v}-\tau_{e,v}}
\,.
\end{equation*}
The equivalent SM model has damping factor $D={\tau C_{dc}}/{\eta^2}$ and an inertia $M={C_{dc}(1+J)}/{\eta^2}$.

Our controller design can be regarded as equivalent to that of a virtual synchronous machine induced by the matching control extended via outer-loop control by particular choice of $\eta$ dynamics as described in \eqref{eq: mat-vsm}

\begin{remark}[Closed-loop simulations]
We simulate the DC/AC converter with the following parameters for the above suggested controller using the gain $\tau=100s^{-1}$ for tracking a desired frequency of $f_{ref}=50 Hz$ as shown in Figure \ref{fig: eta_ctr_freq}, where the DC capacitor voltage is initialized with the nominal voltage $v_{dc}(0)=v_{ref}=1000V$ corresponding to nominal frequency and we set the initial condition $\eta(t=0)=\omega_{ref}/v_{ref}=0.3142\,rad/Vs$. The simulation results are shown in Figure \ref{fig: eta_ctr_freq}
\begin{figure}[h!]
\centering
\includegraphics[scale=0.3]{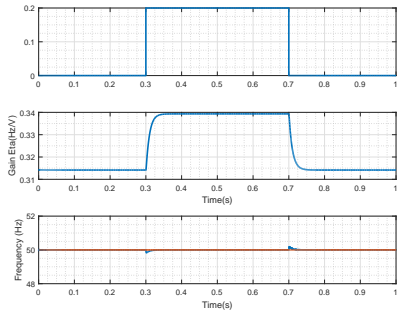}
\caption{Time domain simulations of tracking of a desired frequency $f_{ref}$ using the gain $\eta$ described by the control law \eqref{eq: eta-law}.}
\label{fig: eta_ctr_freq}
\end{figure}

We now increase the damping in the system and simulate with $\tau=2000s^{-1}$ and get the enhanced frequency response as shown in Figure \ref{fig: eta_ctr_freq_kappa}.

\begin{figure}[h!]
\centering
\includegraphics[scale=0.3]{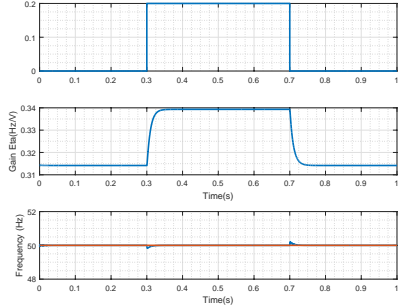}
\caption{Time domain simulations of tracking of a desired frequency $f_{ref}$ using the gain $\eta$ described by the control law, after increasing the damping $\tau>0$.}
\label{fig: eta_ctr_freq_kappa}
\end{figure}
\end{remark}

{\em As a summary} there are many possibilities offered by high-level control extensions in designing $i_{dc},\,\eta,\,\mu$ as shown in Figure \ref{fig: highlevel} aiming to track a given reference in amplitude or frequency. Our design ranges from choosing a proper PID control for the DC current source which implies an extra inertia and damping for the system in the closed-loop fashion to applying nonlinear tools such as feedback linearization for designing $\eta$ and $\mu$ in order to achieve different tracking objectives.

\begin{figure}[h!]
\centering
\resizebox{.8\textwidth}{3cm}{
	\begin{circuitikz}[american voltages]
	\draw
	(0,0) to [short, *-] (3,0)
	(5,3) to [american current source, i>=$i_{load}$] (5,0) 
	(4,3) to [short, *-] (5,3)
	(4,0) to [short, *-] (5,0)
	(3,0) to (4,0)
	(3,3) to (4,3)
	(0.5,0) to [open, v^>=${v}_x$] (0.5,3) 
	(0,3) 
	to [short,*-, i=$\color{red}i_{\alpha\beta}$] (1,3) 
	to [R, l=$R$] (2,3) 
	to [L, l=$L$] (3.3,3) 
	to [short,*-, i_=$ $] (3.3,2) 
	to [C, l=$C$] (3.3,1) 
	to [short] (3.3,0)
	
	(2.2,3) to [open, v^<=$\color{red} v_{\alpha\beta}$] (2.2,0); 
	
	
	\draw[red, dashed, very thick] 
	(-2,-0.5) to (-2,3.5)
	(-0.5,1) to [Tpnp,n=pnp] (-0.5,2)
	(-2,3.5) to (0,3.5)
	(0,3.5) to (0,-0.5)
	(-2,-0.5) to (0,-0.5);
	\node[draw=none] at (-1,4) {$\color{red}\boxed{\eta,\mu}$};
	
	\draw
	(-2,0) to [short, *-] (-4,0)
	(-4.5,0) to [short, *-] (-4,0);
	
	\draw
	(-10,0) to (-4.5,0) 
	(-10,3) to (-4.5,3);
	\draw[red, dashed, very thick] 
	(-10,0) to [american current source, i>=$\boxed{\color{red}i_{dc}}$] (-10,3); 
	\draw
	(-8,3) to [R,l=$G_{dc}\color{red}+K_{p,ic}$] (-8,0) 
	(-4.1,3) to [C, l=$C_{dc}\color{red}+K_{d,ic}$] (-4.1,0) 
	(-2,3) to [short,*-, i<=$i_{x}$] (-3,3)
	(-4.5,3) to [short, *-] (-3,3);
	\draw[red, very thick] 
	(-5.8,3) to [L, l=$K_{i,dc}$] (-5.8,0); 
	\draw
	(-9.5,3) to [open, v^<=${v}_{dc}$] (-9.5,0);
	\end{circuitikz}
}\caption{Summary of the possibilities offered by high-level control}
\label{fig: highlevel}
\end{figure}
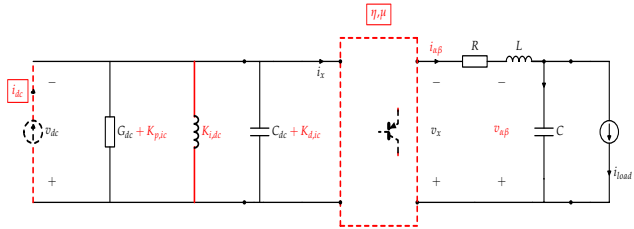

\section{Shaping the reactive power response profile}

Due to the absence of a droop behavior for the reactive power to the amplitude or/and frequency at the output of the modulation block, we aim in this section to design controllers leading to the shaping of reactive power response by inducing a trade-off between the reactive power $Q_x$ and the amplitude $\hat {v}_x>0$ and angular velocity $\omega_x$ of the voltage at the output of the modulation block $v_x$. This goal can be achieved through a proper choice of the degrees of freedom of our matching controller  $\mu,\,\eta,\,i_{dc}$.

In order to simplify the analysis, we consider the DC/AC converter at the output of the modulation block, i.e without the RLC filter.

\begin{proposition}[Droop behavior via feedback of reactive power]
We assume that there is no active power at the output of the modulation block, i.e $P_x=0$.
We can induce a trade-off between the reactive power $Q_x$ and 
\begin{itemize}
\item [$(1)$]the amplitude $\hat {v}_x>0$ of the AC voltage at the  output of the modulation block $v_x$, when we design a controller acting on the modulation amplitude according to a resistive droop control 
\begin{equation}
\mu=\mu_0+k_{\mu}\,Q_x,\,k_{\mu}>0,\, \mu_0={2\hat v}/{v_{dc,ref}}
\,.
\label{eq: case1}
\end{equation}

The range of reactive power that can be delivered for a capacitive/inductive load is	 
\begin{equation*}
-\frac{\mu_0}{k_\mu}<Q_{x}<\frac{\mu_0}{k_\mu}
\,.
\end{equation*}

\item [$(2)$]the angular velocity $\omega_x$ by acting on the frequency gain $\eta$ according to the inductive droop control gain
\begin{equation}
\eta=\eta_0 + k_{\eta}\, Q_x,\,k_{\eta}>0
\,,
\label{eq: case2}
\end{equation} 
where $\mu_0=\omega_0/{v_{dc,0}}$.

\item [$(3)$] both the amplitude and the frequency of $v_x$, as follows
\begin{equation}
i_{dc}=i_{dc,0}+k_i Q_x
\,,
\label{eq: case3}
\end{equation}
and this is according to the governor control law, where $i_{dc,0}=G_{dc}{v_{dc,ref}}$. 


The range of reactive power that can be delivered for an inductive/capacitive load, in the absence of active power, at the at the output of the modulation block is

\[
-\frac{i_{dc,0}}{k_i}<Q_{x}<\frac{i_{dc,0}}{k_i}
\,.
\]
\end{itemize}
\end{proposition}

\begin{proof}
We consider the reactive power at the same node, which can be expressed as follows: 
\begin{subequations}
\begin{align*}
Q_x &= v_x^{\top}Ji_{\alpha\beta}= v_x^{\top}J\,i_{load}
\,.
\end{align*}
\end{subequations}
We take $i_{load}=b_{load}Jv_x$, where $J$ is the rotation matrix of ${\pi}/{2}$. We get the following expression of the reactive power $Q_x$
\begin{equation*}
Q_x=-b_{load}\,\hat {v}_x^2
\,,
\end{equation*}
where $\hat{v}_x^2=v_x^{\top}v_x$.
We consider the analytical solutions relating the active power $P_x$ to the amplitude and frequency of $v_x$  introduced in \eqref{eq: amplitude and frequency}.
\begin{itemize}
\item[(1)] By applying the controller in \eqref{eq: case1}, it yields for the voltage $v_x$
\begin{subequations}
	\begin{align*}
	\hat {v}_x &=\frac{(\mu_0 + k_{\mu}\,Q_x)}{4G_{dc}}(i_{dc}+\sqrt{i_{dc}^{2}-4G_{dc}P_x})
	\\
	\omega_x &=\frac{\eta}{2G_{dc}}(i_{dc}+\sqrt{i_{dc}^{2}-4G_{dc}P_x})
	\end{align*}
\end{subequations}
For calculation of the maximal and minimal value of the reactive power, first consider inductive load, i.e $b_{load}<0$ and set $P_x=0$ in \eqref{eq: case 1}.

\begin{subequations}
\begin{align*}
\hat {v}_x &=\frac{(\mu_0+k_{\mu}\,Q_x)i_{dc}}{2G_{dc}}\\
&=\frac{(\mu_0-k_{\mu}\,b_{load}\,r_x^2)i_{dc}}{2G_{dc}}
\,,
\end{align*}
\end{subequations}
and get the following quadratic function
\begin{equation*}
k_\mu\, b_{load}\, \hat {v}_x^2+2 \hat {v}_x G_{dc}-\mu_0=0
\,.
\end{equation*}
This equation has a solution $\hat {v}_x$, as long as for some $b_{max}<b_{load}<0$ , where:
\begin{subequations}
\begin{align*}
b_{max}&=-\frac{1}{v_{dc,0}^{2}k_\mu\mu_0}
\\
\bar{v}_{x}&= \mu_0 v_{dc,0}
\\
Q_{x,max}&=\frac{\mu_0}{k_\mu}
\,.
\end{align*}
\end{subequations}
For a purely capacitive load, the condition $\hat {v}_x>0$ holds which implies that: 
\begin{equation*}
\mu_0+k_\mu Q_{x}>0
\,,
\end{equation*}
and we derive the following limit of capacitive power:
\begin{equation*}
Q_{x,min}=-\frac{\mu_0}{k_\mu}
\,.
\end{equation*}

\item[(2)]By using \eqref{eq: case2}, we write the amplitude and frequency of the voltage $v_x$ as
\begin{subequations}
	\begin{align*}
	\hat {v}_x &=\frac{\mu}{4G_{dc}}(i_{dc}+\sqrt{i_{dc}^{2}-4G_{dc}P_x})
	\\
	\omega_x &=\frac{\eta_0+k_{\eta}\, Q_x}{2G_{dc}}(i_{dc}+\sqrt{i_{dc}^{2}-4G_{dc}P_x})
	\end{align*}
\end{subequations}

\item[(3)]By applying the controller in \eqref{eq: case3}, it yields for the voltage $v_x$
\begin{subequations}
	\begin{align*}
	\hat {v}_x &=\frac{\mu}{4G_{dc}}(i_{dc,0}+k_i Q_x+\sqrt{(i_{dc,0}+k_i Q_x)^{2}-4G_{dc}P_x})
	\\
	\omega_x &=\frac{\eta}{2G_{dc}}(i_{dc,0}+k_i Q_x+\sqrt{(i_{dc,0}+k_i Q_x)^{2}-4G_{dc}P_x})
	\end{align*}
\end{subequations}
For calculation of the maximal and minimal value of the reactive power, we consider the expression of $i_{dc}$, where we set $P_x=0$:
\begin{equation*}
\hat {v}_x =\frac{\mu\,(i_{dc,0}+k_{i}\,Q_x)}{2G_{dc}}
\,.
\label{eq: Amplitude case 3}
\end{equation*} 
We first consider inductive load with $b_{load}<0$ and we get the following equation where we substitute $Q_x=-b_{load}\,\hat {v}_x^2$.
\begin{equation*}
k_i\,\mu\, b_{load}\, \hat {v}_x^2+2G_{dc}\hat {v}_x-\mu\, i_{dc,0}=0
\,.
\end{equation*}
This equation has solution $\hat {v}_x>0$, for $b_{max}< b_{load}<0$, where: 
\begin{subequations}
\begin{align*}
b_{max} &=-\frac{G_{dc}^2}{\mu^2i_{dc,0}k_i}\\
\bar {v}_{x} &=\mu v_{dc,0}\\
Q_{max} &=\frac{i_{dc,0}}{k_i}
\,.
\end{align*}
\end{subequations} 
For a purely capacitive load, where $b_{load}>0$ and $(P_x=0)$, it holds that $\hat {v}_x>0$ where we can derive from \eqref{eq: Amplitude case 3} that:
\begin{equation*}
Q_{min}=-\frac{i_{dc,0}}{k_i}
\,.
\end{equation*}
\end{itemize}
\end{proof}


\begin{remark} [Choice of the gains $k_\mu,\,k_i$]
The choice of the gains $k_\mu$ and $k_i$ determines the maximal reactive power the three-phase inverter can deliver, which should be made with respect to the maximal active and reactive power of the DC/AC converter. 
\end{remark}

\subsubsection{Influence of the power flow on the DC circuit}

The DC capacitor voltage is a solution of the equation \eqref{eq: steady state vdc} and at steady state,i.e. $\dot v_{dc}=0$: 
\begin{equation*}
v_{dc}=\frac{i_{dc}+\sqrt{i_{dc}^{2}-4G_{dc}P_x}}{2 G_{dc}}
\,.
\end{equation*}
If we choose a constant value of the current source $i_{dc}$, DC circuit is only influenced by active and not reactive power.
Due to the choice of $i_{dc}=i_{dc,0}+k_i Q_x$ according to \eqref{eq: case3}, the following equation holds:
\begin{equation*}
v_{dc}=\frac{i_{dc,0}+k_i Q_x+\sqrt{(i_{dc,0}+k_i Q_x)^{2}-4G_{dc}P_x}}{2 G_{dc}}
\,,
\end{equation*}
which reveals that a relationship between the reactive power and the DC voltage is induced as a consequency of our control design in \eqref{eq: case3} .

\subsubsection{Simulation results}
\begin{figure}[h!]
\subfloat[Controller (1)]{\includegraphics[scale=0.3]{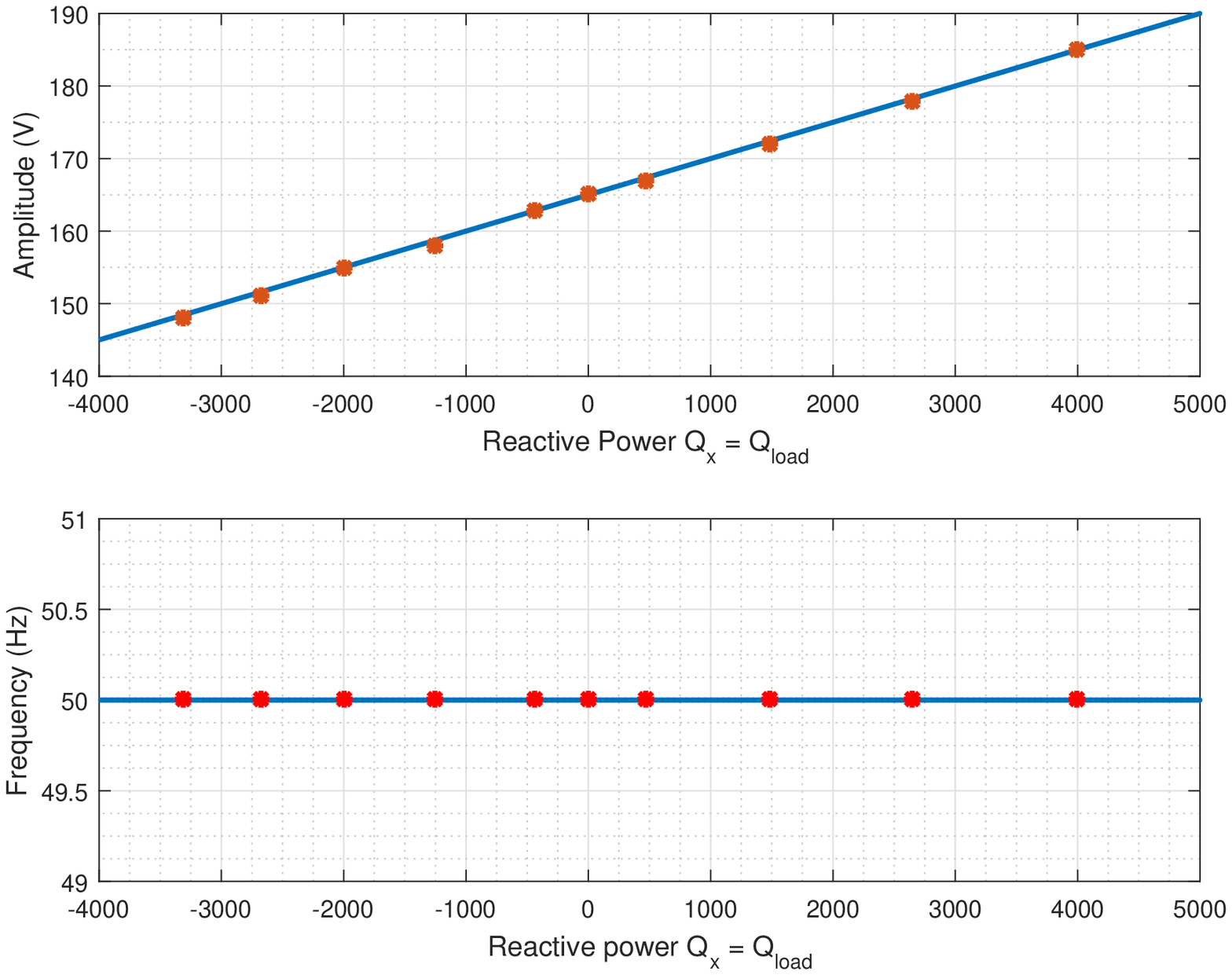}}
\subfloat[Controller (2)]{\includegraphics[scale=0.3]{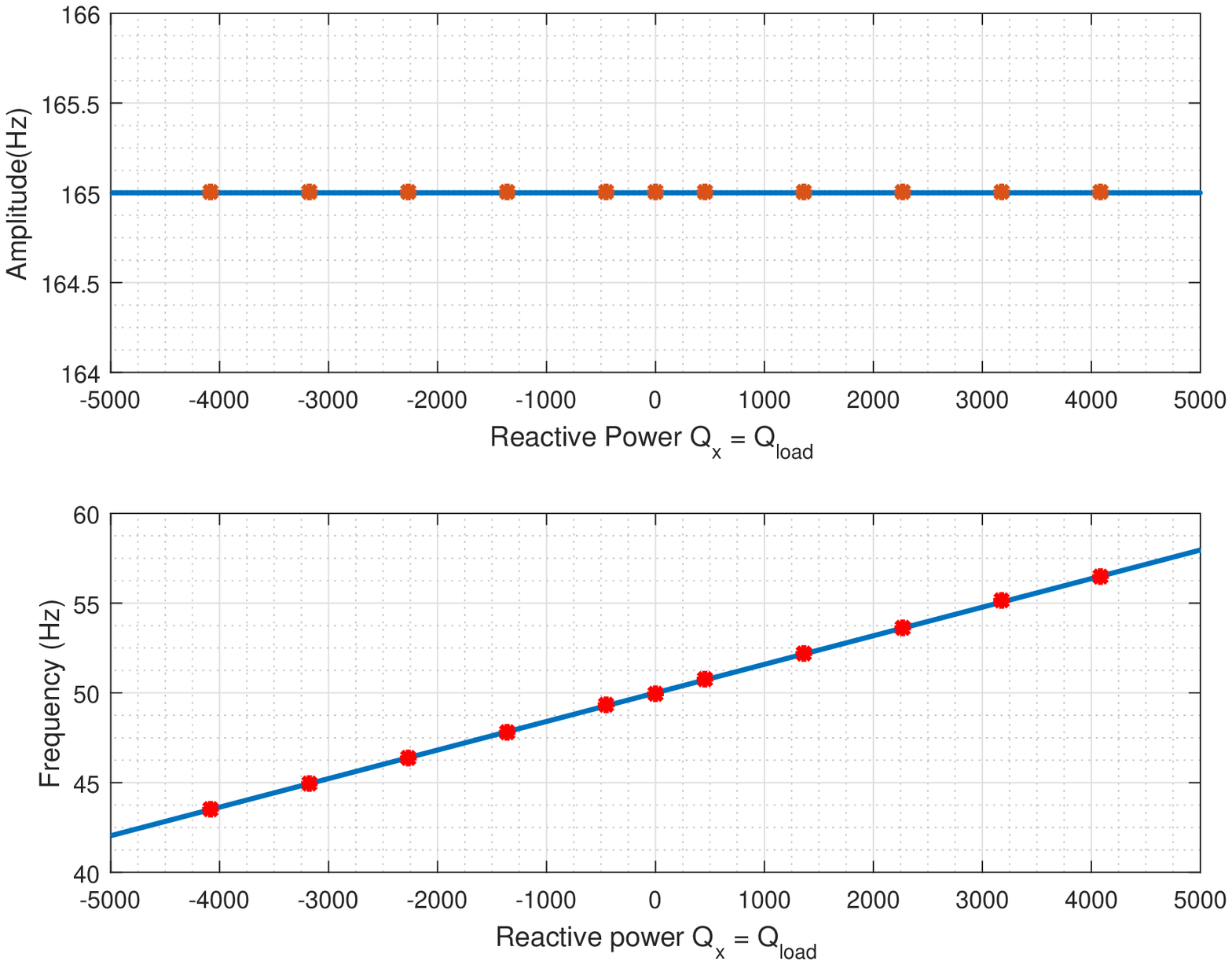}}\\
\centering
\subfloat[Controller (3)]{\includegraphics[scale=0.3]{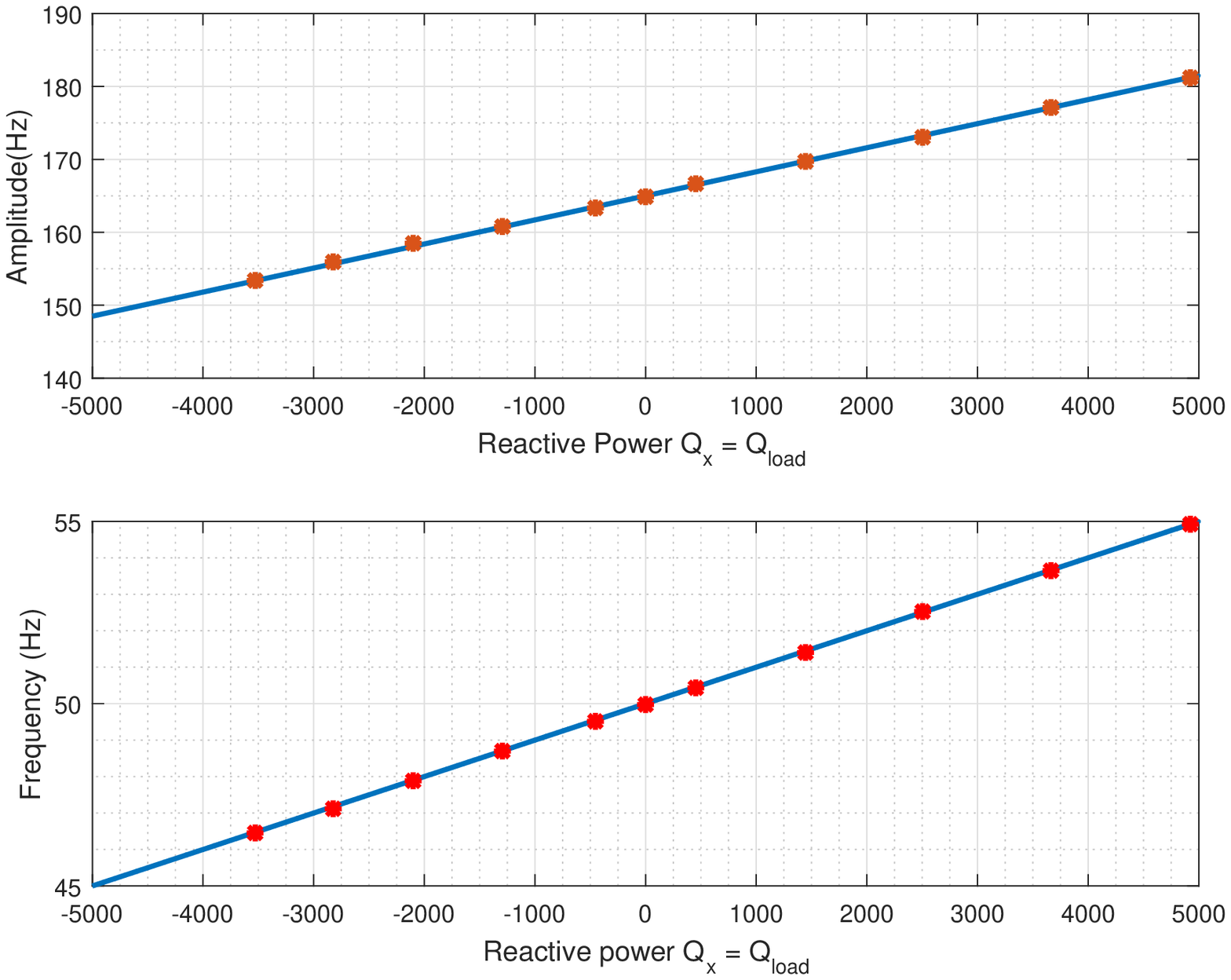}}
\caption{Droop behavior describing a trade-off between the reactive power $Q_x$ and the amplitude and the frequency of the voltage at the output of the modulation block}
\label{fig: Reactive power cases}	
\end{figure}
We simulate with nominal values in
$Q_0=5000VAR,\,f_{ref}=50Hz,\,\hat v= 165V,\, i_{dc,0}=100A,v_{dc,0}=1000V$.

\begin{itemize}
\item[(1)] We simulate with the following parameter values of the controller  
\begin{equation*}
i_{dc}=100A,\,\mu_0=\frac{2 \hat v}{v_{dc,0}}=0.33,\,k_{\mu}=10^{-5}VAR^{-1},\,\eta=0.314\, rad/Vs
\,.
\end{equation*}
Taking $P_x=0$, the equation \eqref{eq: case 1} simplifies to: 
\begin{equation*}
\hat {v}_x =\frac{(\mu_0+k_{\mu}\,Q_x)i_{dc}}{2G_{dc}}
\,,
\end{equation*}
where the amplitude $\hat {v}_x$ is linearly dependent on the reactive power $Q_x$ as depicted in Figure \ref{fig: Reactive power cases}.
\item[(2)] We take the following values to implement the controller

\begin{equation*}
i_{dc}=100A,\,\mu=0.33,\,k_{\eta}=10^{-5}\,rad/VARVs,\eta_0=0.314\, rad/Vs
\,.
\end{equation*}
The equation \eqref{eq: case 2} simplifies to:
\begin{equation*}
\omega_x =\frac{(\eta_0+k_{\eta}\, Q_x)i_{dc}}{G_{dc}}
\,.
\end{equation*} 
Thus, the frequency $\omega_x$ is linearly dependent on the reactive power $Q_x$ matching the simulation results in Figure \ref{fig: Reactive power cases}.
\item[(3)] Finally, we choose the following parameters to implement the controller 
\begin{equation*}
i_{dc,0}=100A,\,\mu=0.33,\,k_{i}=2\cdot 10^{-3}A/VAR,\,\eta=0.314\, rad/Vs
\,.
\end{equation*}
and we simplify \eqref{eq: case 3} as follows
\begin{subequations}
\begin{align*}
\hat {v}_x &=\frac{\mu\,(i_{dc,0}+k_{i}\,Q_x)}{2G_{dc}}
\\		
\omega_x &=\frac{\eta\,(i_{dc,0}+k_{i}\,Q_x)}{G_{dc}}
\,.
\end{align*}
\end{subequations}
The relationship of the active power to the amplitude $r_x$ and the frequency $\omega_x$ is shown in Figure \ref{fig: Reactive power cases}.
\end{itemize}

For the controllers $(1)$ and $(3)$, we choose a maximal value of $Q_{max}=10000VAR$, where we choose the corresponding gain to be
\[ 
k_\mu=\frac{\mu_0}{Q_{max}}=3.3 \cdot 10^{-5}VAR^{-1}
\,.
\]\

We choose a maximal value of $Q_{max}=10000VAR$ , where we choose the corresponding gain to be
\[ 
k_i=\frac{i_{dc,0}}{Q_{max}}=10^{-2}A/VAR
\,.
\]\
Simulation results are plotted in Figure \ref{fig: Maximal reactive power} and compared to the analytical values of $Q_{max}$.

\begin{figure}
\subfloat[Controller (1)]{\includegraphics[scale=0.3]{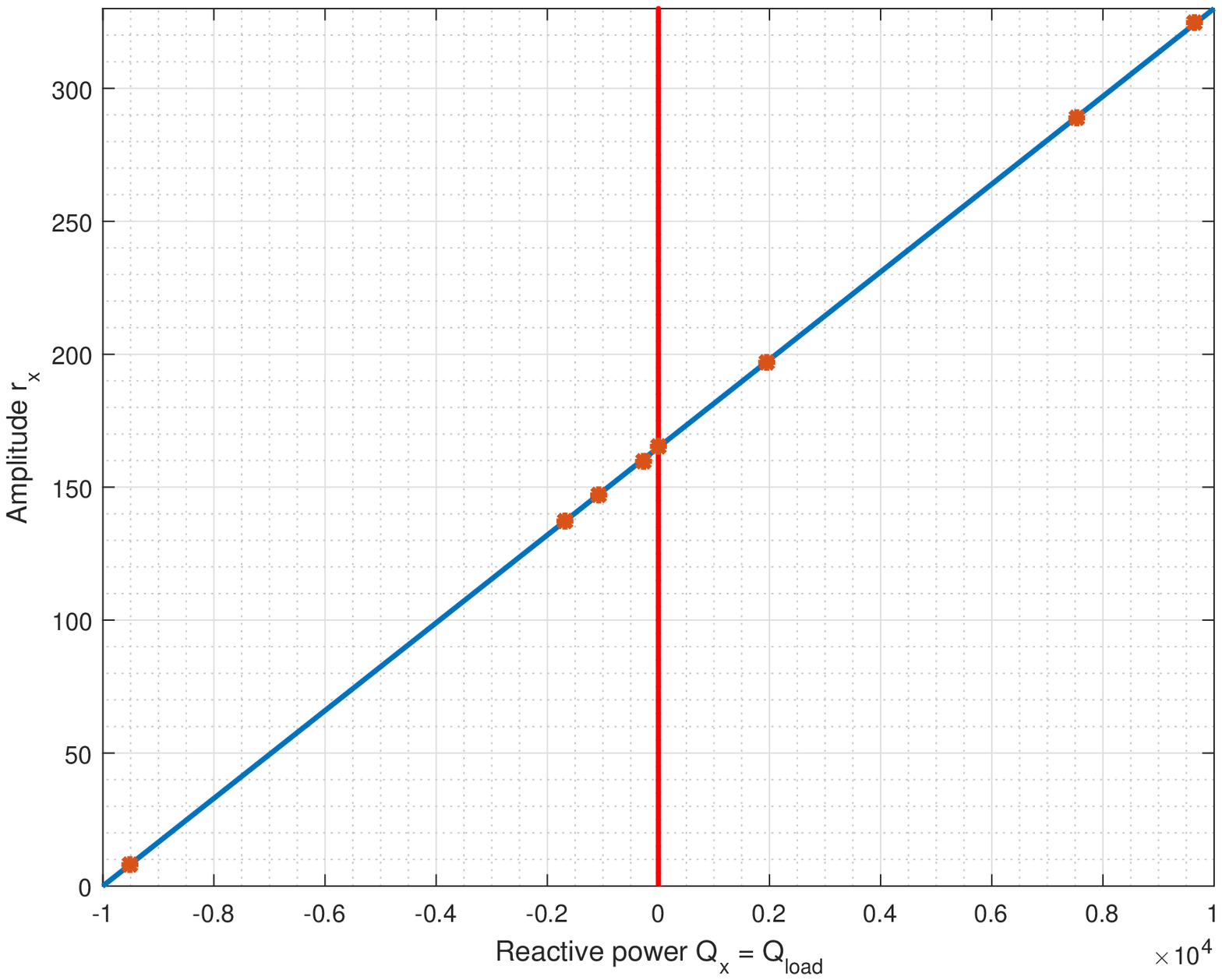}}
\subfloat[Controller (3)]{\includegraphics[scale=0.3]{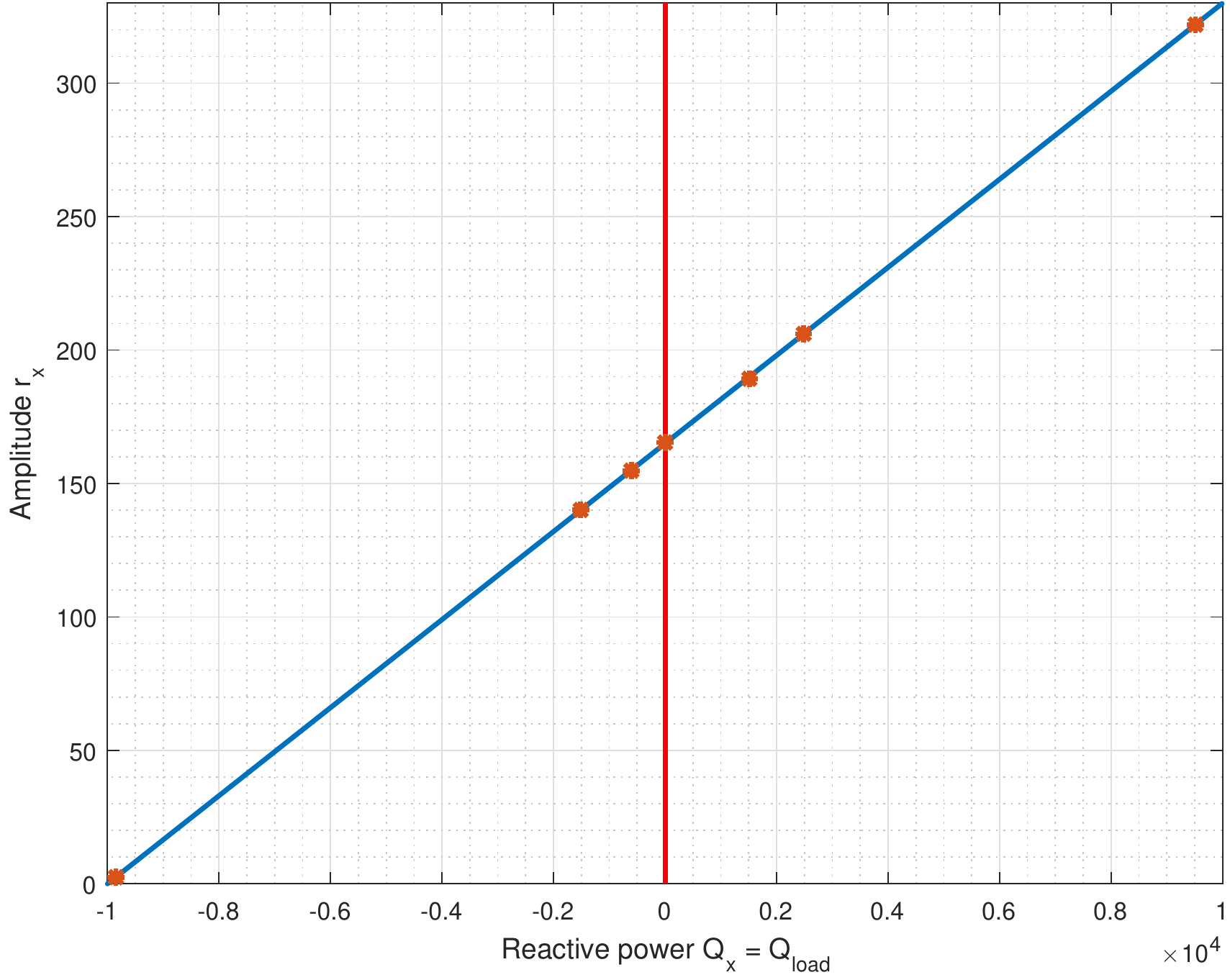}}	
\caption{Characteristic curves of reactive power $Q_x$ vs amplitude $r_x$}
\label{fig: Maximal reactive power}	
\end{figure}

Depending on the desired trade-off between the reactive power to the amplitude $\hat {v}_x$ or/and the frequency $\omega_x$ and the DC current source $i_{dc}$, we can shape the reactive power response at the output of the modulation block.

{\em As a summary}, we first designed $i_{dc},\, \mu,\, \eta$  as parameters representing degrees of freedom in order to track given references in amplitude and frequency. PID design of the current source enhances system performance in terms of increasing damping and inertia, which help stabilize the system response after a disturbance. This is motivated by the passivity analysis conducted earlier in this work, where the DC current source is considered to be a passive input.  
Stability proofs for the different outer-loop controls have been investigated after transformation into the rotating $dq0$.
We then shaped the reactive power response profile and induced a droop behavior relating reactive power to the amplitude and frequency of the voltage at the output of the modulation block $v_x$ by feedback design of reactive power $Q_x$ or feedback of load current $i_{load}$, where we can design the current source $i_{dc}$ as well as the matching control gains $\eta$ and $\mu$ to achieve this control objective. We identified the limits on reactive power in case they exist. This approach represents an alternative way of how one might exploit these degrees of freedom offered by our control design in order to ensure droop behavior, as a control objective towards a generalization of the network case.

\chapter{Case study: Simulation of the two network topologies}

\label{sec: network-case}
We aim to simulate decentralized matching control in the multiple converters case, where we properly choose the modulation signal as input for each of the DC/AC converters according to the matching control law introduced previously in Section \ref{sec: matching control}. 
We introduce two different network topologies composed of two identical DC/AC converters connected to loads of constant impedance.
\begin{definition}[Graph theory]
A network is considered to be a directed graph $G(V,E)$, where V is a set of the so-called vertices or nodes represented by the DC/AC converters and the load connected to the ground. E is a set of ordered pair of vertices called edges.
\end{definition} 
In fact, node dynamics are determined by the shunt capacitor to the ground at each converter terminal voltage for converters. Load nodes dynamics differ from one topology to the other and are determined in the following by Kirchhoff's laws.
Moreover, the dynamics for edges are represented by an inductance $L_{net}>0$ set in series with a resistance $R_{net}>0$ at each phase in $(\alpha\beta)$- frame. 

\section{Tree Topology}
Consider the following network of DC/AC converters connected to a grid, where two identical DC/AC converters are interconnected via an edge. Each of the converters is set in parallel with a load of constance impedance $G_{load}\in\real^{2\times2}$ to the ground as depicted in Figure \ref{fig: Tree-topology}.

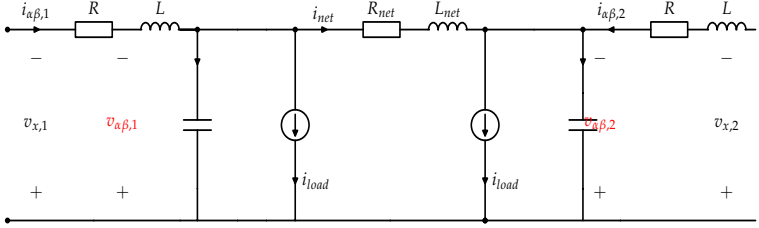
\begin{figure}[h!]
\centering
\resizebox{10cm}{3cm}{
	\begin{circuitikz}[american voltages]
	\draw
	(0,0) to [short, *-] (3,0)
	(5,3) to [american current source, i>=$i_{load}$] (5,0) 
	(4,3) to [short] (5,3)
	(4,0) to [short] (5,0)
	(3,0) to (4,0)
	(3,3) to (4,3)
	(0.5,0) to [open, v^>=${v}_{x,1}$] (0.5,3) 
	(0,3) 
	to [short,*-, i=$i_{\alpha\beta,1}$] (1,3) 
	to [R, l=$R$] (2,3) 
	to [L, l=$L$] (3.3,3) 
	to [short,*-, i_=$ $] (3.3,2) 
	to [C, l=$ $] (3.3,1) 
	to [short] (3.3,0)
	
	(2.0,3) to [open, v^<=$\color{red} v_{\alpha\beta,1}$] (2.0,0); 

	\draw
	(5,3) to [short,*-, i=$i_{net}$] (6,3) 
	to [R, l=$R_{net}$] (7,3) 
	to [L, l=$L_{net}$] (8.3,3) 
	to [short, *-] (10,3);
	
	\draw 
	(10,3) to [short,*-, i_=$ $] (10,2) 
	(10,2) to [C, l=$ $] (10,1) 
	to [short] (10,0);
	\draw
	(8.3,0)	to [short,*-] (13,0)
	(10.3,3) to [open, v^<=$\color{red} v_{\alpha\beta,2}$] (10.3,0); %
	
	\draw
	(8.3,3) to [american current source, i>=$i_{load}$] (8.3,0) %
	(8.3,0) to [short, *-] (5,0)
	(10,3) to [short, i<=$i_{\alpha\beta,2}$] (11,3) 
	to [R, l=$R$] (12,3) 
	to [L, l=$L$] (13,3) 
	(12.5,0) to [open, v^>=${v}_{x,2}$] (12.5,3);
	\end{circuitikz}}
\caption{Four-node network corresponding to two identical DC/AC converters connected via resistance $R_{net}>0$ and an inductance $L_{net}>0$ in series. Each of the converters is set in parallel with identical load of constant impedance to the ground.}
\label{fig: Tree-topology}	
\end{figure}

Based on the circuit diagram described above, we can write the following equations after applying Kirchhoff's laws for the inductance at the edge and the capacitance at each node.

\begin{subequations}
\begin{align}
L_{net}(\dot i_{net})+ R_{net} i_{net} &= v_{\alpha\beta,1}-v_{\alpha\beta,2}
\\
C\dot v_{\alpha\beta,k} &=i_{\alpha\beta,k}-i_{net}-i_{load,k}
\,,
\end{align}
\end{subequations}
where $k\in\{1,2\}$.

\begin{remark}[Simulation results]
We apply the following balanced load impedance $G_{load}\in\real^{2\times 2}$
\begin{subequations}
\begin{align*}
i_{load,k}= G_{load}\,v_{\alpha\beta,k},\,\, G_{load}=\begin{bmatrix}
-g & -b
\\ b & -g
\end{bmatrix}
\,,
\end{align*}
\end{subequations}
where $k\in\{1,2\}$, and we use the following edge parameters

\begin{equation}
R_{net}= 5\cdot R,\,
L_{net}= \frac{L}{10}
\,,
\end{equation}
where we initialize the DC capacitor with $v_{dc}(0)=1000V$.

\begin{figure}[h!]
\centering
\includegraphics[scale=0.3]{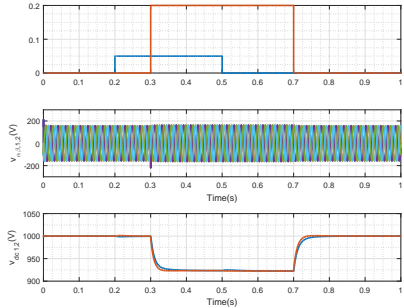}
\caption{Time-domain simulations of four-node tree topology after a step change in the load impedance.}
\label{fig: tree-topo}
\end{figure}	

The network of DC/AC converters in closed-loop fashion and interconnected to the grid according to a tree topology preserves droop behavior properties, identified in earlier sections as depicted in Figure \ref{fig: tree-topo}.
\end{remark}

\section{Star topology}
Next, we study a network of DC/AC converters connected to the grid via star topology. The two nodes represented by the inverters are connected to a common load by a reactive and resistive edges. This load has a constant conductance and is connected to the ground.

We use the following topology depicted in the diagram of Figure \ref{fig: star-topo} 

\begin{figure}[h!]
\centering
\resizebox{10cm}{3cm}{
	\begin{circuitikz}[american voltages]
	\draw
	(0,0) to [short, *-] (3,0)
	(6.7,3) to [american current source, i>=$i_{load}$] (6.7,0) 
	
	(4,0) to [short] (5,0)
	(3,0) to (4,0)
	(3,3) to (4,3)
	(0.5,0) to [open, v^>=${v}_{x,1}$] (0.5,3) 
	(0,3) 
	to [short,*-, i=$i_{\alpha\beta,1}$] (1,3) 
	to [R, l=$R$] (2,3) 
	to [L, l=$L$] (3.3,3) 
	to [short,*-, i_=$ $] (3.3,2) 
	to [C, l=$ $] (3.3,1) 
	to [short] (3.3,0)
	
	(2.0,3) to [open, v^<=$\color{red} v_{\alpha\beta,1}$] (2.0,0); 

	\draw
	(4,3) to [short, i<=$i_{net,1}$] (3.5,3) 
	(4,3)to [R, l=$R_{net}$] (5,3) 
	(5,3)to [L, l=$L_{net}$] (6,3) 
	to [short] (7,3);
	
	\draw
	(9.5,3) to [short, i=$i_{net,2}$] (9,3) 
	(7,3)to [R, l=$R_{net}$] (8,3) 
	to [L, l=$L_{net}$] (9,3) 
	to [short] (10,3);	
	\draw 
	(10,3) to [short,*-, i_=$ $] (10,2) 
	(10,2)to [C, l=$ $] (10,1) 
	to [short] (10,0);
	\draw
	(8.3,0)	to [short] (10,0)
	(10.3,3) to [open, v^<=$\color{red} v_{\alpha\beta,2}$] (10.3,0); %
	
	\draw
	(8.3,0) to [short] (5,0)
	(13,3) to [short,*-, i>=$i_{\alpha\beta,2}$] (12,3) 
	(10,3)to [R, l=$R$] (11,3) 
	
	(11,3)to [L, l=$L_{}$] (12,3) 
	(12.5,3) to [open, v^>=${v}_{x,2}$] (12.5,0)
	(13,0) to [short, *-] (10,0);
	\end{circuitikz}
}
\caption{Star Topology of two DC/AC converters and one load.}
\label{fig: star-topo}
\end{figure}
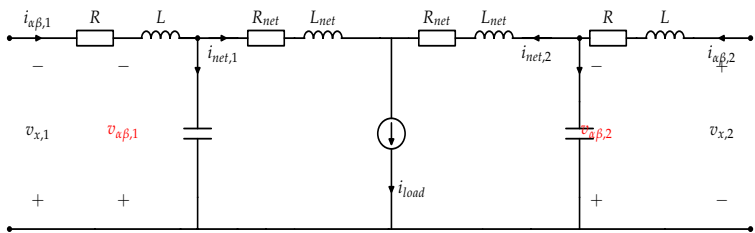

We can deduce the following equations relating network currents and voltages 

\begin{subequations}
\begin{align*}
L_{net} \dot{i}_{net,k}+R_{net}i_{net,k} &= v_{\alpha\beta,k}-v_{load}
\\
C \dot v_{\alpha\beta,k} &= i_{\alpha\beta,k}-i_{net,k}
\\
i_{load} &=i_{net,1}+i_{net,2}
\\
v_{load} &=R_{load}\, i_{load} 
\,,
\end{align*}
\end{subequations}
with $k\in\{1,2\}, \, R_{load}=\begin{bmatrix}
r & 0 \\ 0  & r
\end{bmatrix}$.

\begin{remark}[Simulation results]
We simulate the DC/AC converter with a balanced nonzero time varying load  undergoing a step change at $t=0.3s$ in the load resistance.
\begin{equation}
R_{net}= 5\cdot R,\,
L_{net}= \frac{L}{10}
\,,
\end{equation}
where we initialize the DC capacitor voltage with $v_{dc}(0)=1000V$.
\begin{figure}[h!]
\centering
\includegraphics[scale=0.3]{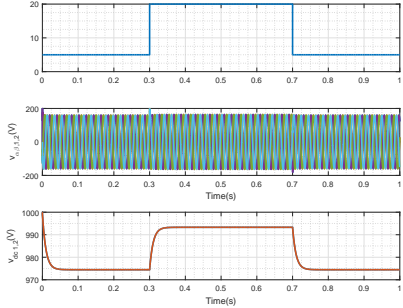}
\caption{Time-domain simulations of star network composed of two converters connected to a single load after a step change in the load resistance.}
\label{fig: star-topology}
\end{figure}

Droop properties are preserved in the case of a network in star topology as shown in Figure \ref{fig: star-topology}.

\end{remark}

{\em Summary} As a first observation of the simulations of two DC/AC converters connected to the load with different topologies (i.e  according to tree or star topology), droop behavior is preserved. Extensions might be related to frequency regulation, which can be achieved for example via outer-loops of DC current source, after a proper tuning of the PID control gains.
This and many another important aspects can be of further investigation in the network case.

\chapter{Conclusion and outlook}
\label{sec: conclusions}
In this work, we considered a detailed model of a DC/AC converter including DC cap dynamics usually neglected and represented solely by a voltage source in conventional power system literature. Followed by a review of the main control strategies in this area, we presented the traditional approach of controlling converters by usage of nested inner-loops. Outer-loops are constructed based on droop, virtual limit-cycle (VOC) as well as virtual synchronous machines (VSM) methods. We analyzed each of these different control schemes and validated it through simulation. Moreover, we highlighted its main challenges ranging from hard-to-justify assumptions, for instance a quasi-stationary steady state and operation on phasors but also large-time delays which may deteriorate system performance as well as the usage of non-tractable control architecture, considered as "blackbox" due to the hierarchical complexity of traditional control schemes. These issues may lead to unexpected system behavior preventing a stable primary network regulation.

Next, we proposed a novel converter control strategy that is motivated by the similarities between the rotor dynamics of a synchronous machine and the DC-link storage present in a converter. Our controller matches these two models, induces droop properties in amplitude and frequency which are key requirements towards a generalization for the network case. Different simulation results confirm the predicted behavior. Our control strategy involves adding only a single integrator and requires readily available DC-side measurements.
It preserves passivity characteristics in closed-loop fashion considered to be a stability and robustness certificate for our control approach.
We investigated stability of the DC/AC converter model based on the induced structural similarity to a synchronous machine and the convergence to the desired set of equilibria is guaranteed under sufficient condition related to a proper choice of the parameters of the DC/AC converter.

The matching control of synchronous machines can also be regarded as a an inner-loop controller. Based on it, we set the basis for further enhancement of the performance of our controller by extending its design in order to achieve tracking of a reference and saturation of currents and voltage.
Many cross links to the reviewed control approach arise while studying the matching control. Indeed, our controller can be interpreted as virtual control strategies for instance virtual oscillator control (VOC) and using outer-loops as a virtual synchronous machine (VSM).

Even though this novel control approach offers many degrees of freedom, it should be exploited by more systematic approach, especially when considering the networked viewpoint, which is an important aspect of further analysis of the proposed controller.
First simulations of DC/AC converter in a networked topology reveal that different properties are inherited from a single DC/AC converter. 
Prior goals in the network case are mainly related to frequency regulation and stable power sharing between multiple converters connected to the grid.

Another possible investigation is that of the adaption to single-phase settings, since the matching control is introduced and studied in three-phase.

\begin{appendix}

\end{appendix}



\chapter*{Eigenst\"{a}ndigkeitserkl\"{a}rung}
\selectlanguage{ngerman}
  Ich versichere hiermit, dass ich, \authorstring{}, die vorliegende
  Arbeit selbstst\"{a}ndig angefertigt, keine anderen als die
  an\-ge\-ge\-benen Hilfsmittel benutzt und sowohl w\"{o}rtliche, als auch
  sinngem\"{a}{\ss} entlehnte Stellen als solche kenntlich gemacht
  habe. Die Arbeit hat in gleicher oder \"{a}hnlicher Form noch keiner
  anderen Pr\"{u}fungsbeh\"{o}rde vorgelegen. Weiterhin best\"{a}tige ich, 
  dass das elektronische Exemplar mit den anderen Exemplaren \"{u}bereinstimmt.
  
\vspace{1cm}
\noindent{}\rule{0.39\textwidth}{0.4pt} \hspace{1.9cm} \rule{0.39\textwidth}{0.4pt}

\noindent{}Ort, Datum  \hspace{4.3cm} Unterschrift

%
\end{document}